\documentclass[11pt]{amsart}
\usepackage[english]{babel}
\usepackage[utf8]{inputenc}
\usepackage[
pdftex,hyperfootnotes]{hyperref}

   \usepackage{esint}

\usepackage{mathtools,arydshln}\mathtoolsset{centercolon}

     \newcommand{\PARENS}[1]{\left(#1\right)}
          \newcommand{\ccases}[1]{\begin{cases}#1\end{cases}}
        \usepackage{amsmath}
        \usepackage{mathrsfs}
      \usepackage{amssymb,amsthm,amscd}
       \usepackage{upgreek}

   \usepackage[scaled]{helvet} 
   \usepackage{courier} 
   \usepackage{eulervm} 
   \normalfont

\usepackage[toc,page]{appendix}
\usepackage{tocvsec2}

  \usepackage[T1]{fontenc}
  \usepackage[toc,page]{appendix}
  \usepackage[total={18cm,24.5cm},top=2cm, left=1.5cm]{geometry}

\renewcommand{\d}{\operatorname{d}}

\newcommand{\C}{\mathbb C}
\newcommand{\T}{\mathbb T}
\newcommand{\R}{\mathbb R}

\newtheorem{definition}{Definition}
\newtheorem{proposition}{Proposition}
\newtheorem{theorem}{Theorem}
\newtheorem{corollary}{Corollary}
\newtheorem{lemma}{Lemma}
\theoremstyle{definition}

\begin{document}
\author[G. A. Cassatella-Contra]{Giovanni A. Cassatella-Contra}
\address{Departamento de Física
  Teórica II (Métodos Matemáticos de la Física), Facultad de
  Físicas, Universidad Complutense de Madrid, Plaza de Ciencias nº 1,  Ciudad Universitaria, 28040 Madrid, Spain}
\email{gaccontra@fis.ucm.es}
\author[M. Mañas]{Manuel Mañas}
\address{Departamento de Física
	Teórica II (Métodos Matemáticos de la Física), Facultad de
	Físicas, Universidad Complutense de Madrid, Plaza de Ciencias nº 1,  Ciudad Universitaria, 28040 Madrid, Spain}
\email{manuel.manas@ucm.es}
\thanks{
	GCC benefited of the financial support of  \emph{Acción Especial  Ref. AE1/13-18837 }of the \emph{Universidad Complutense de Madrid}.
	The research of MM has been supported by  the grant MTM2015-65888-C4-3-P,  
	\emph{Ministerio de Economía y Competitividad}, Spain.
}
\thanks{The authors  acknowledge Prof. Piergiulio Tempesta for many illuminating conversations.}

\title[RH problem for  Szegő  matrix biorthogonal polynomials and the matrix dPII system]{ Matrix biorthogonal polynomials in the unit circle: \\Riemann--Hilbert problem  and\\the matrix discrete Painlevé II system}
	\keywords{Szegő matrix biorthogonal polynomials, quasi-determinants,  Cauchy transfoms, Riemann--Hilbert problem, recursion relations, generalized Pearson equations, monodromy free systems, Fuchsian and non-Fuchsian systems, matrix discrete Painlevé II sytems}
	\subjclass{14J70,15A23,33C45,37K10,37L60,42C05,46L55}

\begin{abstract}
	Matrix Szegő biorthogonal polynomials for quasi-definite matrices of measures are studied. For matrices of Hölder weights   a Riemann--Hilbert problem  is uniquely solved in terms of the matrix  Szegő polynomials and its Cauchy transforms. The Riemann--Hilbert problem is given as an appropriate  framework  for the discussion of the Szegő matrix and the associated Szegő recursion relations for the matrix orthogonal polynomials and its Cauchy transforms. Pearson type differential systems characterizing the matrix of weights are studied. These are linear systems of ordinary differential equations which are required to be monodromy free. Linear ordinary differential equations for the matrix Szegő polynomials and its Cauchy transforms are derived. It is shown how these Pearson systems  lead to nonlinear difference equations for the Verblunsky matrices and two examples, of Fuchsian and non-Fuchsian type, are considered.   For both cases a new matrix version of the discrete Painlevé II equation for the Verblunsky matrices is found.  Reductions of these matrix discrete Painlevé II systems presenting locality are discussed.
\end{abstract}
\maketitle

\tableofcontents

\section{Introduction}

The purpose of this paper is to explore the connection between the theory of
integrable discrete equations of Painlevé type and the Riemann--Hilbert
problem for matrix orthogonal polynomials.  In
the present work, we shall extend the Riemann--Hilbert approach to the unit
circle, for a class of biorthogonal polynomials of Szegő type defined in
terms of a matrix of Hölder weights.  With the aid of the Riemann--Hilbert problem we will find a matrix version of the discrete Painlevé II equations that holds for the Verblunsky matrices.

The unit circle is denoted by $\T:=\{z \in \C:
 |z|=1\}$, while its interior, the unit disk by $\mathbb{D}:=\{z \in \C : |z|<1\}$ and its exterior by  $\bar{\mathbb{D}}:=\{z \in \C : |z|<1\}$ . Given  complex Borel measure $\mu$ supported in $\T$ we say that is positive definite if it maps measurable sets into non-negative numbers, that in the  absolutely continuous situation, with respect to the Lebesgue measure  $\d m(\zeta)=\dfrac{\d \zeta}{2\pi\operatorname{i}\zeta}$, has the form $w({\zeta}) \d m(\zeta)$, $\zeta\in\T$, with the weight function $w(\zeta)$, integrable or Hölder, depending on the context. For the positive definite situation the orthogonal polynomials in the unit circle (OPUC) or Szegő polynomials are defined as those  monic polynomials  $P_n$ of degree  $n$ that satisfy  $\int_{\T}P_n(z) z^{-k} \d \mu(z)=0$, for $ k\in\{0,1,\dots,n-1\}$, \cite{szego}. We refer the reader to Barry Simon's books \cite{Simon-1} and \cite{Simon-2} for a very detailed studied of OPUC, and  for a survey on matrix orthogonal polynomials we refer the reader to \cite{Damanik}.

  Orthogonal polynomials on the real line (OPRL) supported in the interval  $[-1,1]$ and OPUC are deeply connected as has been shown in several papers, see \cite{Freud,Berriochoa}. From the side of an spectral approach a study   of the operator of multiplication by $z$ is required, this analysis is associated to the existence of recursion relations.   For OPRL the three term recurrence laws provide a tridiagonal matrix, the so called Jacobi operator, while  for OPUC one is faced with a Hessenberg matrix, being a more involved scenario that the Jacobi one (as it is not a sparse matrix with a finite number of non vanishing diagonals). In fact, A better approach are Szegő recursion relation, which need of the reciprocal or reverse Szegő polynomials $\tilde P_l(z):=z^l \overline{P_l(\bar z^{-1})}$ and the reflection or Verblunsky coefficients $\alpha_l:=P_l(0)$. Szegő recursion relations for
for OPUC  are
\begin{align*}
\PARENS{\begin{matrix} P_l (z)\\ \tilde P_{l} (z)\end{matrix}}=
\PARENS{\begin{matrix} z & \alpha_l \\ z \bar \alpha_l & 1 \end{matrix}}
\PARENS{\begin{matrix} P_{l-1}(z) \\ \tilde P_{l-1}(z)\end{matrix}}.
\end{align*}
  The study of zeroes of  OPUC has been a very active area, see for example \cite{Alfaro,Ambrolazde,Barrios-Lopez,Garcia,Godoy,Golinskii2,Mhaskar,Totik}, and   interesting applications to signal analysis theory \cite{Jones-1,Jones-2,Pan-1,Pan-2} have been found.  Let us stress that  in general Szegő polynomials do not provide  a dense set in the Hilbert space $L^2(\T,\mu)$;  Szegő prove that for a nontrivial probability measure $\mu$  on $\T$ with
 Verblunsky coefficients $\{\alpha_n\}_{n=0}^\infty$  the corresponding Szegő's polynomials are dense in $L^2(\T,\mu)$ if and
 only if $\prod_{n=0}^\infty (1-|\alpha_n|^2)=0$. This can be refined for an absolutely continuous probability measure, indeed the   Kolmogorov's density theorem  ensures that density of the  OPUC in $L^2(\T,\mu)$  happens if and only if  the so called Szegő's condition $\int_{\T}\log\big (w(\theta)\big)\d\theta=-\infty$ is satisfied \cite{Simon-S}.

 Toda integrable equations and OPUC are deeply bond. Adler and van Moerbeke studied such connections   \cite{Adler-Van-Moerbecke-Toeplitz} and introduced  the so called  Toeplitz lattice. Golinskii \cite{Golinskii}  has studied some classes of reductions of the Toeplitz lattice connected with Schur flows for a measure  invariant under conjugation, see \cite{Simon-Schur} \cite{Fay1}) and \cite{Mukaihira}. The Toeplitz lattice is equivalent  to the well known  Ablowitz--Ladik lattice introduced  in \cite{a-l-1,a-l-2}.
The integrable structure of Schur flows and its connection with Ablowitz--Ladik  was studied  from a Hamiltonian point of view in \cite{Nenciu}, and other works
 also introduce connections with Laurent polynomials and $\tau$-functions, like \cite{Fay2}, \cite{Fay3} and \cite{Bertola}.

The Cantero--Moral--Velázquez  (CMV) \cite{CMV}  has great advantages when analyzing orthogonality in the unit circle. We need to replace polynomials wit Laurent polynomials (OLPUC) which are essentially equivalent to the original Szegő polynomials. For example,  OLPUC are always dense in $L^2(\T,\mu)$ this is not true in general for the OPUC,   \cite{Bul} and \cite{Barroso-Vera}. This bijection  OLPUC-OPUC  implies the replacement of the Szegő recursion relations with a five-term relations similar to the OPRL situation.  Orthogonality in the unit circle has been studied from this new point of view in a number of papers, see
for example \cite{CMV-Simon,Killip}. Alternative or generic orders for the CMV ordering  can be found in \cite{Barroso-Snake}.
Despite these facts, we must mention that he discovery of the advantages of the CMV ordering goes back to previous work \cite{watkins}. The CMV ordering was used in \cite{carlos} to study the relations between OLPUC and Toda theory.

In 1992, when studying  2D quantum gravity a Riemann-Hilbert problem was solved in terms of OPRL, \cite{FIK}. Namely, it was found that the solution of a $2\times 2$
Riemann--Hilbert problem can be expressed in terms of orthogonal polynomials in
the real line and its Cauchy transforms. Deift and Zhou  combined these ideas with   a non-linear steepest descent analysis in a series of important works   \cite{deift1,deift2,deift3,deift4} which as a byproduct generated a large activity in the field.  To mention just a few  relevant results let us cite the study of strong asymptotic with applications in random matrix theory, \cite{deift1,deift5}, the analysis of determinantal point processes \cite{daems2,daems3,kuijlaars2,kuijlaars3}, orthogonal Laurent polynomials \cite{McLaughlin1, McLaughlin2} and  Painlevé equations \cite{kuijlaars4,dai}. For the case of OPUC,  a Riemann--Hilbert problem was discussed  in \cite{baik}, see also \cite{mf1,mf2}.
An excellent introduction, in the realm of integrable systems, of the Riemann Hilbert problem is given \cite{its}.
In the monograph we can find a modern account of the Painlevé equations and the Riemann--Hilbert or isomonodromy method.  For a very good account of the very close theme of the 21rst Hilbert problem see \cite{anosov}. For more on integrable systems, Painlevé equations and its discrete versions see \cite{AC,AHH}.

The study of equations for the recursion coefficients for  OPRL or OPUC has been a subject of interest.  The question of how the form of the weight and its properties, for example to satisfy a Pearson type equation,
  translates  to the recursion coefficients has been treated in several places, a good review is \cite{VAssche}. It was in \cite{freud0} were Géza Freud studied  weights  in $\R$ of exponential variation  $w(x)=|x|^\rho\exp(-|x|^m)$, $\rho>-1$ and $m>0$.  For $m=2,4,6$ he constructed  relations among them as well as determined its asymptotic behavior. However, Freud did not found the role of the discrete Painlevé I, that was discovered  later by Magnus \cite{magnus}.  For the unit circle and a weight of the form $w(\theta)=\exp(k\cos\theta)$, $k\in\mathbb R$, Periwal and Shevitz \cite{Periwal0,Periwal}, in the context of matrix models, found the discrete Painlevé II equation for the recursion relations of the corresponding orthogonal polynomials. This result was rediscovered latter and connected with  the Painlevé III equation  \cite{hisakado}.
  In \cite{baik0}  the discrete Painlevé II was found  using the Riemann--Hilbert problem given in   \cite{baik}, see also \cite{tracy}. For a nice account of the relation of these discrete Painlevé equations and integrable systems see \cite{cresswell}. We also mention the recent paper \cite{clarkson} where
  a discussion on the  relationship between the recurrence coefficients of orthogonal polynomials with respect to a semiclassical Laguerre weight and classical solutions of the fourth Painlevé equation can be found.

 Back in 1949, Krein \cite{krein1, krein2}  used orthogonal polynomials with matrix coefficients on the real line
 and thereafter were studied sporadically until the last decade of the XX century, being some relevant papers  \cite{bere},  \cite{geronimo} and  \cite{nikishin}.  For a  kind of discrete Sturm--Liouville operators  the scattering problem is solved in \cite{nikishin}, finding that the
 polynomials that satisfy a  relation of the form
 \begin{align*}
 xP_{k}(x)&=A_{k}P_{k+1}(x)+B_{k}P_{k}(x)+A_{k-1}^{*}P_{k-1}(x),& k&=0,1,\dots,
 \end{align*}
 are orthogonal with respect to a positive definite measure; i.e.,  a  matrix version of Favard's theorem.
 Then, in the 1990's and the 2000's it was  found that matrix orthogonal  polynomials (MOP) satisfy in some cases  properties as do the classical orthogonal polynomials.
For example, Laguerre, Hermite and  Jacobi polynomials, i.e., the scalar-type Rodrigues' formula
 \cite{duran20051,duran20052} and a second order differential equation  \cite{duran1997,duran2004,borrego}. It has been proven \cite{duran2008} that operators of the form
 $D$=${\partial}^{2}F_{2}(t)$+${\partial}^{1}F_{1}(t)$+${\partial}^{0}F_{0}$ have as eigenfunctions different infinite families of MOP's.  A new family of MOP's satisfying second order differential equations whose coefficients do not behave asymptotically as the identity
 matrix was found in
 \cite{borrego}; see also \cite{cantero}.
  We have studied \cite{cum,carlos2} matrix extensions of the generalized polynomials studied in \cite{adler-van-moerbeke,adler-vanmoerbeke-2}.
Recently, in \cite{nuevo}, we have extended the Christoffel transformation to MOPRL obtaining a new matrix Christoffel formula, and in \cite{nuevo2} more general transformations  --of Geronimus and Uvarov type-- where also considered.

  In \cite{CM} the Riemann--Hilbert problem
  for this matrix situation and the appearance of non-Abelian discrete versions of Painlevé I were explored, showing singularity confinement \cite{CMT}. The
 singularity analysis for a matrix discrete version of the Painlevé I
 equation was performed. It was found that the singularity confinement holds
 generically, i.e. in the whole space of parameters except possibly for
 algebraic subvarieties.
 For  an alternative discussion of the use of Riemann--Hilbert problem for MOPRL see  \cite{dominguez}. Let us mention that in \cite{miranian,miranian2} and \cite{Cafasso} the MOP are expressed in terms of Schur complements that play the role of determinants in the standard scalar case. In \cite{Cafasso} an study of matrix Szegő polynomials and the relation with a non Abelian Ablowitz--Ladik lattice is carried out, and in \cite{AM} the CMV ordering is applied to study orthogonal Laurent polynomials in the circle.

The layout of the paper is as follows. We start in \S \ref{S:biorthogonal} recalling some facts of measure theory in $\T$ and, in particular, we discuss matrices of measures in the unit circle.
We then proceed with the construction of matrix Szegő  polynomials in the quasi-definite scenario, we introduce reverse polynomials, their quasi-determinantal expressions, the Verblunsky coefficients and the Gauss--Borel factorization for the moment matrix, that leads to biorthogonal families of matrix polynomials. We also discuss some symmetry properties.  In \S \ref{S:RHP}, following \cite{Cima}, we  introduce in the first place the Cauchy transforms of the matrix orthogonal polynomials, and then discuss the Riemann--Hilbert problem in this more general context. 
Next, in \S \ref{S:M} we apply the Riemann--Hilbert problem to derive the Szegő recursion relations and some differential relations.  We conclude the paper in \S \ref{S:dPII} with the finding of a matrix discrete Painlevé II equation, which holds for the Verblunsky matrices whenever the matrix of weights has a right logarithmic derivative of certain form ---of Fuchsian and non-Fuchsian singularity type--- associated with a monodromy free system. This Pearson type equation  for the matrix of weights allows us to avoid serious difficulties,  that appear if one insists in deriving these matrix discrete Painlevé  equations directly from  explicit Freud's weights.  Let us notice that these  discrete Painlevé  systems have non local terms, involving not only near neighbors. This issue is also consider and non trivial reductions  that get rid of these non local contributions are presented.
We also include, in Appendix \ref{S:examples}, some examples of matrices of weights, for the Fuchsian situation, constructed as solutions to the mentioned  Pearson equations.

\section{Matrix Szegő biorthogonal polynomials}\label{S:biorthogonal}

In this section we consider the matrix extension of Szegő biorthogonal polynomials in the unit circle \cite{szego}, see \cite{Cafasso,AM}.

\subsection{Matrices of measures in the unit circle}
We recall here some facts regarding measure theory on the unit circle $\mathbb{T}$, we follow \cite{Rudin,Evans,Cima}.
The Lebesgue measure,  for $\zeta\in \T$, is
\begin{align*}
\d m(\zeta):=\frac{\d \zeta}{2\pi\operatorname{i}\zeta}.
\end{align*}
We shall consider a matrix of  finite complex valued Borel measures $\mu$ supported in ${\mathbb T}$, we denote the set of such measures by $\mathcal B$.  A matrix of measures $\mu\in\mathcal B$ is said absolutely continuous, with respect to the Lebesgue measure $m$, written $\mu\ll m$, if $\mu(A)=0_N$, where $0_N\in\C^{N\times N}$ denotes the zero matrix,  for all Borel sets $A$ of Lebesgue zero measure, $m(A)=0$. A matrix of measures $\mu\in\mathcal B$ in singular, with respect to the Lebesgue measure $m$,  written $\mu\perp m$,  if  for two disjoint Borel sets $A,B$ such that $\T=A\cup B$ we have $\mu(A)=0_N$ and $m(B)=0$.  Any matrix of measures $\mu\in\mathcal B$ can be decomposed uniquely $\mu=\mu_a+\mu_s$, $\mu_a\ll m $ and $\mu_s\perp m$.
As was proven by Radon and Nikodym a $\mu\in\mathcal B$ is absolutely continuous  if and only if there exist a matrix $w:\T\to\C^{N\times N}$ built up with  $L^1(\T,\mu)$ weights such that $\d\mu(\zeta)=w(\zeta)\d m(\zeta)$, i.e.,
\begin{align*}
\mu(B)=\int_{B}w(\zeta)\d m(\zeta);
\end{align*}
the  matrix of weights $w(\zeta)$ is called the matrix of Radon--Nikodym derivatives  and we write
\begin{align*}
w(\zeta)=\frac{\d\mu}{\d m}(\zeta).
\end{align*}
Therefore, according to the Lebesgue--Radon--Nikodym theorem \cite{Rudin,Evans} for any matrix of measures $\mu\in\mathcal B$ and any Borel set $B\subset \T$ we can write
\begin{align*}
\mu(B)=\int_B\frac{\d\mu}{\d m}(\zeta)\d m(\zeta)+\mu_s(B),
\end{align*}
where,  following \cite{Cima}, we have introduced the matrix of Radon--Nikodym derivatives of $\mu$, $\mu=\mu_a+\mu_s$, with respect to the Lebesgue measure, as the matrix of Radon--Nikodym derivatives of its absolutely continuous  component,
\begin{align*}
\dfrac{\d\mu}{\d m}:=\dfrac{\d\mu_a}{\d m}.
\end{align*}
For any Borel measure  we can consider its differential \cite{Rudin,Evans,Cima}, let $I(\zeta,t)$ be the arc of  the unit circle subtended by the points $\zeta \operatorname{e}^{\operatorname{i}t}$ and $\zeta \operatorname{e}^{-\operatorname{i}t}$ and consider, for $\mu\in\mathcal B$,
\begin{align*}
(\underline{D}\mu)(\zeta)&:=\liminf_{t\to 0^+}\frac{\mu(I(\zeta,t))}{m(I(\zeta,t))},&
(\overline{D}\mu)(\zeta)&:=\limsup_{t\to 0^+}\frac{\mu(I(\zeta,t))}{m(I(\zeta,t))}.
\end{align*}
When these two matrices are bounded and equal,  $(\underline{D}\mu)(\zeta)=(\overline{D}\mu)(\zeta)=:({D}\mu)(\zeta)$, we say that $\mu$ is differentiable (with respect to the Lebesgue measure) at $\zeta\in\T$ with matrix of differentials $D\mu(\zeta)$.
Then, see \cite{Rudin, Evans,Cima}, $\mu\in\mathcal B$ is differentiable $m$-almost for every $\zeta\in\T$, moreover its matrix of differentials $D\mu\in \big(L^1(\T,\mu)\big)^{N\times N}$ is  a matrix of  integrable functions and for any Borel set $B$ we have
\begin{align*}
\mu(B)=\int_B (D\mu)(\zeta)\d m(\zeta)+\mu_s(B),
\end{align*}
where $\mu_s\perp m$ and $D\mu_s(\zeta)=0$ for $m$-almost every $\zeta\in\T$. In this situation,  the matrix of differentials and  the matrix of Radon--Nikodym derivatives coincide,
$(D\mu)(\zeta)=\dfrac{\d\mu}{\d m}(\zeta)$,
	for $m$-almost every $\zeta\in\T$.

\subsection{Matrix Szegő polynomials on the unit circle}
Here we  follow \cite{Cafasso} and \cite{AM}.
\begin{definition}[Szegő matrix polynomials]\label{def:szego}
	Given a matrix of measures  $\mu$, the  left and right \textit{monic matrix Szegő polynomials} $P^{L}_{1,n}(z)$, $P^{R}_{1,n}(z)$, $P^{L}_{2,n}(z)$, $P^{R}_{2,n}(z)$ are monic polynomials
	\begin{align*}
P^{L}_{1,n}(z) &=P^{L}_{1,n,0}+\cdots +P^{L}_{1,n,n-1}z^{n-1}+I_Nz^n,&
P^{R}_{1,n}(z) &=P^{R}_{1,n,0}+\cdots+P^{R}_{1,n,n-1}z^{n-1}+I_Nz^n,  \\
P^{L}_{2,n}(z) &=P^{L}_{2,n,0}+\cdots+P^{L}_{2,n,n-1}z^{n-1}+I_Nz^n,&
P^{R}_{2,n}(z) &=P^{R}_{2,n,0}+\cdots+P^{R}_{2,n,n-1}z^{n-1}+I_Nz^n,
	\end{align*}
where $I_N\in\C^{N\times N}$ is the identity matrix and  $P^L_{1,n,j},P^R_{1,n,j},P^L_{2,n,j},P^R_{2,n,j}\in\mathbb C^{N\times N}$, such that the following orthogonality conditions
\begin{align}
\oint_{\mathbb{T}}P^{L}_{1,n}(\zeta)
\d \mu(\zeta)\bar\zeta^{j}&=0_N,
\label{eq:ortogonal}\\
\oint_{\mathbb{T}}\bar\zeta^{j}\d \mu(\zeta)P^{R}_{1,n}(\zeta)
&=0_N,
\label{eq:ortogonall}\\
\oint_{\mathbb{T}}
\zeta^{j}\d \mu(\zeta)\big(P^{L}_{2,n}(\zeta)\big)^\dagger
&=0_N,
\label{eq:2ortogonal}\\
\oint_{\mathbb{T}}\big(P^{R}_{2,n}(\zeta)\big)^\dagger
\d \mu(\zeta)\zeta^{j}&=0_N,
\label{eq:2ortogonall}
\end{align}
stand for all $j\in\{0, \dots, n-1\}$.
\end{definition}

From the second families of left and right Szegő matrix polynomials $P^L_{2,n}(z)$ and $P^R_{2,n}(z)$ we construct
\begin{definition}[Reciprocal  Szegő polynomials]\label{def:reciprocal}
The reciprocal (or reverse) left and right Szegő matrix polynomials $\tilde P^L_{2,n}(z)$ and $\tilde P^R_{2,n}(z)$ are given by
\begin{align*}
\tilde P^L_{2,n}(z)&:=z^n \big(P^L_{2,n}\big(\bar z^{-1}\big)\big)^\dagger=I_N+\big(P^{L}_{2,n,n-1}\big)^\dagger z+\cdots+\big(P^{L}_{2,n,0}\big)^\dagger z^n,\\
\tilde P^R_{2,n}(z)&:=z^n \big(P^R_{2,n}\big(\bar z^{-1}\big)\big)^\dagger=I_N+\big(P^{R}_{2,n,n-1}\big)^\dagger z+\cdots+\big(P^{R}_{2,n,0}\big)^\dagger z^n.
\end{align*}
\end{definition}

\begin{definition}[Verblunsky matrices]
	The Verblunsky matrices are the evaluations at the origin, $z=0$, of the Szegő polynomials
	\begin{align*}
		\alpha_{1,n}^L&:=P_{1,n}^L(0), & 	\alpha_{1,n}^R&:=P^R_{1,n}(0),&  	\alpha_{2,n}^L&:=P_{2,n}^L(0) ,&	\alpha_{2,n}^R&:=P_{2,n}^R(0).
	\end{align*}
\end{definition}

\begin{proposition}
	In terms of the Verblunsky coefficients, the Szegő matrix polynomials of type 1 and its reciprocals can be written as follows
	\begin{align*}
	P^{L}_{1,n}(z) &=\alpha^{L}_{1,n}+\cdots+P^{L}_{1,n,n-1}z^{n-1}+I_Nz^n,&
	P^{R}_{1,n}(z) &=\alpha^{R}_{1,n}+\cdots+P^{R}_{1,n,n-1}z^{n-1}+I_Nz^n,\\
\tilde P^L_{2,n}(z)&=I_N+\big(P^{L}_{2,n,n-1}\big)^\dagger z+\cdots+\big(\alpha^{L}_{2,n}\big)^\dagger z^n,&
\tilde P^R_{2,n}(z)&=I_N+\big(P^{R}_{2,n,n-1}\big)^\dagger z+\cdots+\big(\alpha^{R}_{2,n}\big)^\dagger z^n.
	\end{align*}
\end{proposition}

\begin{proposition}
The reciprocal   Szegő matrix polynomials $\tilde P^L_{2,n}(z)$ and $\tilde P^R_{2,n}(z)$ satisfy the following orthogonality relations
\begin{align}
\oint_{\mathbb{T}}
\bar\zeta^{j}\d \mu(\zeta) \tilde P^{L}_{2,n}(\zeta)
&=0_N,
\label{eq:r2ortogonal}\\
\oint_{\mathbb{T}}\tilde P^{R}_{2,n}(\zeta)
\d \mu(\zeta)\bar\zeta^{j}&=0_N.
\label{eq:r2ortogonall}
\end{align}
for  all $j\in\{1, \dots, n\}$.

\end{proposition}
\begin{proof}
	From \eqref{eq:2ortogonal} and \eqref{eq:2ortogonall} we get for $j\in\{0, \dots, n-1\}$.
	\begin{align*}
	\oint_{\mathbb{T}}
	\zeta^{j-n}\d \mu(\zeta) \tilde P^{L}_{2,n}(\zeta)&=0_N,&
	\oint_{\mathbb{T}}\tilde P^{R}_{2,n}(\zeta)
	\d \mu(\zeta)\zeta^{j-n}&=0_N,
	\end{align*}
	and relabeling the indexes we get  the stated orthogonality relations.
\end{proof}
Following  \cite{AM} we introduce
\begin{definition}\label{def:moments}
	The moments or Fourier coefficients $\hat \mu(j)\in\C^{N\times N}$ of the matrix of measures $\mu$ are defined by
	\begin{align*}
	\hat \mu(j)&:=\oint_{\mathbb T}  \bar\zeta^{j}\d\mu(\zeta),
	\end{align*}
	with $ j\in\mathbb Z$.
\end{definition}

\begin{definition}
	We introduce the  left and right semi-infinite moment  matrices
	\begin{align*}
	\mathcal	M^L&:=
		\PARENS{\begin{matrix}
			\hat \mu(0)& \hat \mu(-1)&\hat \mu(-2)&\cdots\\
			\hat \mu(1) & \hat \mu(0)&\hat \mu(-1)& \\
			\hat \mu(2)&\hat \mu(1)& \hat \mu(0)&\ddots \\
			\vdots& &\ddots &\ddots\\
		\end{matrix}},&
	\mathcal	M^R&:=\PARENS{\begin{matrix}
		\hat \mu(0)& \hat \mu(1)&\hat \mu(2)&\cdots\\
		\hat \mu(-1) & \hat \mu(0)&\hat \mu(1)& \\
		\hat \mu(-2)&\hat \mu(-1)& \hat \mu(0)&\ddots \\
		\vdots& &\ddots &\ddots\\
		\end{matrix}},
	\end{align*}
and its truncations
\begin{align*}
\mathcal M^L_{[n]}&:=\PARENS{\begin{matrix}
\hat \mu(0)& \hat \mu(-1)&\hat \mu(-2)&\cdots&\hat \mu(-n+1)\\
\hat \mu(1) & \hat \mu(0) &\hat \mu(-1)& &\hat \mu(-n+2)\\
\hat \mu(2) &\hat \mu(1)& \hat \mu(0) &\ddots &\hat \mu(-n+3)\\
\vdots& &\ddots &\ddots\\
\hat \mu(n-1) & \hat \mu(n-2) &\hat \mu(n-3)&\cdots &\hat \mu(0)
\end{matrix}},\\
\mathcal M^R_{[n]}&:=\PARENS{\begin{matrix}
\hat \mu(0)& \hat \mu(1)&\hat \mu(2)&\cdots&\hat \mu(n-1)\\
\hat \mu(-1) & \hat \mu(0) &\hat \mu(-1)& &\hat \mu(n-2)\\
\hat \mu(-2) &\hat \mu(-1)& \hat \mu(0) &\ddots &\hat \mu(n-3)\\
\vdots& &\ddots &\ddots\\
\hat \mu(-n+1) & \hat \mu(-n+2) &\hat \mu(-n+3)&\cdots &\hat \mu(0)
\end{matrix}}.
\end{align*}
The matrix of measures  $\d\mu(\zeta)$ is   quasi-definite if $\det \mathcal M^L_{[n]}\neq 0$ and $\det\mathcal M^R_{[n]}\neq 0$
for all $n\in\{1,2,\dots\}$.
\end{definition}
Observe that these truncated moment matrices are block Toeplitz matrices organized by block diagonals.
We now need of the notion of quasi-determinant, see \cite{Gelfand1991Determinants,Gelfand1995Noncommutative,Gelfand2005,olver}.
\begin{proposition}
The matrix of measures  $\mu$ is   quasi-definite if the last quasi-determinants $ \Theta_*M^L_{[n]}$ and $ \Theta_*M^R_{[n]}$ are not singular matrices.
\end{proposition}
\begin{proposition}\label{pro:quasideterminants0}
	The   Szegő matrix  polynomials exists whenever the matrix of measures $\mu$ in  quasi-definite. Moreover, they can be expressed in terms of last quasi-determinants of bordered truncated moment matrices
\begin{align*}
P_{1,n}^L(z)&=\Theta_*\PARENS{\begin{matrix}
\hat\mu(0)& \hat\mu(1)&\hat\mu(2)&\cdots&\hat\mu(n-1)&I_N\\
\hat\mu(-1) & \hat\mu(0) &\hat\mu(1)& &\hat\mu(n-2)&I_N z\\
\hat\mu(-2) &\hat\mu(-1)& \hat\mu(0) &\ddots &\hat\mu(n-3)& I_N z^2\\
\vdots& &\ddots &\ddots\\
\hat\mu(-n+1) & \hat\mu(-n+2) &\hat\mu(-n+3)&\cdots &\hat\mu(0)&I_N z^{n-1}\\
\hat\mu(-n) & \hat\mu(-n+1) &\hat\mu(-n+2) & \cdots & \hat\mu(-1) &I_N z^n
\end{matrix}}, \\
P_{1,n}^R(z)&=\Theta_*\PARENS{\begin{matrix}
\hat\mu(0)& \hat\mu(-1)&\hat\mu(-2)&\cdots&\hat\mu(-n+1)&\hat\mu(-n)\\
\hat\mu(1) & \hat\mu(0) &\hat\mu(-1)& &\hat\mu(-n+2)&\hat\mu(-n+1)\\
\hat\mu(2) &\hat\mu(1)& \hat\mu(0) &\ddots &\hat\mu(-n+3)&\hat\mu(-n+2)\\
\vdots& &\ddots &\ddots&\vdots&\vdots\\
\hat\mu(n-1) & \hat\mu(n-2) &\hat\mu(n-3)&\cdots &\hat\mu(0)
&\hat\mu(-1)\\
I_N &I_N z & I_Nz^2&\dots & I_Nz^{n-1} &I_Nz^n
\end{matrix}},
\end{align*}
and
\begin{align*}
\big(P^L_{2,n}(z)\big)^\dagger
&=\Theta_*\PARENS{\begin{matrix}
\hat\mu(0)& \hat\mu(-1)&\hat\mu(-2)&\cdots&\hat\mu(-n+1)&\hat\mu(1)\\
\hat\mu(1) & \hat\mu(0) &\hat\mu(-1)& &\hat\mu(-n+2)&\hat\mu(2)\\
\hat\mu(2) &\hat\mu(1)& \hat\mu(0) &\ddots &\hat\mu(-n+3)&\hat\mu(3)\\
\vdots& &\ddots &\ddots&\vdots&\vdots\\
\hat\mu(n-1) & \hat\mu(n-2) &\hat\mu(n-3)&\cdots &\hat\mu(0)
&\hat\mu(n)\\I_N &I_N \bar z & I_N \bar z ^2 & \cdots & I_N\bar z ^{n-1} & I_N\bar z ^n
\end{matrix}},\\
\big(P^R_{2,n}(z)\big)^\dagger
&=\Theta_*\PARENS{\begin{matrix}
\hat\mu(0)& \hat\mu(1)&\hat\mu(2)&\cdots&\hat\mu(n-1)& I_N\\
\hat\mu(-1) & \hat\mu(0) &\hat\mu(1)& &\hat\mu(n-2)& I_N \bar z \\
\hat\mu(-2) &\hat\mu(-1)& \hat\mu(0) &\ddots &\hat\mu(n-3) & I_N\bar z^2\\
\vdots& &\ddots &\vdots&\vdots&\\
\hat\mu(-n+1) & \hat\mu(-n+2) &\hat\mu(-n+3)&\cdots &\hat\mu(0)&I_N \bar z^{n-1}\\
\hat\mu(1) &\hat\mu(2) & \hat\mu(3) & \cdots  & \hat\mu(n) & I_N \bar z^n
\end{matrix}}.
\end{align*}
\end{proposition}
\begin{proof}
	In terms of  moments of the matrix of measures the orthogonality relations \eqref{eq:ortogonal} and \eqref{eq:ortogonall} read
\begin{align*}
(P^L_{1,n,0},\dots, P^L_{1,n,n-1})\PARENS{\begin{matrix}
\hat\mu(0)& \hat\mu(1)&\hat\mu(2)&\cdots&\hat\mu(n-1)\\
\hat\mu(-1) & \hat\mu(0) &\hat\mu(1)& &\hat\mu(n-2)\\
\hat\mu(-2) &\hat\mu(-1)& \hat\mu(0) &\ddots &\hat\mu(n-3)\\
\vdots& &\ddots &\ddots\\
\hat\mu(-n+1) & \hat\mu(-n+2) &\hat\mu(-n+3)&\cdots &\hat\mu(0)
\end{matrix}}
&=
-(\hat\mu(-n),\dots,\hat\mu(-1)), \\
\PARENS{\begin{matrix}
\hat\mu(0)& \hat\mu(-1)&\hat\mu(-2)&\cdots&\hat\mu(-n+1)\\
\hat\mu(1) & \hat\mu(0) &\hat\mu(-1)& &\hat\mu(-n+2)\\
\hat\mu(2) &\hat\mu(1)& \hat\mu(0) &\ddots &\hat\mu(-n+3)\\
\vdots& &\ddots &\ddots\\
\hat\mu(n-1) & \hat\mu(n-2) &\hat\mu(n-3)&\cdots &\hat\mu(0)
\end{matrix}}
\PARENS{\begin{matrix}
P^R_{1,n,0}\\\vdots\\ P^R_{1,n,n-1}
\end{matrix}}&=
-\PARENS{\begin{matrix}
\hat\mu(-n)\\\vdots\\
\hat\mu(-1)
\end{matrix}},
\end{align*}
while orthogonality relations \eqref{eq:2ortogonal} and \eqref{eq:2ortogonall} read
\begin{align*}
\PARENS{\begin{matrix}
\hat\mu(0)& \hat\mu(-1)&\hat\mu(-2)&\cdots&\hat\mu(-n+1)\\
\hat\mu(1) & \hat\mu(0) &\hat\mu(-1)& &\hat\mu(-n+2)\\
\hat\mu(2) &\hat\mu(1)& \hat\mu(0) &\ddots &\hat\mu(-n+3)\\
\vdots& &\ddots &\ddots\\
\hat\mu(n-1) & \hat\mu(n-2) &\hat\mu(n-3)&\cdots &\hat\mu(0)
\end{matrix}}\PARENS{\begin{matrix}
\big(P^L_{2,n,0}\big)^\dagger\\ \vdots\\ \big(P^L_{2,n,n-1}\big)^\dagger
\end{matrix}}&=-\PARENS{\begin{matrix}
\hat\mu(1)\\\vdots\\
\hat\mu(n)
\end{matrix}}, \\
\Big(\big(P^R_{2,n,0}\big)^\dagger,\dots,\big(P^R_{2,n,n-1}\big)^\dagger\Big)
\PARENS{\begin{matrix}
\hat\mu(0)& \hat\mu(1)&\hat\mu(2)&\cdots&\hat\mu(n-1)\\
\hat\mu(-1) & \hat\mu(0) &\hat\mu(1)& &\hat\mu(n-2)\\
\hat\mu(-2) &\hat\mu(-1)& \hat\mu(0) &\ddots &\hat\mu(n-3)\\
\vdots& &\ddots &\ddots\\
\hat\mu(-n+1) & \hat\mu(-n+2) &\hat\mu(-n+3)&\cdots &\hat\mu(0)
\end{matrix}}
&=
-(\hat\mu(1),\dots,\hat\mu(n)).
\end{align*}
Thus, assuming  the quasi-definite condition we get
\begin{align*}
(P^L_{1,n,0},\dots, P^L_{1,n,n-1})
&=
-(\hat\mu(-n),\dots,\hat\mu(-1))\PARENS{\begin{matrix}
\hat\mu(0)& \hat\mu(1)&\hat\mu(2)&\cdots&\hat\mu(n-1)\\
\hat\mu(-1) & \hat\mu(0) &\hat\mu(1)& &\hat\mu(n-2)\\
\hat\mu(-2) &\hat\mu(-1)& \hat\mu(0) &\ddots &\hat\mu(n-3)\\
\vdots& &\ddots &\ddots\\
\hat\mu(-n+1) & \hat\mu(-n+2) &\hat\mu(-n+3)&\cdots &\hat\mu(0)
\end{matrix}}^{-1}, \\
\PARENS{\begin{matrix}
P^R_{1,n,0}\\\vdots\\ P^R_{1,n,n-1}
\end{matrix}}&=
-\PARENS{\begin{matrix}
\hat\mu(0)& \hat\mu(-1)&\hat\mu(-2)&\cdots&\hat\mu(-n+1)\\
\hat\mu(1) & \hat\mu(0) &\hat\mu(-1)& &\hat\mu(-n+2)\\
\hat\mu(2) &\hat\mu(1)& \hat\mu(0) &\ddots &\hat\mu(-n+3)\\
\vdots& &\ddots &\ddots\\
\hat\mu(n-1) & \hat\mu(n-2) &\hat\mu(n-3)&\cdots &\hat\mu(0)
\end{matrix}}^{-1}\PARENS{\begin{matrix}
\hat\mu(-n)\\\vdots\\
\hat\mu(-1)
\end{matrix}},
\end{align*}
and
\begin{align*}
\PARENS{\begin{matrix}
\big(P^L_{2,n,0}\big)^\dagger\\ \vdots\\ \big(P^L_{2,n,n-1}\big)^\dagger
\end{matrix}}&=-\PARENS{\begin{matrix}
\hat\mu(0)& \hat\mu(-1)&\hat\mu(-2)&\cdots&\hat\mu(-n+1)\\
\hat\mu(1) & \hat\mu(0) &\hat\mu(-1)& &\hat\mu(-n+2)\\
\hat\mu(2) &\hat\mu(1)& \hat\mu(0) &\ddots &\hat\mu(-n+3)\\
\vdots& &\ddots &\ddots\\
\hat\mu(n-1) & \hat\mu(n-2) &\hat\mu(n-3)&\cdots &\hat\mu(0)
\end{matrix}}^{-1}\PARENS{\begin{matrix}
\hat\mu(1)\\\vdots\\
\hat\mu(n)
\end{matrix}}, \\
\Big(\big(P^R_{2,n,0}\big)^\dagger,\dots,\big(P^R_{2,n,n-1}\big)^\dagger\Big)
&=
-(\hat\mu(1),\dots,\hat\mu(n))\PARENS{\begin{matrix}
\hat\mu(0)& \hat\mu(1)&\hat\mu(2)&\cdots&\hat\mu(n-1)\\
\hat\mu(-1) & \hat\mu(0) &\hat\mu(1)& &\hat\mu(n-2)\\
\hat\mu(-2) &\hat\mu(-1)& \hat\mu(0) &\ddots &\hat\mu(n-3)\\
\vdots& &\ddots &\ddots\\
\hat\mu(-n+1) & \hat\mu(-n+2) &\hat\mu(-n+3)&\cdots &\hat\mu(0)
\end{matrix}}^{-1}.
\end{align*}
Therefore, for the first family of left Szegő matrix polynomials we have
\begin{align*}
P^L_{1,n}(z)&=I_Nz^n+(P^L_{1,n,0},\dots, P^L_{1,n,n-1})\PARENS{\begin{matrix}
I_N\\ \vdots\\I_Nz^{n-1}
\end{matrix}}\\&=I_Nz^n-(\hat\mu(-n),\dots,\hat\mu(-1))\PARENS{\begin{matrix}
\hat\mu(0)& \hat\mu(1)&\hat\mu(2)&\cdots&\hat\mu(n-1)\\
\hat\mu(-1) & \hat\mu(0) &\hat\mu(1)& &\hat\mu(n-2)\\
\hat\mu(-2) &\hat\mu(-1)& \hat\mu(0) &\ddots &\hat\mu(n-3)\\
\vdots& &\ddots &\ddots\\
\hat\mu(-n+1) & \hat\mu(-n+2) &\hat\mu(-n+3)&\cdots &\hat\mu(0)
\end{matrix}}^{-1}\PARENS{\begin{matrix}
I_N\\ \vdots\\ I_Nz^{n-1}
\end{matrix}},
\end{align*}
while for the first family of right Szegő polynomials we find
\begin{align*}
P^R_{1,n}(z)&=I_Nz^n+(I_N,\dots,I_Nz^{n-1})\PARENS{\begin{matrix}
P^R_{1,n,0}\\\vdots\\ P^R_{1,n,n-1}
\end{matrix}}\\
&=I_Nz^n-(I_N,\dots,I_Nz^{n-1})\PARENS{\begin{matrix}
\hat\mu(0)& \hat\mu(-1)&\hat\mu(-2)&\cdots&\hat\mu(-n+1)\\
\hat\mu(1) & \hat\mu(0) &\hat\mu(-1)& &\hat\mu(-n+2)\\
\hat\mu(2) &\hat\mu(1)& \hat\mu(0) &\ddots &\hat\mu(-n+3)\\
\vdots& &\ddots &\ddots\\
\hat\mu(n-1) & \hat\mu(n-2) &\hat\mu(n-3)&\cdots &\hat\mu(0)
\end{matrix}}^{-1}\PARENS{\begin{matrix}
\hat\mu(-n)\\\vdots\\
\hat\mu(-1)
\end{matrix}}.
\end{align*}
For the second family of left Szegő matrix polynomials we deduce
\begin{align*}
\big(P^L_{2,n}(z)\big)^\dagger&=\bar z ^n I_N+ \big(1,\dots, \bar z ^{n-1}\big)\PARENS{\begin{matrix}
\big(P^L_{2,n,0}\big)^\dagger\\ \vdots\\\big(P^L_{2,n,n-1}\big)^\dagger
\end{matrix}}\\
&=\bar z ^n I_N -\big(I_N,\dots, I_N\bar z ^{n-1}\big)\PARENS{\begin{matrix}
\hat\mu(0)& \hat\mu(-1)&\hat\mu(-2)&\cdots&\hat\mu(-n+1)\\
\hat\mu(1) & \hat\mu(0) &\hat\mu(-1)& &\hat\mu(-n+2)\\
\hat\mu(2) &\hat\mu(1)& \hat\mu(0) &\ddots &\hat\mu(-n+3)\\
\vdots& &\ddots &\ddots\\
\hat\mu(n-1) & \hat\mu(n-2) &\hat\mu(n-3)&\cdots &\hat\mu(0)
\end{matrix}}^{-1}\PARENS{\begin{matrix}
\hat\mu(1)\\\vdots\\
\hat\mu(n)
\end{matrix}},
\end{align*}
while for the second family of right Szegő matrix polynomials we get
\begin{align*}
\big(P^R_{2,n}(z)\big)^\dagger&=\bar z ^n I_N+\Big(\big(P^R_{2,n,0}\big)^\dagger,\dots,\big(P^R_{2,n,n-1}\big)^\dagger\Big)\PARENS{\begin{matrix}
I_N\\\vdots\\I_N\bar z ^{n-1}\end{matrix}}\\
&=\bar z ^n I_N
-(\hat\mu(1),\dots,\hat\mu(n))\PARENS{\begin{matrix}
\hat\mu(0)& \hat\mu(1)&\hat\mu(2)&\cdots&\hat\mu(n-1)\\
\hat\mu(-1) & \hat\mu(0) &\hat\mu(1)& &\hat\mu(n-2)\\
\hat\mu(-2) &\hat\mu(-1)& \hat\mu(0) &\ddots &\hat\mu(n-3)\\
\vdots& &\ddots &\ddots\\
\hat\mu(-n+1) & \hat\mu(-n+2) &\hat\mu(-n+3)&\cdots &\hat\mu(0)
\end{matrix}}^{-1}
\PARENS{\begin{matrix}
I_N\\\vdots\\ I_N\bar z ^{n-1}
\end{matrix}}.
\end{align*}
From these relations the quasi-determinantal expressions follow immediately.
\end{proof}

Observe that this result is just informing us that the Szegő matrix polynomials can be expressed as \emph{rational} functions of the moments. For example,
\begin{align*}
P^L_{1,1}(z)&=\Theta_*\PARENS{\begin{matrix}
\hat\mu(0)&I_N\\
\hat\mu(-1)& I_N z
\end{matrix}}\\&=-\hat\mu(-1)(\hat\mu(0))^{-1}+I_N z,\\
P^L_{1,2}(z)&=\Theta_*\PARENS{\begin{matrix}
\hat\mu(0)&\hat\mu(1)&I_N\\
\hat\mu(-1)& \hat\mu(0)&I_N z\\
\hat\mu(-2) &\hat\mu(-1) &I_N z^2
\end{matrix}}\\
&=\begin{multlined}[t]
-\hat\mu(-2)(\hat\mu(0))^{-1}+\big(\hat\mu(-1)+\hat\mu(-2)(\hat\mu(0))^{-1}\hat\mu(1)\big)\big(\hat\mu(0)-\hat\mu(-1)(\hat\mu(0))^{-1}\hat\mu(1)\big)^{-1}\hat\mu(-1)(\hat\mu(0))^{-1}\\-
\big(\hat\mu(-1)+\hat\mu(-2)(\hat\mu(0))^{-1}\hat\mu(1)\big)\big(\hat\mu(0)-\hat\mu(-1)(\hat\mu(0))^{-1}\hat\mu(1)\big)^{-1}z+I_N z^2.
\end{multlined}
\end{align*}
In fact, they are polynomials in the moments and the inverses of the   last quasi-determinants $\Theta_*M^L_{[n]}$ and
$ \Theta_*M^R_{[n]}$.

These expressions allow us to find
\begin{proposition}\label{pro:quasideterminants}
	The reciprocal Szegő polynomials have the following quasi-determinantal expressions
	\begin{align*}
\tilde P^L_{2,n}(z)
	&=\Theta_*\PARENS{\begin{matrix}
	\hat\mu(0)& \hat\mu(-1)&\hat\mu(-2)&\cdots&\hat\mu(-n+1)&\hat\mu(1)\\
	\hat\mu(1) & \hat\mu(0) &\hat\mu(-1)& &\hat\mu(-n+2)&\hat\mu(2)\\
	\hat\mu(2) &\hat\mu(1)& \hat\mu(0) &\ddots &\hat\mu(-n+3)&\hat\mu(3)\\
	\vdots& &\ddots &\ddots&\vdots&\vdots\\
	\hat\mu(n-1) & \hat\mu(n-2) &\hat\mu(n-3)&\cdots &\hat\mu(0)
	&\hat\mu(n)\\I_N z^n&I_N z^{n-1}& I_N z^{n-2} & \cdots & I_N z & I_N
	\end{matrix}},\\
\tilde	P^R_{2,n}(z)
	&=\Theta_*\PARENS{\begin{matrix}
	\hat\mu(0)& \hat\mu(1)&\hat\mu(2)&\cdots&\hat\mu(n-1)& I_N z^n\\
	\hat\mu(-1) & \hat\mu(0) &\hat\mu(1)& &\hat\mu(n-2)& I_N z^ {n-1}\\
	\hat\mu(-2) &\hat\mu(-1)& \hat\mu(0) &\ddots &\hat\mu(n-3) & I_Nz^{n-2}\\
	\vdots& &\ddots &\vdots&\vdots&\\
	\hat\mu(-n+1) & \hat\mu(-n+2) &\hat\mu(-n+3)&\cdots &\hat\mu(0)&I_N z\\
	\hat\mu(1) &\hat\mu(2) & \hat\mu(3) & \cdots  & \hat\mu(n) & I_N
	\end{matrix}}
	\end{align*}
\end{proposition}

\begin{definition}
	The Gauss--Borel factorization of the moments matrices is
	\begin{align}\label{eq:gauss}
\mathcal	M^L&=(S^L_1)^{-1} H^L (S^L_2)^{-\dagger}, &
	\mathcal	M^R=(S_2^R)^{-1} H^R (S_1^R)^{-\dagger},
	\end{align}
	where $S_1^L,S^L_2,S^R_1,S^R_2$ are lower unitriangular block semi-infinite matrices and $H^L$ and $H^R$ are block diagonal matrices.
\end{definition}

\begin{proposition}
	The Gauss--Borel factorization can be performed when the the matrix of measures $\mu$ is quasi-definite.
\end{proposition}
\begin{proof}
	Follows the proof of Proposition 1 of \cite{AM} by replacing the moments matrices there by our moment matrices.
\end{proof}
\begin{definition}
	We introduce the semi-infinite vector of monomials
	\begin{align*}
	\chi(z)=\PARENS{\begin{matrix}
	I_N\\I_Nz\\I_Nz^2\\\vdots
	\end{matrix}},
	\end{align*}
and	the semi-infinite vectors of 	polynomials
	\begin{align*}
	P^L_1(z)&:=S^L_1\chi(z), & P^L_2(z)&:=S^L_2\chi(z), & (P^R_1(z))^\top&:=(\chi(z))^\top (S^R_1)^\dagger, & (P^R_2(z))^\top&:=(\chi(z))^\top (S^R_2)^\dagger.
	\end{align*}
\end{definition}
\begin{proposition}
	The moment matrices can be written as follows
\begin{align*}
M^L&=\oint_\T \chi(\zeta) \d \mu(\zeta) (\chi(\zeta))^\dagger, &
M^R&=\oint_\T \big((\chi(\zeta)) ^\top\big)^\dagger \d \mu(\zeta)(\chi(\zeta))^\top.
\end{align*}
\end{proposition}

\begin{proposition}
	We have the biothogonality relations
	\begin{align}
\oint_\T P^L_1(\zeta)
\d\mu(\zeta)
\big(P^L_2(\zeta)\big)^\dagger=H^L,\label{eq:biorthogonalityL}\\
	\oint_\T \big((P^R_2(\zeta) )^\top\big) ^\dagger
	\d \mu(\zeta)
	\big(P^R_1(\zeta)\big)^\top&=	H^R.\label{eq:biorthogonalityR}
	\end{align}
\end{proposition}
\begin{proof}
	To prove \eqref{eq:biorthogonalityL} we just notice that
	\begin{align*}
	\oint_\T P^L_1(\zeta) \d \mu(\zeta)
	\big(P^L_2(\zeta)\big)^\dagger&=S_1^L \oint_\T \chi(\zeta) \d \mu(\zeta)
	\big(\chi(\zeta)\big)^\dagger (S_2^L)^\dagger\\&=S_1^L M^L (S_2^L)^\dagger\\&=	H^L& \text{use \eqref{eq:gauss}.}
	\end{align*}
Now for \eqref{eq:biorthogonalityR}
			\begin{align*}
			\oint_\T \big((P^R_2(\zeta) )^\top\big) ^\dagger \d \mu(\zeta)
			\big(P^R_1(\zeta)\big)^\top&=S^R_2 \oint_\T \big((\chi(\zeta)) ^\top\big)^\dagger \d \mu(\zeta)(\chi(\zeta))^\top (S_1^R)^\dagger\\&=S^R_2 M^R (S_1^R)^\dagger\\&=	H^R & \text{use \eqref{eq:gauss}.}
			\end{align*}
\end{proof}
\begin{proposition}
	The components $P^L_{1,n}$, $P^R_{1,n}$, $P^L_{2,n}$ and $P^R_{2.n}$  of the semi-infinite vectors $P^L_1$, $P^R_1$, $P^L_2$ and $P^R_2$
\begin{enumerate}
	\item Satisfy the biorthogonal relations
		\begin{align}\label{eq:norms}
		\oint_\T P^L_{1,n}(\zeta) \d \mu(\zeta)
		\big(P^L_{2,m}(\zeta)\big)^\dagger&=	H^L_n\delta_{n,m}, &
		\oint_\T \big(P^R_{2,m}(\zeta)\big)^\dagger \d \mu(\zeta)
		P^R_{1,n}(\zeta)&=	H_n^R\delta_{n,m}.
		\end{align}
	\item 	The components $P^L_{1,n}$, $P^R_{1,n}$, $P^L_{2,n}$ and $P^R_{2.n}$  of the semi-infinite vectors $P^L_1$, $P^R_1$, $P^L_2$ and $P^R_2$ are the Szegő matrix polynomials of Definition \ref{def:szego}.
	\item The  Szegő polynomials and its reciprocals satisfy
	\begin{align}\label{eq:norms2}
	\oint_\T P^L_{1,n}(\zeta) \d \mu(\zeta)\bar\zeta^m \tilde P^L_{2,m}(\zeta)
&=	H^L_n\delta_{n,m}, &
	\oint_\T \bar\zeta^m \tilde P^L_{2,m}(\zeta) \d \mu(\zeta)
	P^R_{1,n}(\zeta)&=	H_n^R\delta_{n,m}.
	\end{align}
\item The quasi-tau functions can be expressed as
	\begin{align}\label{eq:quasitau1l}
	H^L_n&=\oint_\T P^L_{1,n}(\zeta) \d \mu(\zeta)
	\bar\zeta^n\\\label{eq:quasitau2L}
		&=\oint_\T \d \mu(\zeta)\tilde P^L_{2,n}(\zeta), \\\label{eq:quasitau1R}
		H_n^R&=	\oint_\T \bar\zeta^n\d \mu(\zeta)
		P^R_{1,n}(\zeta)\\\label{eq:quasitau2R}
	&=	\oint_\T \tilde  P^R_{2,n}(\zeta)\d \mu(\zeta).		
	\end{align}
\end{enumerate}
\end{proposition}
\begin{proof}\begin{enumerate}
		\item Elementary.
		\item Observe that \eqref{eq:norms} implies  the orthogonal relations  \eqref{eq:ortogonal}, \eqref{eq:ortogonall}, \eqref{eq:2ortogonal} and \eqref{eq:2ortogonall}.
		\item Use Definition \ref{def:reciprocal}.
\item It follows from \eqref{eq:norms} and \eqref{eq:norms2}.
	\end{enumerate}\end{proof}
\begin{definition}
	The matrices $H^L_N$ and $H^R_n$ are called quasi-tau matrices.
		\end{definition}

\begin{proposition}\label{pro:Hquasideterminant}
The quasi-tau matrices must be not singular and have the following last quasi-determinantal expressions
	\begin{align*}
H^L_n&=\Theta_*\PARENS{\begin{matrix}
\hat\mu(0)& \hat\mu(-1)&\hat\mu(-2)&\cdots&\hat\mu(-n+1)\\
\hat\mu(1) & \hat\mu(0) &\hat\mu(-1)& &\hat\mu(-n+2)\\
\hat\mu(2) &\hat\mu(1)& \hat\mu(0) &\ddots &\hat\mu(-n+3)\\
\vdots& &\ddots &\ddots\\
\hat\mu(n-1) & \hat\mu(n-2) &\hat\mu(n-3)&\cdots &\hat\mu(0)
\end{matrix}},\\
H^R_n&=\Theta_*\PARENS{\begin{matrix}
\hat\mu(0)& \hat\mu(1)&\hat\mu(2)&\cdots&\hat\mu(n-1)\\
\hat\mu(-1) & \hat\mu(0) &\hat\mu(1)& &\hat\mu(n-2)\\
\hat\mu(-2) &\hat\mu(-1)& \hat\mu(0) &\ddots &\hat\mu(n-3)\\
\vdots& &\ddots &\ddots\\
\hat\mu(-n+1) & \hat\mu(-n+2) &\hat\mu(-n+3)&\cdots &\hat\mu(0)
\end{matrix}}.
\end{align*}
\end{proposition}
\begin{proof}
It follows from the Gaussian factorization, see \cite{olver,AM,Cafasso}.
\end{proof}
\subsection{Symmetry  properties}

\begin{proposition}\label{pro:commutativity0}
	Assume that there is a matrix $C\in\C^{N\times N}$ such that 
	\begin{align*}
	[C,\mu]&=0.
	\end{align*}
	Then, the matrix $C$ commute with all the moments $\hat\mu(n)$, i.e.,
	\begin{align*}
	[C,\hat\mu(n)]&=0,& n\in\mathbb Z,
	\end{align*}
	and with the Szeg\H{o} matrix polynomials, reciprocals, corresponding Verblunsky coefficients and quasi-tau matrices
	\begin{gather*}
	[C,P^L_{1,n}(z)]=[C,P^R_{1,n}(z)]=\big[C, \tilde P^L_{2,n}(z)\big]=\big[C, \tilde P^R_{2,n}(z)\big]=0,\\
	[C,\alpha^L_{1,n}]=[C,\alpha^R_{1,n}]=[C, \big(\alpha^L_{2,n}\big)^\dagger]=[C, \big(\alpha^R_{2,n}\big)^\dagger]=0,\\
	[C,H^L_n]=[C,H^R_n].
	\end{gather*}
	for all $n\in\{0,1,2,\dots\}$ and $z\in\mathbb C$.
\end{proposition}
\begin{proof}
	It is obvious that $C$ commutes with the moments. Then, as Proposition \ref{pro:quasideterminants0}  ensures that  Szeg\H{o} matrix polynomials $,P^L_{1,n}(z)$ and $P^R_{1,n}(z)$ are rational functions of the moments, and Proposition \ref{pro:quasideterminants} ensures the same for the reciprocal polynomials $\tilde P^L_{2,n}(z)$ and $\tilde P^R_{2,n}(z)$, we see that $C$ commutes with then. The result for the Verblunsky coefficients follow immediately. The property regarding the quasi-tau matrices is deduced from Proposition \ref{pro:Hquasideterminant}.
\end{proof}

\begin{proposition}\label{pro:commutativity}
	Suppose that for each pair of Borel sets $A,B\subset\T$ we have
	\begin{align*}
	[\mu(A),\mu(B)]=0.
	\end{align*}
	Then, the set of the moments  $\C\{\hat\mu(n)\}_{n\in\mathbb Z}$ is an Abelian algebra
	\begin{align*}
	[\hat\mu(i),\hat\mu(j)]&=0,& i,j\in\mathbb Z.
	\end{align*}
	Moreover, the family of matrix polynomials 
	\begin{align*}
	\big\{
	P^L_{1,n}(z_1), P^R_{1,n}(z_2),\tilde P^L_{2,n}(z_3), \tilde P^R_{2,n}(z_4)
	\big\}_{\substack{n\in\{0,1,2,\dots\}\\ z_1,z_2,z_3,z_4\in\mathbb C}}
	\end{align*}
	is Abelian.
	Analogously, the set of Verblunsky matrices and quasi-tau matrices
	\begin{align*}
	\big\{
	\alpha^L_{1,n}, \alpha^R_{1,n} \big(\alpha^L_{2,n}\big)^\dagger,  \big(\alpha^R_{2,n}\big)^\dagger, H^L_n,H^R_n
	\big\}_{n=0}^\infty
	\end{align*}
	is Abelian.
\end{proposition}
\begin{proof}
	For simplicity we give the proof for the absolutely continuous case, i.e.,
	we assume that $\mu=w\d m$ with
	\begin{align*}
	[w(u),w(v)]&=0, &\forall u,v\in\operatorname{supp} (w(z))\subset \T.
	\end{align*}
	First, we see that the moments commute among then. Indeed, according to Definition \ref{def:moments} 
	\begin{align*}
	[\hat\mu(i),\hat\mu(j)]&=\bigg[\oint_{\mathbb T} w(u)  u^{-i}\frac{\d u}{2\pi\operatorname{i}u},\oint_{\mathbb T} w(v)  v^{-j}\frac{\d v}{2\pi\operatorname{i}v}\bigg]\\&=-\oiint_{\mathbb T^2}[w(u),w(v)] u^{-i}v^{-j}\frac{d u\d v}{4 \pi^2uv}\\&=0.
	\end{align*} 
	As in the previous proof Proposition \ref{pro:quasideterminants0}  ensures that  $,P^L_{1,n}(z)$ and $P^R_{1,n}(z)$ are rational functions of the moments, and Proposition \ref{pro:quasideterminants}  the same for the reciprocal polynomials $\tilde P^L_{2,n}(z)$ and $\tilde P^R_{2,n}(z)$, and the commutativity property follows. The result for the Verblunsky coefficients follow immediately. From Proposition \ref{pro:Hquasideterminant} we deduce the commutativity with the quasi-tau matrices.
\end{proof}

\section{The Riemann-Hilbert problem}\label{S:RHP}
In this section, following the seminal paper \cite{FIK}  we find a general Riemann-Hilbert problem whose solution characterizes matrix Szegő polynomials. This  problem constitute the keystone for the finding of the matrix discrete Painlevé II system.
\subsection{Cauchy transforms}
We began we some facts regarding Cauchy transforms for matrices of measures on the unit circle, we  follow  the excellent monograph \cite{Cima}.
\begin{definition}
	Given a finite Borel matrix of measures $\mu$ on the unit circle $\T$ its Cauchy transform is defined by
\begin{align*}
(C\mu)(z):=\oint_\T \frac{\d \mu(\zeta)}{1-\bar{\zeta}z},
\end{align*}
where the integration along $\T$ is taken counterclockwise.  We denote by $(C\mu)_+$ the restriction to the unit disk $\mathbb D$
and by $(C\mu)_-$ the restriction to the annulus $\bar{\mathbb D}$.
\end{definition}

Let us review some properties of the Cauchy transform that are relevant for this paper.
\paragraph{\textbf{Properties of the Cauchy transform} }
\begin{enumerate}
	\item The Cauchy transform $(C\mu)(z)$ is analytic on $\bar \C\setminus\T$ , where $\bar\C=\C\cup\{\infty\}$, with analytic continuation across the complement of $\operatorname{supp}(\mu)$.
		\item The restriction of the Cauchy transform $(C\mu)_-\in\cap_{0<p<1} H^p(\bar{\mathbb D})$, where  $H^p(\bar{\mathbb D})$ is the Hardy space of the exterior of the unit circle.\footnote{We understand that  $f(z)\in H^p(\bar{\mathbb D})$ iff $f(z^{-1})\in H^p(\mathbb D)$. The Hardy space $H^p(\mathbb D)$ contains all analytic functions $f(z)$ on $\mathbb D$ such that $\sup _{0<r<1}\big(\int _\T 
			|f(r\zeta)|^{p}\d m(\zeta) \big)^{\frac {1}{p}}<\infty$ .} 
	\item We have $\| (C\mu)_+\|_{H^p(\bar {\mathbb D})}=O\big((1-p)^{-1}\big)$ for $p\to 1^-$, and consequently has exterior non-tangential limits almost everywhere on $\T$.
\item The Taylor series of  $(C\mu)_+$ at $z=0$ is
\begin{align*}
(C\mu)_+(z)&=\sum_{n=0}^\infty \hat\mu(n) z^n, &\forall z\in\mathbb D,
\end{align*}
and of
$(C\mu)_-$ about the point of infinity is
\begin{align*}
(C\mu)_-&=-\sum_{n=1}^\infty	\hat{\mu}(-n) z^{-n}, & \forall z\in\bar{\mathbb D}.
	\end{align*}
	\item For almost every $\zeta\in\T$ the limits
	\begin{align*}
	(C\mu)_+(\zeta)&=\lim_{r\to 1^-}(C\mu)_+(r\zeta), &(C\mu)_-(\zeta)&=\lim_{r\to 1^-}(C\mu)_+(\zeta/r),
	\end{align*}
	exist.
	
\item The  Fatou's jump
\begin{align*}
(C\mu)_+(\zeta)-(C\mu)_-(\zeta)=\frac{\d\mu}{\d m}(\zeta)
\end{align*}
holds $m$-almost every\footnote{With respect to the Lebesgue measure $m(\zeta)$.}, $\zeta\in\T$.
	\item Privalov's theorem: for $m$-almost every $\zeta\in\T$ we have
	\begin{align*}
	\operatorname{P V}\oint_\T\frac{\d\mu(\xi)}{1-\bar\xi \zeta}=\frac{1}{2}((C\mu)_+(\zeta)+(C\mu)_-(\zeta)).
	\end{align*}
	Here the principal value at $\zeta\in\T$ is defined as
	\begin{align*}
	\operatorname{P V}\oint_\T\frac{\d\mu(\xi)}{1-\bar\xi \zeta}:=\lim_{\epsilon\to 0^+}\int_{|\xi-\zeta|>\epsilon}\frac{\d\mu(\xi)}{1-\bar\xi \zeta},
	\end{align*}
	whenever the limit exists.
\end{enumerate}

The results of Privalov and Fatou are extensions of the Sokhotski--Plemelj formulas, first discussed by Sokhotski in 1873 for $\d\mu(\zeta)=w(\zeta)\d m(\zeta)$ where $w(\zeta)$ is a Lipschitz function and then refined by Plemelj in 1908. Privalov results goes back to 1919.
We are now prepared to introduce the  Cauchy transform according to the matrix of measures $\mu$ of the Szegő polynomials and their reciprocals.

\begin{definition}[Cauchy transforms]
We consider the following matrix Cauchy transforms of the Szegő matrix polynomials
	\begin{align*}
		Q^{L}_{1,n}(z)&:=\oint_{\mathbb{T}}\bar\zeta^{n}P^{L}_{1,n}(\zeta)\frac{\d \mu(\zeta)}{1-\bar \zeta z}, &
		Q^L_{2,n}(z)&:=\oint_{\mathbb{T}}\frac{\d\mu(\zeta)}{1-\bar \zeta z}\tilde P^L_{2,n}(\zeta)\bar\zeta^{n+1}, \\
		Q^{R}_{1,n}(z)&:=\oint_{\mathbb{T}}\frac{\d\mu(\zeta)}{1-\bar \zeta z}P^{R}_{1,n}(\zeta)\bar\zeta^{n},  &
		Q^R_{2,n}(z)&:=\oint_{\mathbb{T}}\bar\zeta^{n+1} \tilde P^R _{2,n}(\zeta)
		\frac{\d\mu(\zeta)}{1-\bar \zeta z}.
	\end{align*}
\end{definition}
Observe that
\begin{align}\label{eq:fourierQ0}
	Q_{1,0}^L(z)=Q_{1,0}^R(z)&=\oint_{\mathbb{T}}\frac{\d\mu(\zeta)}{1-\bar\zeta z}=\ccases{
	\sum\limits_{n=0}^\infty \hat \mu(n)z^n, & z\in\mathbb D,\\[5pt]
		-\sum\limits_{n=1}^\infty \hat \mu(-n)z^{-n}, & z\in\bar{\mathbb D}.
	}
			\end{align}
in terms of the Fourier coefficients.\footnote{Notice that for $m$-almost every $\zeta\in\T$ the Fatou's jump take place
			\begin{align*}
			\frac{\d\mu}{\d m}(\zeta)=\sum\limits_{n=-\infty}^\infty \hat \mu(n)\zeta^n,
			\end{align*}
			which is the Fourier series of the measure, and the Privalov's principal value holds:
			\begin{align*}
			\operatorname{P V}	\oint_{\mathbb{T}}\frac{\d\mu(\zeta)}{1-\bar\zeta z}=\frac{	1}{2}\Big(\sum\limits_{n=0}^\infty \hat \mu(n)\zeta^{n}	-\sum\limits_{n=1}^\infty \hat \mu(-n)\zeta^{-n}\Big).
			\end{align*}}
			Now, observing that
			\begin{align*}
			\frac{\bar{\zeta}}{1-\bar\zeta z}=z^{-1}\Big(		\frac{1}{1-\bar\zeta z}-1\Big)
			\end{align*}
we get			
	\begin{align*}
	Q_{2,0}^L(z)=Q_{2,0}^R(z)&=\oint_{\mathbb{T}}\frac{\d\mu(\zeta)}{1-\bar\zeta z}\bar\zeta\\
	&=z^{-1}\Big(\oint_{\mathbb{T}}\frac{\d\mu(\zeta)}{1-\bar\zeta z}-\mu(\mathbb{T})\Big).
\end{align*}
\begin{definition}
	We will use the following Fourier coefficients
	\begin{align*}
	\hat P^{L}_{1,n}(j)&:=\oint_{\mathbb{T}}
	P^{L}_{1,n}(\zeta)\d\mu(\zeta)\bar\zeta^{-j}, & \hat P^{L}_{2,n}(j)&:=\oint_{\mathbb{T}}
\bar \zeta^{-j+1}	\d\mu(\zeta) \tilde P^{L}_{2,n}(\zeta), \\\hat P^{R}_{1,n}(j)&:=\oint_{\mathbb{T}}
	\bar\zeta^{-j}\d\mu(\zeta)P^{R}_{1,n}(\zeta), &  \hat P^{R}_{2,n}(j)&:=\oint_{\mathbb{T}}
	\tilde P^{R}_{2,n}(\zeta)\d\mu(\zeta)\bar\zeta^{-j+1},
	\end{align*}	
	for $j\in\mathbb Z$.
\end{definition}

\begin{proposition}
\begin{enumerate}
	\item 	We have the following cancellations
	\begin{align*}
	&\hat P^{L}_{1,n}(j)= \hat P^{L}_{2,n}(j)=\hat P^{R}_{1,n}(j)=\hat P^{R}_{2,n}(j)=0, & &j\in\{0,-1,,\dots,-n+1\}.
	\end{align*}
	\item The quasi-tau matrices are
		\begin{align*}
		H^L_n&=\hat P^{L}_{1,n}(-n)=\hat P^{L}_{2,n}(1),&
		H_n^R&=	\hat P^{R}_{1,n}(-n)=\hat P^{R}_{2,n}(1).		
		\end{align*}
\end{enumerate}
\end{proposition}
\begin{proof}
\begin{enumerate}
	\item 	Observe that \eqref{eq:ortogonal}, \eqref{eq:ortogonall}, \eqref{eq:r2ortogonal} and \eqref{eq:r2ortogonall}  are equivalent to these cancellations.
\item A simple consequence of\eqref{eq:quasitau1R},\eqref{eq:quasitau1l}, \eqref{eq:quasitau2L} and \eqref{eq:quasitau2R}.
\end{enumerate}
\end{proof}

\begin{proposition}[Power series of the Cauchy transforms]\label{pro:asymptotics}
	The following  Taylor series  of the Cauchy transforms
		\begin{align*}
		Q^{L}_{1,n}(z)&=H^L_n+\sum_{j=1}^\infty \hat P^{L}_{1,n}(-n-j)z^{j},& Q^{L}_{2,n}(z)&= \sum_{j=0}^\infty \hat P^{L}_{2,n}(-n-j)z^{j},\\
		Q^{R}_{1,n}(z)&=H^R_n+\sum_{j=1}^\infty \hat P^{R}_{1,n}(-n-j)z^{j},& Q^{R}_{2,n}(z)&=\sum_{j=0}^\infty \hat P^{R}_{2,n}(-n-j)z^{j},
		\end{align*}
		converge on $\mathbb D$ .
The following Taylor series about infinity
	\begin{align*}
	Q^{L}_{1,n}(z)&=- \sum_{j=1}^\infty \hat P^{L}_{1,n}(j)z^{-n-j},& Q^{L}_{2,n}(z)&=-H^L_n z^{-n-1}-\sum_{j=2}^\infty \hat P^{L}_{2,n}(j)z^{-n-j},\\
	Q^{R}_{1,n}(z)&= -\sum_{j=1}^\infty\hat P^{R}_{1,n}(j)z^{-n-j},& Q^{R}_{2,n}(z)&= -H^R_n z^{-n-1}-\sum_{j=2}^\infty \hat P^{R}_{2,n}(j)z^{-n-j},
	\end{align*}
	converge in $\bar{\mathbb D}$.
\end{proposition}
\begin{proof}
First we recall that
		\begin{align*}
		\bar\zeta^n	\frac{1}{1-\bar \zeta z}&=\bar\zeta^n+z\bar\zeta^{n+1}+z^{2}\bar\zeta^{n+2}+z^{3}\bar\zeta^{n+3}-\cdots, & |z|&<1, & \zeta&\in\mathbb T,
		\end{align*}
		uniformly, from where we get the Taylor series in the unit disk.
	
For the Taylor expansions at infinity we observe that
	\begin{align*}
\bar\zeta^n	\frac{1}{1-\bar \zeta z}&=-\bar\zeta^{n-1} z^{-1}	\frac{1}{1-z^{-1} \bar\zeta^{-1}}
\\
		&=-z^{-1}\bar\zeta^{n-1}-z^{-2}\bar\zeta^{n-2}-z^{-3}\bar\zeta^{n-3}-\cdots, & |z|&>1, & \zeta&\in\mathbb T,
	\end{align*}
	uniformly.
	Then,
	\begin{align*}
		Q^{L}_{1,n}(z)&:=-\sum_{j=1}^{\infty}z^{-j}\oint_{\mathbb{T}}
		P^{L}_{1,n}(\zeta)\d\mu(\zeta)\bar\zeta^{n-j}, & Q^{L}_{2,n}(z)&:=-\sum_{j=1}^{\infty}z^{-j}\oint_{\mathbb{T}}
\bar\zeta^{n-j+1}	\d\mu(\zeta) \tilde P^{L}_{2,n}(\zeta), \\Q^{R}_{1,n}(z)&:=-\sum_{j=1}^{\infty}z^{-j}\oint_{\mathbb{T}}
	\bar\zeta^{n-j}	\d\mu(\zeta)P^{R}_{1,n}(\zeta), &  Q^{R}_{2,n}(z)&:=-\sum_{j=1}^{\infty}z^{-j}\oint_{\mathbb{T}}
		\tilde P^{R}_{2,n}(\zeta)\d\mu(\zeta)\bar \zeta^{n-j+1},
	\end{align*}
	and consequently, recalling the orthogonal relations  \eqref{eq:ortogonal}, \eqref{eq:ortogonall}, \eqref{eq:r2ortogonal} and \eqref{eq:r2ortogonall}, we get
	\begin{align*}
	Q^{L}_{1,n}(z)&:=-\sum_{j=n+1}^{\infty}z^{-j}\oint_{\mathbb{T}}
	P^{L}_{1,n}(\zeta)\d\mu(\zeta)\bar\zeta^{n-j}, & Q^{L}_{2,n}(z)&:=-\sum_{j=n+1}^{\infty}z^{-j}\oint_{\mathbb{T}}
	\bar\zeta^{n-j+1}	\d\mu(\zeta) \tilde P^{L}_{2,n}(\zeta), \\Q^{R}_{1,n}(z)&:=-\sum_{j=n+1}^{\infty}z^{-j}\oint_{\mathbb{T}}
	\bar\zeta^{n-j}	\d\mu(\zeta)P^{R}_{1,n}(\zeta), &  Q^{R}_{2,n}(z)&:=-\sum_{j=n+1}^{\infty}z^{-j}\oint_{\mathbb{T}}
	\tilde P^{R}_{2,n}(\zeta)\d\mu(\zeta)\bar \zeta^{n-j+1},
	\end{align*}
	and the result follows.
\end{proof}

\begin{proposition}[Fatou's jump formul{\ae}]\label{pro:plemej}
	The Cauchy transforms have the following jumps $m$-almost every  $\zeta\in\mathbb T$
	\begin{align*}
	(Q^{L}_{1,n})_+(\zeta)-	(Q^{L}_{1,n})_-(\zeta)&=\bar\zeta^{n}P^L_{1,n}(\zeta)\frac{\d\mu}{\d m}(\zeta),&
		(Q^{L}_{2,n})_+(\zeta)-	(Q^{L}_{2,n})_-(\zeta)&=\frac{\d\mu}{\d m}(\zeta)\tilde P^L_{2,n}(\zeta)\bar{\zeta}^{n+1},\\
		(Q^{R}_{1,n})_+(\zeta)-	(Q^{R}_{1,n})_-(\zeta)&=\frac{\d\mu}{\d m}(\zeta)P^R_{1,n}(\zeta)\bar\zeta^{n},&
		(Q^{R}_{2,n})_+(\zeta)-	(Q^{R}_{2,n})_-(\zeta)&=\bar\zeta^{n+1}\tilde P^R_{2,n}(\zeta)\frac{\d\mu}{\d m}(\zeta).
	\end{align*}
\end{proposition}

\subsection{Riemann--Hilbert problems}
Inspired by the seminal paper \cite{FIK} and  following  the implementation for scalar Szegő polynomials given in  \cite{baik}, see also \cite{mf1,mf2}, here we propose  a block type Riemann--Hilbert problem in the unit circle and find its solution in terms of matrix Szegő polynomials. We start with a general situation, the weak problem,  with integrable jump functions and then move to the more classical scenario, the strong problem, with Hölder jump functions.

\begin{definition}[Weak and strong Riemann--Hilbert problems]\label{RHP}
The Riemann--Hilbert problem (RHP) consists in the finding of a
$2N \times 2N$ matrix function, $Y_n(z)\in \mathbb{C}^{2N\times 2N}$,  for a given $n\in\{0,1,2,\dots\}$,
 such  that
\begin{enumerate}
\item Is analytic in $\mathbb{C} \setminus \mathbb{T}$.
\item  Satisfies the jump condition at $\zeta\in\mathbb T$
\begin{align*}
( Y_{n})_+(\zeta)=(Y_{n})_-(\zeta) \PARENS{\begin{matrix} I_N &  w(\zeta)\bar\zeta^{-n}
\\ 0_N &  I_N \end{matrix}},
\end{align*}
where  $w=(w_{i,j})$ is a matrix of   weights. In the weak RHP we have $w_{i,j}\in L^1$ and request the jump to hold $m$-almost every $\zeta\in\T$. In the strong RHP the weights $w_{i,j}$ are Hölder\footnote{A function $f:\T\to\C$ is said Hölder if $|f(\zeta_1)-f(\zeta_2)|\leq M|\zeta_1-\zeta_2|^\alpha$ for $M>0$ and $0<\alpha\leq 1$, see \cite{Muskhelishvili}.}  and the jump must hold for all $\zeta\in\T$.
\item About infinity has the following asymptotic
\begin{align*}
Y_n(z) &= (I_{2N}+O(z^{-1})) \PARENS{\begin{matrix} I_Nz^n
& 0_N \\ 0_N &  I_Nz^{-n} \end{matrix}}, &|z| &\rightarrow \infty.
\end{align*}
\end{enumerate}
\end{definition}

We have the following result
\begin{theorem}
A solution of the Riemann-Hilbert problem stated in Definition \ref{RHP} is given by the block matrix function
\begin{align}
Y_n(z)&:=\PARENS{\begin{matrix} P^{L}_{1,n}(z) & Q^{L}_{1,n}(z)
\\ -\big(H^R_{n-1}\big)^{-1}\tilde P^R_{2,n-1}(z) &
-\big(H^R_{n-1}\big)^{-1} Q^R_{2,n-1}(z)\end{matrix}},&  n&\in\{1,2,\dots\},
\label{eq:ypsilon}
\end{align}
and
\begin{align}
Y_0(z)&:=\PARENS{\begin{matrix} {I_N} & Q^{L}_{1,0}(z) \\ 0_N &
{I_N}\end{matrix}}.
\label{eq:ypsilon0}
\end{align}
For the weak RH problem  the measure $\mu$, involved in the Szegő polynomials and the Cauchy transform, is taken  such that its Radon--Nikodym derivative is the matrix of weights $w=\frac{\d\mu}{\d m}$.
For the strong RH problem the measure is $\d\mu=w\d m$  and the solution is unique.
\end{theorem}

\begin{proof}
Let us show that the matrix given in \eqref{eq:ypsilon} is a solution of the RHP. We do it in three steps:
\begin{enumerate}
	\item The polynomials $P_{1,n}^L(z)$ and $\tilde P^R_{2,n-1}(z)$ are analytic functions in $\mathbb C$ and the Cauchy transforms $Q^L_{1,n}(z)$ and $ Q_{2,n-1}^R(z)$ are analytic in $\mathbb C\setminus \mathbb T$. Therefore, $Y_n(z)$ is analytic in $\mathbb C\setminus \mathbb T$.
	\item From Proposition \ref{pro:asymptotics} we can deduce the following asymptotics
	\begin{align*}
	Y_n(z)=\PARENS{\begin{matrix}
	I_Nz^n+O(z^{n-1})  & O(z^{-n-1})\\
	  O(z^{n-1}) &I_Nz^{-n}+O(z^{-n-1})
	\end{matrix}}=(I_{2N}+O(z^{-1}))\PARENS{\begin{matrix}
	I_Nz^n & 0_N\\
	0_N & I_Nz^{-n}
	\end{matrix}}.
	\end{align*}
	\item For the weak RHP, Proposition \ref{pro:plemej} gives
	\begin{align*}
\big(Y_n\big)_+(\zeta)-\big(Y_n\big)_-(\zeta)&=	\PARENS{\begin{matrix} \big(P^{L}_{1,n}\big)_+(\zeta)-\big(P^{L}_{1,n}\big)_- (\zeta)& \big(Q^{L}_{1,n}\big)_+(\zeta)-\big(Q^{L}_{1,n}\big)_- (\zeta)
	\\ \big(H^R_{n-1}\big)^{-1}\Big(\big(P^{R}_{2,n-1}\big)_+(\zeta)-\big(P^{R}_{2,n-1}\big)_-(\zeta) \Big) &
	-\big(H^R_{n-1}\big)^{-1}\Big(\big(Q^{R}_{2,n-1}\big)_+(\zeta)-\big(Q^{R}_{2,n-1}\big)_- (\zeta)\Big) \end{matrix}}
	\\&=\PARENS{\begin{matrix}
	0_N&P^L_{1,n}(\zeta) w(\zeta)\bar \zeta^{n}\\
	0_N & \big(H^R_{n-1}\big)^{-1}\tilde P^R_{2,n-1}(\zeta) w(\zeta) \bar\zeta^{n}
	\end{matrix}}\\
	&= \big(Y_n\big)_-(\zeta) \PARENS{\begin{matrix}
	0_N& w(\zeta)\bar\zeta^{n}\\
	0_N &0_N
	\end{matrix}},
	\end{align*}
	for $m$-almost every $\zeta\in\T$. When we consider the strong RH situation then the matrix of weights is Hölder and the Sokhotski--Plemelj holds for every $\zeta\in\T$, see \cite{Gakhov}.
\end{enumerate}
Once we have proven the existence of a solution to the Riemann--Hilbert problem, let us show its uniqueness for the strong RHP.
We notice that $\det Y_n(z)$ is an analytic function in $\mathbb C\setminus\mathbb T$, and has jump  at $\zeta\in \mathbb T$, indeed
\begin{align*}
\det (Y_n)_+(\zeta)&=\det (Y_n)_-(\zeta)\begin{vmatrix} I_N &  w(\zeta)\bar\zeta^{n}
\\ 0_N &  I_N \end{vmatrix}\\
&=\det (Y_n)_-(\zeta).
\end{align*}
Therefore, $\det Y_n(z)$ is analytic in $\mathbb C$, and the Liouville theorem ensures that is constant,  but $\det Y_n\to 1$ as $|z|\to\infty$, consequently,  we have  that $\det Y_n(z)=1$.\footnote{Incidentally, we observe that in the weak RHP we only know that is analytic everywhere but for a Borel set $B\subset \T$ of zero Lebesgue measure, $m(B)=0$, and we can not apply the Liouville theorem.} Thus, $(Y_n(z))^{-1}$ is a matrix of analytic functions for all $z\in\mathbb C\setminus\mathbb T$. Given two solutions $\hat Y_n,Y_n$ to the Riemann--Hilbert problem, the block matrix $\hat Y_n \big(Y_n\big)^{-1}$ has no jump at $\zeta\in\mathbb T$, and therefore is analytic in the whole complex plane, thus is a constant matrix, and the asymptotic implies  that this constant is the identity.
\end{proof}

As we have seen, the non uniqueness in the weak case  is related to the weak Fatou jump corollary, that holds only almost everywhere in the circle, allowing therefore for singularities.  Moreover, if the weak situation the  matrix of $L^1(\T,\mu)$ weights only fixes the absolutely continuous part of the measure $\mu_a=w\d m\ll m$, and we have the freedom of adding any singular measure $\mu_s\perp m$ as $D\mu_s=0$.  Therefore, given  $Y_n$ constructed for $w\d m$   we may consider $\tilde Y_n$ associated with $w\d m+\mu_s$ and we will have another solution to the weak RHP. For the strong RHP we refer the reader to \cite{Gakhov}, observe that in this situation the Lebesgue integration coincides with the Riemann integration.

From hereon we consider only the strong RH problem.
 \begin{definition}
 	We define the block matrix
 	\begin{align*}
 	X_n(z)&:=\PARENS{\begin{matrix} P^{L}_{1,n}(z) & Q^{L}_{1,n}(z)
 	\\ -\big(H^R_{n-1}\big)^{-1}\tilde P^R_{2,n-1}(z) &
 	-\big(H^R_{n-1}\big)^{-1} Q^R_{2,n-1}(z)\end{matrix}}\PARENS{\begin{matrix} I_N z^{-n}& 0_N \\ 0_N &
 	I_Nz^{n} \end{matrix}}, & n&\in\{1,2,\dots\},
 	\end{align*}
 	and
 	\begin{align*}
 	X_0(z)&:=\PARENS{\begin{matrix} {I_N} & Q^{L}_{1,0}(z) \\ 0_N &
 	{I_N}\end{matrix}}.
 	\end{align*}
 \end{definition}
 Observe that $X_0(z)=Y_0(z)$.

 \begin{proposition} \label{def:Xn}
 	For each $n\in\{0,1,2,\dots\}$,  $X_{n}(z)$ is the unique matrix function such that
 	\begin{enumerate}
 		\item $X_n(z)\begin{psmallmatrix}
 		I_N z^n & 0_N\\0_N& I_N z^{-n}
 		\end{psmallmatrix}$ is analytic in $\mathbb{C} \setminus \mathbb{T}$.
 		\item Satisfies the jump condition for $\zeta\in\mathbb T$
 		\begin{align*}
 		\big(X_n\big)_+(\zeta)=	\big(X_n\big)_-(\zeta)\PARENS{\begin{matrix}
 		I_N & w(\zeta)\bar\zeta^{-n}\\
 		0_N & I_N
 		\end{matrix}}.
 		\end{align*}
 		\item Asymptotically behaves as
 		$X_n(z) = I_{2N}+O(z^{-1})$ for $|z| \rightarrow \infty $.
 	\end{enumerate}
 \end{proposition}

 \begin{proposition}[Series for $X_n$]\label{pro:Xlaurent}
 	\begin{enumerate}
 		\item The following Laurent series
 			\begin{multline*}
 			X_n(z)=\PARENS{\begin{matrix}
 		\alpha_{1,n} & 0_N\\
 		-\big(H^R_{n-1}\big)^{-1}z^{-n-1}& 0_N
 			\end{matrix}}z^{-n}+\cdots+\PARENS{\begin{matrix}
 		P^L_{1,n,n-1}& 0_N\\
 			-\big(H^R_{n-1}\big)^{-1}\big(\alpha^R_{2,n-1}\big)^\dagger &0_N
 			\end{matrix}}z^{-1}+\PARENS{\begin{matrix}
 			I_N & 0_N\\
 			0_N& 0_N
 			\end{matrix}}\\+\PARENS{\begin{matrix}
 		0_N& \hat P^L_{1,n}(-n)\\
 		0_N& -\big(H^R_{n-1}\big)^{-1}\hat P^R_{2,n}(-n-1)
 			\end{matrix}}z^n+\PARENS{\begin{matrix}
 		0_N& \hat P^L_{1,n}(-n-1)\\
 	0_N& -\big(H^R_{n-1}\big)^{-1}\hat P^R_{2,n}(-n-2)
 			\end{matrix}}z^{n+1}+\cdots
 			\end{multline*}
 			converges in the annulus $\mathbb{D}\setminus\{0\}$.
 		\item  The following Taylor series about infinity	
 	\begin{multline*}
 	X_n(z)=I_{2N}+	\PARENS{\begin{matrix}
 P^L_{1,n,n-1} & -\hat P^L_{1,n}(1)\\
 	-\big(H^R_{n-1}\big)^{-1}\big(\alpha^R_{2,n-1}\big)^\dagger & \big(H^R_{n-1}\big)^{-1}\hat P^R_{2,n-1}(2)
 	\end{matrix}}z^{-1}
 	+\cdots \\+	
 	\PARENS{\begin{matrix}
 \alpha_{1,n}& -\hat P^L_{1,n}(n)\\
 	-\big(H^R_{n-1}\big)^{-1}&\big(H^R_{n-1}\big)^{-1}\hat P^R_{2,n-1}(n+1)
 	\end{matrix}}z^{-n}
 +
 		\PARENS{\begin{matrix}
 	0_N& -\hat P^L_{1,n}(n+1)\\
 	0_N& \big(H^R_{n-1}\big)^{-1}\hat P^R_{2,n-1}(n+2)
 		\end{matrix}}z^{-n-1}+\cdots
 		\end{multline*}
 	 converges at $\bar{\mathbb{D}}$.
 	\end{enumerate}
 \end{proposition}
 We see that the Laurent expansion at the origin is rather peculiar in its block structure.
 \begin{enumerate}
 	\item
 The matrix $\big[X_n(z)\big]_{\text{principal}}=(X_n(z)-I_{2N})\begin{psmallmatrix}I_N & 0_N\\ 0_N &0 _N \end{psmallmatrix}$
 is the principal part of the Laurent series of $X_N$ at $z=0$. Thus, in the principal part the second block column cancels:
 \begin{align*}
\big[X_n(z)\big]_{\text{principal}}=\begin{pmatrix*}z^{-n} P^L_{1,n} -I_N& 0_N\\
-z^{-n}\big(H^R_{n-1}\big)^{-1} \tilde P^R_{2,n-1}&0_N
\end{pmatrix*}.
 \end{align*}

 \item The regular part of the Laurent series is $\big[X_n(z)\big]_{\text{regular}}=\begin{psmallmatrix}I_N & 0_N\\ 0_N &0 _N \end{psmallmatrix}+X_n(z)\begin{psmallmatrix}0_N & 0_N\\ 0_N &I _N \end{psmallmatrix}$. Hence, the regular part
 $\big[X_n(z)\big]_{\text{regular}}-\begin{psmallmatrix}I_N & 0_N\\ 0_N &0 _N \end{psmallmatrix}$ has a zero first block column
 \begin{align*}
\big[X_n(z)\big]_{\text{regular}}=\PARENS{\begin{matrix}
I_N & z^n Q^L_{1,n}(z)\\
0_N & -z^n\big(H^R_{n-1}\big)^{-1} Q^R_{n-1}(z)
\end{matrix}}.
 \end{align*}
 \end{enumerate}

 \begin{definition}
We write the  Taylor series of $X_n(z)$ about infinity, which converges for $z\in\bar{\mathbb{D}}$, as follows
 	\begin{align}\label{eq:taylorXn}
 	X_n(z)&=I_{2N}+X_n^{(1)}
 	z^{-1}+X_n^{(2)}z^{-2}+\cdots
 	\end{align}
 	where
 	\begin{align*}
 	X^{(i)}_n&:=
 	\PARENS{\begin{matrix}
 	a^{(i)}_n & b^{(i)}_n\\
 	c^{(i)}_n & d^{(i)}_n
 	\end{matrix}}
 	&i&\geq 0, & 	a^{(i)}_n , b^{(i)}_n,,	c^{(i)}_n,  d^{(i)}_n\in\C^{N\times N}.
 	 	\end{align*}
 	For $n=1$ we  use the  simplified notation $X^{(1)}_n=\begin{psmallmatrix}
 	a_n & b_n\\c_n &d_n
 	\end{psmallmatrix}$.
 \end{definition}

Then, the matrix function $Y_n(Z)$  for   $z\in\bar{\mathbb{D}}$ can be expressed as
\begin{align}
 Y_{n}(z)=\PARENS{\begin{matrix}
 z^{n}I_N+a_{n}z^{n-1}+O(z^{n-2}) &
 b_{n}z^{-n-1}+b^{(2)}_{n}z^{-n-2}+O(z^{-n-3})\\
 c_{n}z^{n-1}+c^{(2)}_{n}z^{n-2}+O(z^{n-3}) &
 z^{-n}I_N+d_{n}z^{-n-1}+O(z^{-n-2})
 \end{matrix}}.
 \label{eq:asint}
 \end{align}
 Where ---recall \eqref{eq:ypsilon}, \eqref{eq:ypsilon0} and \eqref{eq:asint}---
 we have
 \begin{align}
 a_{n}^{(i)}&=c^{(i)}_{n}=0_N,& i>n,\label{eq:AC0}\\
 d^{(i)}_{0}&=0_N,&i\geq1.\label{eq:D0}
 \end{align}
 Then, Proposition \ref{pro:Xlaurent} implies
 \begin{align*}
 a_n&=P^L_{1,n,n-1}, & b_n&=-\hat P^{L}_{1,n}(1),\\
 c_n &=-\big(H^R_{n-1}\big)^{-1}\big(\alpha^R_{2,n-1}\big)^\dagger,& d_n&= \big(H^R_{n-1}\big)^{-1}\hat P^{R}_{2,n-1}(2),\\
 a_n^{(2)}&=P^L_{1,n,n-2}, & b_n^{(2)}&=-\hat P^{L}_{1,n-1}(2),\\
 c_n^{(2)} &=-\big(H^R_{n-1}\big)^{-1}\big(P^R_{2,n-1,1}\big)^\dagger,& d_n^{(2)}&= \big(H^R_{n-1}\big)^{-1}\hat P^{R}_{2,n-1}(3).
 \end{align*}

\section{Recursion relations and systems of matrix linear ordinary differential equations}\label{S:M}
From hereon we will assume that the matrix of weights $w:\T\to\C^{N\times N}$ has an analytic extension to the annulus $\C\setminus\{0\}$ with an analytic inverse on this annulus.\footnote{Consequently, the right  logarithmic derivative $\dfrac{\d w(z)}{ \d z}\big(w(z)\big)^{-1}$ is also analytic in the annulus $\mathbb C\setminus \{0\}$.} This is, indeed, a  strong assumption but one can imagine many examples fitting in this family, for example weights of Freud type of the form $w(\zeta)=\prod\limits_{i=1}^M\exp (V_i(\zeta))$.
where $V_i(\zeta)$, $i\in\{1,\dots,M\}$,  are matrix  Laurent polynomials in $\zeta\in\T$.
For example
$w(\zeta) =\exp(V\zeta)\exp(V^\dagger\bar{\zeta})$ will be of this type and moreover Hermitian. This example, in the scalar case $N=1$, is called  in \cite{ismail} modified Bessel, as its moment are connected with the modified Bessel functions. Another scenario is to have
\begin{align*}
w(z)=z^{-m}v(z)
\end{align*}
$m\in\{0,1,2,\dots\}$, where $v(z)$ is analytic in $\C$ with Taylor series  $v(z)=v_{0}+v_1 z+\cdots$, where $v_0$ is non singular.

\begin{definition} We consider
\begin{align}\label{eq:Z}
Z_n(z)&:=\PARENS{\begin{matrix} P^{L}_{1,n}(z) & Q^{L}_{1,n}(z)
\\ -\big(H^R_{n-1}\big)^{-1}\tilde P^R_{2,n-1}(z) &
-\big(H^R_{n-1}\big)^{-1} Q^R_{2,n-1}(z)\end{matrix}}\PARENS{\begin{matrix}
 w(z)z^{-n} & 0_N \\ 0_N & I_N
\end{matrix}},& n&\in\{1,2,\dots\}
\end{align}
and
\begin{align*}
Z_0(z)&:=\PARENS{\begin{matrix}w(z) & Q^{L}_{1,0}(z) \\ 0_N &
{I_N}\end{matrix}}.
\end{align*}
\end{definition}
 \begin{proposition}
 	For $\zeta\in \mathbb{T}$ and for each $n\in\{0,1,2,\dots\}$ the block matrix function $Z_n(z)$ satisfies the following jump condition f
\begin{align}\label{eq:jumpZ}
   \big( Z_{n}\big)_+(\zeta)=\big(Z_{n}\big)_-(\zeta)\PARENS{\begin{matrix} I_N & I_N \\ 0_N &
      I_N\end{matrix}}.
\end{align}
 \end{proposition}
 \begin{proof}
For $\zeta\in\T$ we have
\begin{align*}
(Z_n)_+(\zeta)&=(Y_n)_+(\zeta)\PARENS{\begin{matrix}
w(\zeta)\bar{\zeta}^n & 0_N\\
0_N & I_N
\end{matrix}}\\
&=(Y_n)_-(\zeta)\PARENS{\begin{matrix}
I_N & w(\zeta)\bar{\zeta}^n\\
0_N &I_N
\end{matrix}}\PARENS{\begin{matrix}
w(\zeta)\bar{\zeta}^n & 0_N\\
0_N & I_N
\end{matrix}}\\
&= (Y_n)_-(\zeta)\PARENS{\begin{matrix}
 w(\zeta)\bar{\zeta}^n & w(\zeta)\bar{\zeta}^n\\
0_N &I_N
\end{matrix}}\\
&=(Y_n)_-(\zeta)\PARENS{\begin{matrix}
w(\zeta)\bar{\zeta}^n & 0_N\\
0_N &I_N
\end{matrix}}\PARENS{\begin{matrix}
I_N & I_N\\
0_N & I_N
\end{matrix}}\\
&= (Z_n)_-(\zeta)\PARENS{\begin{matrix}
I_N & I_N\\
0_N & I_N
\end{matrix}}.
\end{align*}
 \end{proof}
 It is remarkable that the jump condition is now is expressed in terms of a constant matrix.
The block matrix $Z_n(z)$ can also be regarded as the solution of a Riemann--Hilbert problem. Precisely, the following simple result holds.

\begin{proposition}\label{pro:Z} The block matrix function $Z_{n}(z)$ is the unique matrix such that
\begin{enumerate}
\item $Z_{n}(z)\begin{psmallmatrix}
(w(z))^{-1}z^n & 0_N\\0_N& I_N
\end{psmallmatrix}$ is analytic at $\C\setminus \T$.
\item Satisfies the jump condition \eqref{eq:jumpZ} $\forall\zeta\in \T$.
\item About infinity has the following asymptotic:
$Z_n(z) = (I_{2N}+O(z^{-1})) \PARENS{\begin{matrix} w(z) & 0_N \\ 0_N &
 I_N z^{-n} \end{matrix}}$ for $z \rightarrow \infty $.
\end{enumerate}
\end{proposition}

\subsection{Recursion relations}
\begin{definition}[Szegő matrix]
	For each  $n\in\{0,1,2,\dots\}$ we introduce the Szegő matrices
\begin{align*}
	R_n(z)&:=Z_{n+1}(z)\big(Z_n(z)\big)^{-1}.
\end{align*}
\end{definition}
\begin{definition}
	Given a Laurent series  $L(z)=\sum\limits_{m\in\mathbb Z}L_mz^m$ the expression  stands for the Laurent series gotten from $L(z)$ by disregarding the powers less than $j$, where $j\in\mathbb Z$, i.e., $[L(z)]_{\geq j}:=\sum\limits_{m\geq j}L_mz^m$.
\end{definition}
\begin{proposition}\label{pro:RY}
The Szegő matrix $R_n(z)$ can be alternatively expressed as
\begin{align}\label{eq:RY}
R_n(z)&= Y_{n+1}(z) \PARENS{\begin{matrix}
z^{-1} I_N & 0_N\\
0_N & I_N
\end{matrix}}\big(Y_n(z)\big)^{-1}
 \\&=X_{n+1}(z) \PARENS{\begin{matrix}
 I_N & 0_N\\
 0_N & z^{-1}  I_N
 \end{matrix}}\big(X_n(z)\big)^{-1}.\label{eq:RX}
\end{align}
\end{proposition}
\begin{lemma}\label{lemma:Ranalytic}
	The matrix $R_n$ is analytic at $\C\setminus\{0\}$.
\end{lemma}
\begin{proof}
The matrices $Z_{n+1}(z)$ and  $(Z_{n}(z))^{-1}$ are analytic  at $\mathbb{C} \setminus (\mathbb T\cup \{0\} )$ while the $Z_n(z)$'s jump at $\mathbb T$ is a constant matrix which does not depend on $n$, i.e.,
	$R_n$ is continuous on the unit circle and, consequently,  analytic in the complex plane but for the origin.
\end{proof}

\begin{proposition}
	The Szegő matrix has the form
	\begin{align}
R_n(z)&=R_{n,0}+R_{n,-1}z^{-1},
\label{eq:matrizR}
	\end{align}
	with
	\begin{align*}
	R_{n,0}&:=\PARENS{\begin{matrix}
	I_N & 0_N\\
	0_N& 0_N
	\end{matrix}},&
	R_{n,-1}&:=
	\PARENS{\begin{matrix}
	a_{n+1}-a_{n} & -b_{n}\\
	c_{n+1} & I_N
	\end{matrix}}.
	\end{align*}
\end{proposition}
\begin{proof}
	 Considering  that $R_n(z)$ is analytic on the complex plane but for a possible singularity at $z=0$ and
	 \eqref{eq:RY}, we deduce  that  $R_n(z)$ has  a simple pole at the origin. To explicitly compute $R_n(z)$ we use \eqref{eq:RX} at $z=\infty$  and, as there is no jump at the unit circle $\T$, analytically extend the result  to  the annulus $\C\setminus\{0\}$, which can be easily achieved if we  truncate the Taylor series about infinity and keep only those terms involving  $z^j$ with $j\geq-1$
\begin{align*}
R_n(z)&=\Big[(I_{2N}+X_{n+1}^{(1)}
 z^{-1}+\cdots)\PARENS{\begin{matrix}
  {I_N} & 0_N \\ 0_N &
  {I_N}z^{-1} \end{matrix}}
(I_{2N}-X_{n}^{(1)}z^{-1}+\cdots)\Big ]_{\geq
  -1}\\ &=\PARENS{\begin{matrix} {I_N} & 0_N \\ 0_N &
 {I_N}z^{-1}
\end{matrix}}-\PARENS{\begin{matrix}
{I_N} & 0_N \\ 0_N &
 0_N \end{matrix}}X_n^{(1)}z^{-1}+X_{n+1}^{(1)}
\PARENS{\begin{matrix} {I_N} & 0_N \\ 0_N & 0_N \end{matrix}}z^{-1}\\
&=\PARENS{\begin{matrix} {I_N}+(a_{n+1}-a_{n})z^{-1} &
-b_{n}z^{-1} \\ c_{n+1}z^{-1} & z^{-1}{I_N} \end{matrix}}.
\end{align*}
\end{proof}

\begin{corollary}
The recursion equations
	\begin{align}\label{hola}
	Y_{n+1}(z)=R_n(z)Y_n(z)\PARENS{\begin{matrix} I_N z &
	0_N \\ 0_N & I_N \end{matrix}},\\
	X_{n+1}(z)=R_n(z)X_n(z)\PARENS{\begin{matrix} {I_N} & 0_N
	\\ 0_N & {I_N}z \end{matrix}},
	\label{eq:RS}
	\end{align}
	hold true.
	We also have
	\begin{align}\label{hola1}
	Y_n(z)&=R_{n-1}(z)\cdots R_0(z) \PARENS{\begin{matrix} {I_N} z^n& {\oint_{\mathbb{T}}\frac{w(\zeta)\d m(\zeta)}{1-\bar \zeta z}},\\[10pt] 0_N &
	I_N\end{matrix}}\\
	X_n(z)&=R_{n-1}(z)\cdots R_0(z) \PARENS{\begin{matrix} {I_N} & z^n{\oint_{\mathbb{T}}\frac{w(\zeta)\d m(\zeta)}{1-\bar \zeta z}}\\[10pt] 0_N &
	I_Nz^n\end{matrix}}.\label{eq:RS1}
	\end{align}
\end{corollary}
\begin{proof}
	Equations \eqref{hola} and \eqref{eq:RS} are a direct consequence of Proposition \ref{pro:RY}, and form them we derive \eqref{hola1} and \eqref{eq:RS1}.
\end{proof}
Notice that $R_{n-1}(z)\cdots R_0(z)$ is usually called as transfer matrix, see \cite{golinski}.
The following recursion relations have been proved by algebraic means in \cite{Cafasso} and \cite{AM} for the Szegő matrix polynomials. Here we give a more analytical proof based in the Riemann--Hilbert problem and add  two analogous recursion relations  for the Cauchy transforms, which are not an immediate consequence of the previous recursions
\begin{theorem}[Recursion relations]
The  following equations
	\begin{align*}
	P^{L}_{1,n+1}(z) &=  zP^{L}_{1,n}(z)+\alpha^L_{1,n+1}\tilde P^R_{2,n}(z),\\
	\tilde P^R_{2,n}(z) &= \big(\alpha_{2,n}^R\big)^\dagger P^{L}_{1,n}(z)+\big(I_N-\big(\alpha_{2,n}^R\big)^\dagger \alpha^L_{1,n}\big)\tilde P^R_{2,n-1}(z),\\
	Q^{L}_{1,n+1}(z) &= Q^{L}_{1,n}(z)+ \alpha^L_{1,n+1}Q^R_{2,n}(z),\\
z	Q^R_{2,n}(z) &= \big(\alpha_{2,n}^R\big)^\dagger Q^{L}_{1,n}(z)+ \big(I_N-\big(\alpha_{2,n}^R\big)^\dagger \alpha^L_{1,n}\big)Q^R_{2,n-1}(z),
	\end{align*}	
	are satisfied.
\end{theorem}
\begin{proof}
The first step is to write 
\begin{multline*}
\PARENS{\begin{matrix} P^{L}_{1,n+1}(z) & Q^{L}_{1,n+1}(z)
\\ -\big(H^R_{n}\big)^{-1}\tilde P^R_{2,n}(z) &
-\big(H^R_{n}\big)^{-1} Q^R_{2,n}(z)\end{matrix}}\\=
\PARENS{\begin{matrix} {I_N}+(a_{n+1}-a_{n})z^{-1} &
-b_{n}z^{-1} \\ c_{n+1}z^{-1} &{I_N} z^{-1}\end{matrix}}\PARENS{\begin{matrix} P^{L}_{1,n}(z)z & Q^{L}_{1,n}(z)
\\ -\big(H^R_{n-1}\big)^{-1}\tilde P^R_{2, n-1}(z) z&
-\big(H^R_{n-1}\big)^{-1} Q^R_{2,n-1}(z)\end{matrix}},
\end{multline*}
from where we deduce
	\begin{align*}
	P^{L}_{1,n+1}(z) &= \big({I_N}z+(a_{n+1}-a_{n})\big)P^{L}_{n}(z)+b_{n}\big(H^R_{n-1}\big)^{-1}\tilde P^R_{2,n-1}(z),\\
-\big(H^R_{n}\big)^{-1}	\tilde P^R_{2,n}(z) &= c_{n+1}P^{L}_{1,n}(z)-\big(H^R_{n-1}\big)^{-1}\tilde P^R_{2,n-1}(z),\\
	Q^{L}_{1,n+1}(z) &= \big({I_N}+(a_{n+1}-a_{n})z^{-1}\big)Q^{L}_{1,n}(z)+b_{n}\big(H^R_{n-1}\big)^{-1}Q^R_{2,n-1}(z)z^{-1},\\
-\big(H^R_{n}\big)^{-1}	 Q^R_{2,n}(z) &= c_{n+1}z^{-1}Q^{L}_{1,n}(z)-z^{-1}\big(H^R_{n-1}\big)^{-1} Q^R_{2,n-1}(z),
	\end{align*}
and, consequently,  obtain
	\begin{align*}
	P^{L}_{1,n+1}(z) &= \big({I_N}z+a_{n+1}-a_{n}+b_{n}c_{n+1}\big)P^{L}_{1,n}(z)+ b_{n}\big(H^R_{n}\big)^{-1}\tilde P^R_{2,n}(z),\\
	\tilde P^R_{2,n}(z) &= -H^R_{n}c_{n+1}P^{L}_{1,n}(z)+H^R_{n}\big(H^R_{n-1}\big)^{-1}\tilde P^R_{2,n-1}(z),\\
	Q^{L}_{1,n+1}(z) &= \big({I_N}+(a_{n+1}-a_{n}+b_{n}c_{n+1})z^{-1}\big)Q^{L}_{1,n}(z)+ b_{n}\big(H^R_{n}\big)^{-1}Q^R_{2,n}(z),\\
	Q^R_{2,n}(z) &= -H^R_{n}c_{n+1}Q^{L}_{1,n}(z)z^{-1}+ H^R_{n}\big(H^R_{n-1}\big)^{-1}Q^R_{2,n-1}(z)z^{-1}.
	\end{align*}
Now, from \eqref{eq:ortogonal} we conclude
{\small\begin{align*}
0_N&=\oint_\T  P^{L}_{1,n+1}(\zeta) w(\zeta)\d m(\zeta)\bar\zeta^{n-1}\\&=\oint_\T\Big({I_N}\zeta+a_{n+1}-a_{n}+b_{n}c_{n+1}\big)P^{L}_{1,n}(\zeta)+ b_{n}\big(H^R_{n}\big)^{-1}\tilde P^R_{2,n}(\zeta)\Big)w(\zeta)\d m(\zeta)\bar\zeta^{n-1}\\&=\begin{multlined}[t]
\oint_\T P^{L}_{1,n}(\zeta)w(\zeta)\d m(\zeta)\bar\zeta^{n-1}\\+\big(a_{n+1}-a_{n}+b_{n}c_{n+1}\big)\oint_\T P^{L}_{1,n}(\zeta)w(\zeta)\d m(\zeta)\bar\zeta^{n-1}+b_{n}\big(H^R_{n}\big)^{-1}\oint_\T \tilde P^R_{2,n}(\zeta)w(\zeta)\d m(\zeta)\bar\zeta^{n-1}
\end{multlined}\\
&=\big(a_{n+1}-a_{n}+b_{n}c_{n+1}\big)H^L_n
	\end{align*}}
	and, therefore, we infer that
	\begin{align}\label{eq:abc}
a_{n+1}-a_{n}+b_{n}c_{n+1}=0_N.
	\end{align}
Then, the recursion relations simplifies to
	\begin{align*}
	P^{L}_{1,n+1}(z) &=  P^{L}_{1,n}(z)z+ b_{n}\big(H^R_{n}\big)^{-1}\tilde P^R_{2,n}(z),\\
	\tilde P^R_{2,n}(z) &= -H^R_{n}c_{n+1}P^{L}_{1,n}(z)+H^R_{n}\big(H^R_{n-1}\big)^{-1}\tilde P^R_{2,n-1}(z),\\
	Q^{L}_{1,n+1}(z) &= Q^{L}_{1,n}(z)+ b_{n}\big(H^R_{n}\big)^{-1}Q^R_{2,n}(z),\\
	Q^R_{2,n}(z) &= -H^R_{n}c_{n+1}Q^{L}_{1,n}(z)z^{-1}+ H^R_{n}\big(H^R_{n-1}\big)^{-1}Q^R_{2,n-1}(z)z^{-1}.
	\end{align*}
	A further simplification is obtained by evaluating the first recursion equation at $z=0$, which gives
	\begin{align*}
	\alpha^L_{1,n+1}=b_n\big(H^R_{n}\big)^{-1},
	\end{align*}
	so that
		\begin{align}\label{eq:b}
		b_n=\alpha^L_{1,n+1}H^R_{n}.
		\end{align}
		Moreover, from \eqref{eq:asint} and \eqref{eq:ypsilon} we get
		\begin{align}\label{eq:c}
		c_n=-(H^R_{n-1})^{-1}\big(\alpha_{2,n-1}^R\big)^\dagger,
		\end{align}
		which introduced in the second recursion equation gives
		\begin{align*}
		\tilde P^R_{2,n}(z) &= \big(\alpha_{2,n}^R\big)^\dagger P^{L}_{1,n}(z)+H^R_{n}\big(H^R_{n-1}\big)^{-1}\tilde P^R_{2,n-1}(z),
 		\end{align*}
 		that evaluated at $z=0$ implies
 		\begin{align*}
 			I_N &= \big(\alpha_{2,n}^R\big)^\dagger \alpha^{L}_{1,n}+H^R_{n}\big(H^R_{n-1}\big)^{-1},
 			 		\end{align*}
 			so that	
 			\begin{align}\label{eq:HRVer}
 			H^R_{n}\big(H^R_{n-1}\big)^{-1}=I_N-\big(\alpha_{2,n}^R\big)^\dagger \alpha^L_{1,n}.
 			\end{align}
\end{proof}

\begin{proposition}\label{pro:Verblusnky for ever}
 We have that the following relations
\begin{align*}
H^R_{n}\big(H^R_{n-1}\big)^{-1}&=I_N-\big(\alpha_{2,n}^R\big)^\dagger \alpha^L_{1,n},&
	H^L_{n}\big(H^L_{n-1}\big)^{-1}&=I_N-\alpha^L_{1,n}\big(\alpha_{2,n}^R\big)^\dagger
\end{align*}
are satisfied and, consequently,  the matrices
\begin{align*}
&I_N-\big(\alpha_{2,n+1}^R\big)^\dagger \alpha^L_{1,n+1}, & &I_N-\alpha^L_{1,n+1}\big(\alpha_{2,n+1}^R\big)^\dagger
\end{align*}
are not singular.
At the origin we have
	\begin{align*}
Q^L_{1,n} (0)&=H^L_{n}, &
	Q^R_{2,n}(0)&= -	\big(\alpha_{2,n+1}^R\big)^\dagger H^{L}_{n}= -	H^{R}_{n}\big(\alpha_{2,n+1}^L\big)^\dagger ,
	\end{align*}
	while about infinity the behavior is
	\begin{align*}
	\lim_{z\to\infty} \big(z^{n+1}Q^L_{1,n}(z)\big)&=-\alpha^L_{1,n+1}H^R_n=-H^L_n,\alpha^R_{1,n+1} &
	\lim_{z\to\infty} \big(z^{n+1}Q^R_{2,n}(z)\big)&=H^R_n.
	\end{align*}
\end{proposition}
\begin{proof}
	Part of these statements are just a recollection of some intermediate conclusions obtained  in the  discussion of the previous proof  and also of Proposition \ref{pro:asymptotics}. For the other we  argue as follows.
	As the Cauchy transforms are analytic at the origin we can evaluate the third and fourth recursion relations at $z=0$ to get
		\begin{align*}
		H^L_{n+1}&=H^L_n+ \alpha^L_{1,n+1}Q^R_{2,n}(0),\\
		0_N&= \big(\alpha_{2,n}^R\big)^\dagger H^{L}_{1,n}+ \big(I_N-\big(\alpha_{2,n}^R\big)^\dagger \alpha^L_{1,n}\big)Q^R_{2,n-1}(0).
		\end{align*}
		Now, from \eqref{eq:HRVer} we deduce that $I_N-\big(\alpha_{2,n}^R\big)^\dagger \alpha^L_{1,n}$ is not a singular matrix, so that
		we can clean  $Q^R_{2,n-1}(0)$	and write
		\begin{align}\label{eq:QR0larga}
		Q^R_{2,n-1}(0)&= -\Big(I_N-\big(\alpha_{2,n}^R\big)^\dagger \alpha^L_{1,n}\Big)^{-1}\big(\alpha_{2,n}^R\big)^\dagger H^{L}_{n} ,
		\end{align}	
		that introduced in the first equation gives
		\begin{align*}
		\Big(I_N+ \alpha^L_{1,n+1}\Big(I_N-\big(\alpha_{2,n+1}^R\big)^\dagger \alpha^L_{1,n+1}\Big)^{-1}\big(\alpha_{2,n+1}^R\big)^\dagger\Big)H^L_{n+1}&=H^L_n.
		\end{align*}	
		But, given two matrices $A$ and $B$ such that $I_N-BA$ is not a singular matrix, then the matrix $(I_N-AB)$ is neither singular, and
		$(I_N-AB)^{-1}=I_N+A(I_N-BA)^{-1}B$.  Therefore, we get
		\begin{align*}
		H^L_{n+1}\big(H^L_n\big)^{-1}=I_N-\alpha^L_{1,n+1}\big(\alpha_{2,n+1}^R\big)^\dagger.
		\end{align*}	
		Finally, we look at the behavior about infinity of   third recursion relation to get
		\begin{align*}
		\lim_{z\to\infty}	\big(z^{n+1}Q^{L}_{1,n+1}(z)\big) = 	\lim_{z\to\infty}	\big(z^{n+1}Q^{L}_{1,n}(z)\big)+ \alpha^L_{1,n+1}\lim_{z\to\infty}	\big(z^{n+1}Q^R_{2,n}(z)\big),
		\end{align*}
		and Proposition \ref{pro:asymptotics} 	gives
		\begin{align*}
		0= 	\lim_{z\to\infty}	\big(z^{n+1}Q^{L}_{1,n}(z)\big)+ \alpha^L_{1,n+1}H^R_n.
		\end{align*}
Now we simplify \eqref{eq:QR0larga}, which can be written as
\begin{align*}
Q^R_{2,n}(0)&=-H^R_n\big(H^R_{n+1}\big)^{-1}\big(\alpha^R_{2,n+1}\big)^\dagger H^L_{n+1}\\
&=-H^R_n\big(\alpha^L_{2,n+1}\big)^\dagger \\
&=-\big(\alpha^R_{2,n+1}\big)^\dagger H^L_n,\end{align*}
where we have used Proposition 17 of \cite{AM}.
\end{proof}

\begin{corollary}
	The Szegő matrix can be written as follows
	\begin{align}\label{eq:R_ver}
	R_n(z)=\PARENS{\begin{matrix}
	I_N+\alpha^L_{1,n+1}(\alpha^R_{2,n})^\dagger z^{-1}& -\alpha^L_{1,n+1}H^R_{n}z^{-1}\\ -(H^R_{n})^{-1}\big(\alpha_{2,n}^R\big)^\dagger z^{-1}& I_N z^{-1}
	\end{matrix}}.
	\end{align}
\end{corollary}
\begin{proof}
Just use \eqref{eq:b}, \eqref{eq:c}  and \eqref{eq:abc}.
\end{proof}

\begin{proposition}[Verblunsky parametrization of $X_n$]\label{pro:verblusnky_parametrization}
	The coefficients   $X_n^{(j)}$, for each $n\in\{0,1,2,\dots\}$  can be parametrized in terms of the Verblunsky matrices $\Big\{\alpha^L_{1,m},\big(\alpha^R_{2,m}\big)^\dagger\Big\}_{m=1}^n$,
 the quasi-tau matrices $\big\{H^R_m\big\}_{m=0}^n$ and $X_{n=1}$. In particular,	
\begin{enumerate}
	\item 	For $X_n^{(1)}=\begin{psmallmatrix}
	a_n & b_n\\ c_n & d_n
	\end{psmallmatrix}$ we have the expressions
\begin{align}\label{eq:abcd1}
\begin{aligned}
a_{n}&=\sum_{m=0}^{n-1}\alpha^L_{1,m+1}(\alpha^R_{2,m})^\dagger, & b_n&=\alpha^L_{1,n+1}H^R_n,\\
c_n&=-(H_{n-1}^R)^{-1}\big(\alpha^R_{2,n-1}\big)^\dagger, & d_{n}
&=-\sum_{m=0}^{n-1}(H^R_{m})^{-1}\big(\alpha_{2,m}^R\big)^\dagger \alpha^L_{1,m+1}H^R_{m},
\end{aligned}
\end{align}
	where we have introduced, for convenience, the notation $\alpha^R_{2,0}:=I_N$.
	\item 	For $X_n^{(2)}=\begin{psmallmatrix}
	a_n^{(2)} & b_n^{(2)}\\ c_n^{(2)} & d_n^{(2)}
	\end{psmallmatrix}$ we have the expressions
	\begin{align*}
	a^{(2)}_{n+1}&=a^{(2)}_1+\sum_{m=1}^n A^{(2)}_m, \\
	b^{(2)}_n&=\alpha^L_{1,n+1}H^R_{n+1}-\alpha^L_{1,n+1}(\alpha^R_{2,n})^\dagger \alpha^L_{1,n}H^R_n -\alpha^L_{1,n+1}H^R_{n}\Big(\alpha^R_{1,1}+\sum_{m=1}^{n-1}\big(\alpha^L_{2,m}\big)^\dagger\alpha^R_{1,m+1}\Big),\\
	c^{(2)}_{n+1}&=-(H^R_{n})^{-1}\big(\alpha_{2,n}^R\big)^\dagger \Big(\alpha^L_{1,1}+\sum_{m=1}^{n-1}\alpha^L_{1,m+1}(\alpha^R_{2,m})^\dagger \Big)-(H_{n-1}^R)^{-1}\big(\alpha^R_{2,n-1}\big)^\dagger, \\
	d^{(2)}_{n+1}&=d^{(2)}_{1}+\sum_{m=0}^{n}D^{(2)}_m,
	\end{align*}
	where
	\begin{align*}
	A_n^{(2)}&:=	\alpha^L_{1,n+1}(\alpha^R_{2,n})^\dagger \Big(\alpha^L_{1,1}+\sum_{m=1}^{n-1}\alpha^L_{1,m+1}(\alpha^R_{2,m})^\dagger \Big)+\alpha^L_{1,n+1}H^R_{n}(H_{n-1}^R)^{-1}\big(\alpha^R_{2,n-1}\big)^\dagger,\\
	D_n^{(2)}&:=	-(H^R_{n})^{-1}\big(\alpha_{2,n}^R\big)^\dagger \alpha^L_{1,n+1}H^R_{n+1}-\alpha^L_{1,n+1}(\alpha^R_{2,n})^\dagger \alpha^L_{1,n}H^R_n -\alpha^L_{1,n+1}H^R_{n}\Big(\alpha^R_{1,1}+\sum_{m=1}^{n-1}\big(\alpha^L_{2,m}\big)^\dagger\alpha^R_{1,m+1}\Big).
	\end{align*}
\end{enumerate}
\end{proposition}
\begin{proof}
	The expressions for $b_n$ and  $c_n$ were deduced before, see \eqref{eq:b}, \eqref{eq:c} and \eqref{eq:abc}
implies
\begin{align*}
a_{n+1}-a_n=\alpha^L_{1,n+1}(\alpha^R_{2,n})^\dagger.
\end{align*}
Observe that we can use a telescoping sum to get
\begin{align*}
a_{n+1}-a_1= \sum_{m=1}^n(a_{m+1}-a_{m})=\sum_{m=1}^n\alpha^L_{1,m+1}(\alpha^R_{2,m})^\dagger
\end{align*}
with $a_1=P^L_{1,1,0}=\alpha^L_{1,1}$. Consequently, the expression for $a_n$ follows.

Now we start we the essential part of the proof. We will make a substantial use of
\eqref{eq:RS} in the form
	\begin{align}\label{hello}
	X_{n+1}(z)\PARENS{\begin{matrix} {I_N} & 0_N
	\\ 0_N & {I_N}z^{-1} \end{matrix}}=R_n(z)X_n(z),
	\end{align}
which can be expanded, for $j\in\{1,2,\dots\}$,  as follows
\begin{align*}
X^{(j+1)}_{n+1}\PARENS{\begin{matrix}
I_N & 0_N\\
0_N & 0_N
\end{matrix}}+X^{(j)}_{n+1}\PARENS{\begin{matrix}
0_N & 0_N\\
0_N & I_N
\end{matrix}}=\PARENS{\begin{matrix}
I_N & 0_N\\
0_N & 0_N
\end{matrix}}X^{(j+1)}_n+\PARENS{\begin{matrix}
a_{n+1}-a_n& -b_{n}\\ c_{n+1} & I_N
\end{matrix}}X^{(j)}_n.
\end{align*}
In terms of the different blocks we deduce
\begin{align}\label{eq:abcd-verblusnky1-0}
\begin{aligned}
a^{(j+1)}_{n+1}-a^{(j+1)}_n&=(a_{n+1}-a_n) a^{(j)}_n -b_nc^{(j)}_n,&
b^{(j+1)}_n&=b^{(j)}_{n+1}-(a_{n+1}-a_n)  b^{(j)}_{n} +b_{n}d^{(j)}_n,\\
c^{(j+1)}_{n+1}&=c_{n+1} a^{(j)}_n+c^{(j)}_n, &
d^{(j)}_{n+1}-d^{(j)}_n&=c_{n+1}b^{(j)}_n.
\end{aligned}
\end{align}
which for $j=1$ reads
\begin{align}\label{eq:abcd-verblusnky1-1}
\begin{aligned}
a^{(2)}_{n+1}-a^{(2)}_n&=(a_{n+1}-a_n) a_n -b_nc_n,&
b^{(2)}_n&=b_{n+1}-(a_{n+1}-a_n)  b_{n} +b_{n}d_n,\\
c^{(2)}_{n+1}&=c_{n+1} a_n+c_n, &
d_{n+1}-d_n&=c_{n+1}b_n.
\end{aligned}
\end{align}
In particular, for $j=1$ the last equation is the following telescoping relation
\begin{align*}
d_{n+1}-d_n&=-(H^R_{n})^{-1}\big(\alpha_{2,n}^R\big)^\dagger \alpha^L_{1,n+1}H^R_{n},
\end{align*}
where we have used \eqref{eq:b}. Using the telescoping trick we obtain
\begin{align*}
d_{n+1}=d_1-\sum_{m=1}^{n}(H^R_{m})^{-1}\big(\alpha_{2,m}^R\big)^\dagger \alpha^L_{1,m+1}H^R_{m}.
\end{align*}
Here, according to \eqref{eq:asint} we have
$d_1=-\Big(\oint_\T w(\zeta) \d m(\zeta)\Big)^{-1}\oint_\T w(\zeta) \zeta\d m(\zeta)$, and recalling \eqref{eq:ortogonall} we see $d_1
=-\big(H^R_0\big)^{-1}\alpha_{1,1}^LH^R_0$.\footnote{We have $\mu(\mathbb T)=H^R_0=H^L_0$ and $\oint_\T w(\zeta) \zeta\d m(\zeta)=-\alpha_{1,1}^L\mu(\T)=-\mu(\T) \alpha^R_{1,1}$.}
In terms of Verblunsky coefficients \eqref{eq:abcd-verblusnky1-0} is
\begin{align}\label{eq:abcd-verblusnky2}
\begin{aligned}
a^{(j+1)}_{n+1}-a^{(j+1)}_n&=\alpha^L_{1,n+1}(\alpha^R_{2,n})^\dagger a^{(j)}_n -\alpha^L_{1,n+1}H^R_{n}c^{(j)}_n,&
b^{(j+1)}_n&=b^{(j)}_{n+1}-\alpha^L_{1,n+1}(\alpha^R_{2,n})^\dagger b^{(j)}_{n} +\alpha^L_{1,n+1}H^R_{n}d^{(j)}_n,\\
c^{(j+1)}_{n+1}&=-(H^R_{n})^{-1}\big(\alpha_{2,n}^R\big)^\dagger a^{(j)}_n+c^{(j)}_n, &
d^{(j)}_{n+1}-d^{(j)}_n&=-(H^R_{n})^{-1}\big(\alpha_{2,n}^R\big)^\dagger b^{(j)}_n
\end{aligned}
\end{align}
we conclude that whenever the coefficients $a^{(j)}_n,b_n^{(j)},c^{(j)}_n$ are given for all $n\in\{0,1,2,\dots\}$ we can determine $d^{(j)}_n$ and $a^{(j+1)}_n,b_n^{(j+1)},c^{(j+1)}_n$ also for all $n\in\{0,1,2,\dots\}$. Which is the main statement of the Proposition. For example, let us put again $j=1$ in \eqref{eq:abcd-verblusnky2} for the first three equations and $j=2$ for the last equation to get
\begin{align*}
a^{(2)}_{n+1}-a^{(2)}_n&=\alpha^L_{1,n+1}(\alpha^R_{2,n})^\dagger a_n -\alpha^L_{1,n+1}H^R_{n}c_n,&
b^{(2)}_n&=b_{n+1}-\alpha^L_{1,n+1}(\alpha^R_{2,n})^\dagger b_{n} +\alpha^L_{1,n+1}H^R_{n}d_n,\\
c^{(2)}_{n+1}&=-(H^R_{n})^{-1}\big(\alpha_{2,n}^R\big)^\dagger a_n+c_n, &
d^{(2)}_{n+1}-d^{(2)}_n&=-(H^R_{n})^{-1}\big(\alpha_{2,n}^R\big)^\dagger b^{(2)}_n,
\end{align*}
so that using \eqref{eq:abcd1} we deduce
\begin{align*}
a^{(2)}_{n+1}-a^{(2)}_n&=\alpha^L_{1,n+1}(\alpha^R_{2,n})^\dagger \Big(\alpha^L_{1,1}+\sum_{m=1}^{n-1}\alpha^L_{1,m+1}(\alpha^R_{2,m})^\dagger \Big)+\alpha^L_{1,n+1}H^R_{n}(H_{n-1}^R)^{-1}\big(\alpha^R_{2,n-1}\big)^\dagger,\\
b^{(2)}_n&=\alpha^L_{1,n+1}H^R_{n+1}-\alpha^L_{1,n+1}(\alpha^R_{2,n})^\dagger \alpha^L_{1,n}H^R_n -\alpha^L_{1,n+1}H^R_{n}\Big(\alpha^R_{1,1}+\sum_{m=1}^{n-1}\big(\alpha^L_{2,m}\big)^\dagger\alpha^R_{1,m+1}\Big),\\
c^{(2)}_{n+1}&=-(H^R_{n})^{-1}\big(\alpha_{2,n}^R\big)^\dagger \Big(\alpha^L_{1,1}+\sum_{m=1}^{n-1}\alpha^L_{1,m+1}(\alpha^R_{2,m})^\dagger \Big)-(H_{n-1}^R)^{-1}\big(\alpha^R_{2,n-1}\big)^\dagger, \\
d^{(2)}_{n+1}-d^{(2)}_n&=-(H^R_{n})^{-1}\big(\alpha_{2,n}^R\big)^\dagger \alpha^L_{1,n+1}H^R_{n+1}-\alpha^L_{1,n+1}(\alpha^R_{2,n})^\dagger \alpha^L_{1,n}H^R_n -\alpha^L_{1,n+1}H^R_{n}\Big(\alpha^R_{1,1}+\sum_{m=1}^{n-1}\big(\alpha^L_{2,m}\big)^\dagger\alpha^R_{1,m+1}\Big).
\end{align*}
\end{proof}

\subsection{Pearson equations for the matrix of weights and some of its consequences}

In this subsection we analyze an important matrix $M_n(z)$ binded  to the Riemann--Hilbert problem ---as well as to  to the right logarithmic derivative of the matrix of measures---,  which is analytic in the annulus $\mathbb C \setminus \{0\}$ and provides a linear system of ordinary differential equations for the matrix Szegő polynomials and its Cauchy transforms. Moreover, it satisfies a compatibility condition with the Szegő matrix $R_n(z)$. One of its major virtues, that we will discuss in the next section, is that it leads  to non linear difference equations of Painlevé type  for the Verblunsky coefficients.

\begin{definition}
	We introduce the right logarithmic derivatives of the analytic extension to the annulus $\C\setminus \{0\}$ of matrix of measures $w(z)$ and of  $Z_n(z)$ and
	\begin{align}
		\label{eq:log-weights}
W(z)&:=\dfrac{\d w(z)}{ \d z}\big(w(z)\big)^{-1},\\
M_n(z)&:=\frac{\d Z_n(z)}{\d z} \big(Z_n(z)\big)^{-1}.\label{eq:M}
\end{align}
\end{definition}

Observe that for $n=0$ we have
\begin{align}\label{eq:M0}
M_0(z)=\PARENS{\begin{matrix}
W(z) &\dfrac{\d Q^L_{1,0}(z)}{\d z}- W(z) Q^L_{1,0}(z)\\
0_N & 0_N
\end{matrix}}.
\end{align}

Equation \eqref{eq:log-weights} can be understood as a Pearson equation for the matrix of weights:
\begin{align}\label{eq:log-weights1}
\dfrac{\d w(z)}{ \d z}=W(z)w(z).
\end{align}
This is a linear  first order differential system whose properties are determined by $W(z)$. These systems constitute a very deep and profound branch in Mathematics,  with pioneering work by George Birkhoff \cite{birkhoff}, for different treatments of the subject we refer the reader to \cite{fokas,hille,wasow,Clancey,sibuya,dragan,anosov2,ince}. 
Is relevant to remark the change of the point of view. For scalar systems with $N=1$ we normally take the weight as an explicit Freud type weight. However, for the general matrix scenario we have avoided this approach and preferred   to give $W(z)$, and consider  the extension of the matrix of weights $w(z)$ as a fundamental solution of  \eqref{eq:log-weights1}. The Freud approach will not lead, in a general scenario,  to the matrix discrete Painlevé II systems derived later in \S \ref{S:dPII}.

\begin{proposition}[Differential systems]
Equation \eqref{eq:M} can be understood as a system of differential equations and can be written in the following two alternative forms
	\begin{align}
M_n(z)&=\frac{\d Y_n(z)}{\d  z} \big(Y_n(z)\big)^{-1}+
Y_n(z)\PARENS{\begin{matrix}W(z)-{n}{I_N}z^{-1}& 0_N\\
0_N& 0_N
\end{matrix}}\big(Y_n(z)\big)^{-1}
\label{eq:matrizM},\\
&=\frac{\d X_n(z)}{\d z} \big(X_n(z)\big)^{-1}+ X_n(z) \PARENS{\begin{matrix}
W(z) & 0_N \\
0_N & -n{I_N}{z}^{-1}
\end{matrix}}\big(X_n(z)\big)^{-1}.\label{eq:matrizMS}
\end{align}
\end{proposition}

\begin{proposition}[Analytic  properties of $M_n$]\label{pro:Mana}
The matrix  $M_{n}(z)$  is analytic at the annulus $\C\setminus\{0\} $.
\end{proposition}
\begin{proof}
	It follows from \eqref{eq:matrizM} that $M_n(z)$ is analytic at $\C\setminus (\{0\}\cup\T)$. From \eqref{eq:Z} we infer that  there is no jump at $\T$ and, consequently,   $M_n(z)$ is analytic at $\C\setminus \{0\}$.
\end{proof}

          \begin{proposition}
          The following differential equations are satisfied
          	\begin{align}
          M_n(z)\PARENS{\begin{matrix}
          P^L_{1,n}(z)\\-\big(H^R_{n-1}\big)^{-1} \tilde P^R_{2,n-1}(z)
          \end{matrix}}&=
          \PARENS{\begin{matrix}
          \dfrac{\d P^L_{1,n}(z)}{\d z}-nz^{-1}P^L_{1,n}(z)+P^L_{1,n}(z) W(z)\\[8pt]
          -\big(H^R_{n-1}\big)^{-1}\Big(\dfrac{\d \tilde P^R_{2,n-1}(z)}{\d z}-nz^{-1}\tilde P^R_{2,n-1}(z)+\tilde P^R_{2,n-1}(z) W(z)\Big)
          \end{matrix}},\label{eq:diffMP}\\
          	M_n(z)\PARENS{\begin{matrix}
          	Q^L_{1,n}(z)\\
          	-\big(H^R_{n-1}\big)^{-1} Q^R_{2,n-1}(z)
          	\end{matrix}}&=\PARENS{\begin{matrix}
          	\dfrac{\d Q^L_{1,n}(z)}{\d z}\\[8pt]
          	-\big(H^R_{n-1}\big)^{-1} \dfrac{\d Q^R_{2,n-1}(z)}{\d z}
          	\end{matrix}},\label{eq:diffMQ}
          	\end{align}
          	for $n\in\{1,2,\dots\}$. For $n=0$ we have
          \begin{align*}
          M_0(z)\PARENS{\begin{matrix}
          I_N \\ 0_N
          \end{matrix}}&=\PARENS{\begin{matrix}
          W(z) \\0_N
          \end{matrix}},\\
          M_0(z)\PARENS{\begin{matrix}
          Q^L_{1,0}(z) \\ I_N
          \end{matrix}}&= \PARENS{\begin{matrix}
          \dfrac{\d Q^L_{1,0}(z)}{\d z}\\[8pt] 0_N
          \end{matrix}}.
          \end{align*}
          	
          \end{proposition}
          \begin{proof}
          	It follows directly from  \eqref{eq:M} and \eqref{eq:Z}.
          \end{proof}
Observe that \eqref{eq:diffMP} can be written
\begin{multline}\label{eq:diffMP2}
M_n(z)\PARENS{\begin{matrix}
\alpha^L_{1,n}z^{-n}+\cdots +P^L_{n,1,n-1}z^{-1}+I_N\\
-\big(H^R_{n-1}\big)^{-1}z^{-n}-\dots-\big(H^R_{n-1}\big)^{-1}\alpha^R_{2,n-1}z^{-1}
\end{matrix}}\\=\PARENS{\begin{matrix}
-n\alpha^L_{1,n}z^{-n-1}-\cdots -P^L_{n,1,n-1}z^{-2}+(\alpha^L_{1,n}z^{-n}+\cdots +P^L_{n,1,n-1}z^{-1}+I_N)W(z)\\
n\big(H^R_{n-1}\big)^{-1}z^{-n-1}+\dots+\big(H^R_{n-1}\big)^{-1}\alpha^R_{2,n-1}z^{-2}-\big(\big(H^R_{n-1}\big)^{-1}z^{-n}+\dots+\big(H^R_{n-1}\big)^{-1}\alpha^R_{2,n-1}z^{-1}\big)W(z)
\end{matrix}}.
\end{multline}
Observe  that Pearson equation \eqref{eq:log-weights1} is a first order system of ODE for the matrix of weights $w(z)$. 
\begin{definition}\label{def:typesW}
	In terms of the local behavior at  $z=0$ we distinguish three cases for the matrix of weights, two singular cases and a regular case, depending on the form of the principal part of the Laurent series of the logarithmic derivative $W(z)$ of the matrix of weights $w(z)$ in the annulus $\mathbb C\setminus\{0\}$:
	\begin{enumerate}
		\item \textbf{Ordinary case}.
			The right logarithmic derivative \eqref{eq:log-weights} of the matrix of weights is regular at the origin
			\begin{align}\label{eq:como W}
			W&=W_{0} +W_{1} z+\cdots.
			\end{align}
				\item \textbf{Fuchsian  case}.
				Now, the right logarithmic derivative \eqref{eq:log-weights} of the matrix of weights has a simple pole at the origin
				\begin{align}\label{eq:como W1}
				W&=W_{-1} z^{-1}+W_{0}+\cdots,
				\end{align}
				with $W_{-1}\neq 0_N$.
		\item \textbf{Non-Fuchsian  case}.
			The right logarithmic derivative \eqref{eq:log-weights} of the matrix of weights is the following Laurent polynomial
			\begin{align}\label{eq:como W}
			W&=W_{-r} z^{-r}+W_{-r+1} z^{-r+1}+\cdots,
			\end{align}
			with $r>1$ and $W_{-r}\neq 0_N$.
	\end{enumerate}
	We introduce the notation
	\begin{align*}
	W_{n}^{[0]}:=\ccases{ -nI_N, & \text{ordinary case,}\\
	W_{-1}-nI_N, &  \text{Fuchsian case},\\
	W_{-r}, &  \text{non-Fuchsian case, $r>1$}.
	}
	\end{align*}
\end{definition}
The equation \eqref{eq:como W} tell us that at $z=0$ we have an regular point, or a Fuchsian  singularity  or a non-Fuchsian   singularity of rank $r>1$.

Following \cite{anosov2} we say that  the point $z=0$ is is a regular singular point for \eqref{eq:log-weights1}
if there exists a constant  $k$  such that all its solutions  in every sector in the complex plane with $z=0$ as a vertex,  grow no faster than $|z|^k$ as $z\to 0$ within the sector. Fuchsian singularities are regular singularities but the converse is not always true.


\begin{proposition}\label{pro:monodrmy_free}
	In order to have a Hölder matrix of weights $w(\zeta)$, $\zeta\in\T$, the matrix $W(z)$ must be such that the Pearson system \eqref{eq:log-weights1} has trivial monodromy.
\end{proposition}	
\begin{proof}
		Given the  the monodromy matrix $M=\exp(2\pi\operatorname{i}R)$, $R\in\C^{N\times N}$,  of the system  \eqref{eq:log-weights1} a fundamental solution is of the form $S(z)=P(z) z^R$ where $P:\C\setminus\{0\}\to\operatorname{GL}(N,\C)$ is analytic; i.e., the solutions of \eqref{eq:log-weights1} are multivalued of the form $z^\alpha F(z)$, where $\exp(2\pi\operatorname{i}\alpha)$ is an eigenvalue of the monodromy matrix and $F(z)$ is analytic at the annulus  $\C\setminus\{0\}$. Thus, if we want a Hölder restriction on $\T$ ---and, consequently, single valued functions--- the only possible matrices $R$ are those with integer eigenvalues, and therefore the   monodromy matrix must be  the identity. Hence the Pearson system \eqref{eq:log-weights1} has trivial monodromy.
\end{proof}
	
	Let us notice that any equivalent system, and, therefore with the same trivial monodromy has a corresponding matrix of the form
	\begin{align*}
	\tilde W(z)=\frac{\d \Phi(z)}{\d z}\big(\Phi(z)\big)^{-1}+\Phi(z) W(z) \big(\Phi(z)\big)^{-1}
	\end{align*}
	where $\Phi:\C\setminus\{0\}\to\operatorname{GL}(N,\C)$ is analytic in the annulus with at most a pole at $z=0$.

%

		This triviality of the monodromy could be avoided if we relax the Hölder conditions on the weight, and just request piecewise Hölder weights on $\T$,  allowing at the discontinuities space for the branches of multivalued functions that non--trivial monodromy implies.  Then, we still have a RH problem but in the weak sense and uniqueness is not ensure, and the jump on $\T$ is only ensured almost everywhere. A  much more detailed analysis will be need for the analytic properties of the $R_n(z)$ and $M_n(z)$.  This could be connected with non trivial monodromy problems as there only piecewise Hölderity is required, see \cite{fokas}. For piecewise continuous jump functions see \cite{Clancey,Muskhelishvili}.
		Moreover, 	in Lemma 7.12 in \cite{deift1} we read that as long  $f\in H^1(\T)$, i.e.  $\oint_\T  \big(|f(\zeta)|^2+|f'(\zeta)|^2\big)|\d\zeta|<\infty$, where $f'$ denotes a weak derivative, the jump of its Cauchy transform satisfies $f(\zeta)=(Cf)_+(\zeta)-(Cf)_-(\zeta)$ pointwise in the unit circle $\T$, and not only almost everywhere. This, together with the proof of Theorem 7.18 in \cite{deift1}, could indicate that the RH could be generalize to more general  $H^1(\T)$-matrix of weights, and that instead of analytic extensions to the annulus $\C\setminus\{0\}$ we could deal with analytic functions on the universal cover of the annulus; i.e., with multivalued functions.

\begin{theorem}\label{theorem:lacostosa}
Let us assume  $W$  as prescribed in  in Definition \ref{def:typesW}. 	Then, the   Laurent series of $M_n(z)$ is
\begin{align}\label{eq:M_Laurent}
M_n(z)=\ccases{
M^{[0]}_n z^{-1}+M_{n,0}+M_{n,1}z+\cdots,& \text{ordinary and Fuchsian cases,}\\
M^{[0]}_n z^{-r}+M_{n,-r+1}z^{-r+1}+M_{n,-r+2}z^{-r+2}+\cdots,& \text{non-Fuchsian cases,}
}
\end{align}
where the leading coefficient is
	\begin{align}\label{eq:Wr}
	M_{n}^{[0]}&=\PARENS{\begin{matrix}
\alpha^L_{1,n}W_{n}^{[0]}\big(\alpha_{2,n}^R\big)^\dagger & -\alpha^L_{1,n}W_{n}^{[0]}H^R_{n}\\
-\big(H^R_{n-1}\big)^{-1}W_{n}^{[0]}\big(\alpha_{2,n}^R\big)^\dagger &\big(H^R_{n-1}\big)^{-1}W_{n}^{[0]}H^R_{n}
	\end{matrix}},
	\end{align}
for $n\in\{1,2,\dots\}$ and for $n=0$ we have in both singular cases (regular and irregular)
\begin{align}\label{eq:Wr0}
	M_{n}^{[0]}&=\PARENS{\begin{matrix}
	W_{0}^{[0]}& -W_{0}^{[0]}H^R_{0}\\
	0_N &0_N
	\end{matrix}}.
\end{align}
\end{theorem}

\begin{proof}
To prove \eqref{eq:Wr}  we introduce \eqref{eq:como W} into  \eqref{eq:diffMP2} and \eqref{eq:diffMQ} and then look at
 the leading part to obtain
\begin{align}\label{eq:systemP}
M_{n}^{[0]}\PARENS{\begin{matrix}
\alpha^L_{1,n}\\
-\big(H^R_{n-1}\big)^{-1}
\end{matrix}}&=\PARENS{\begin{matrix}
\alpha_{1,n}^LW_{n}^{[0]}\\-\big(H^R_{n-1}\big)^{-1}W_{n}^{[0]}
\end{matrix}},\\
M_{n}^{[0]}\PARENS{\begin{matrix}
Q^L_{1,n}(0)\\
-\big(H^R_{n-1}\big)^{-1} Q^R_{2,n-1}(0)
\end{matrix}}&=\PARENS{\begin{matrix}
0_N\\
0_N
\end{matrix}}.\label{eq:systemQ}
\end{align}
With the notation
\begin{align*}
M_{n}^{[0]}=\PARENS{\begin{matrix}
A_n & B_n\\C_n & D_n
\end{matrix}},
\end{align*}
 $A_n,B_n,C_n,D_n\in\mathbb C^{N\times N}$,
a component-wise form of the system of matrix equations \eqref{eq:systemP} and \eqref{eq:systemQ} is
        \begin{align*}
      &\begin{aligned}
   A_n \alpha^L_{1,n}-B_n\big(H^R_{n-1}\big)^{-1}&=\alpha_{1,n}W_{n}^{[0]},\\
   A_nH^L_n+	B_n\big(H^R_{n-1}\big)^{-1}\big(\alpha_{2,n}^R\big)^\dagger H^{L}_{n-1}&= 0_N,
        \end{aligned}\\
        &  \begin{aligned}
        C_n \alpha^L_{1,n}-D_n\big(H^R_{n-1}\big)^{-1}&=-(H^R_{n-1}\big)^{-1}W_{n}^{[0]},\\
        C_nH^L_n+	D_n\big(H^R_{n-1}\big)^{-1}\big(\alpha_{2,n}^R\big)^\dagger H^{L}_{n-1}&= 0_N.
        \end{aligned}
        \end{align*}
        Here we have used that, see Proposition \ref{pro:Verblusnky for ever},
        \begin{align*}
        Q^L_{1,n} (0)&=H^L_{n}, &
        Q^R_{2,n-1}(0)&= -	\big(\alpha_{2,n}^R\big)^\dagger H^{L}_{n-1}.
        \end{align*}
Consequently, cleaning $A$ and $C$ in the second equations in each of the two systems we get
\begin{align}\label{eq:ABCD..}
A_n&=-	B_n\big(H^R_{n-1}\big)^{-1}\big(\alpha_{2,n}^R\big)^\dagger H^{L}_{n-1}\big(H^L_n\big)^{-1}, &C_n&=-	D_n\big(H^R_{n-1}\big)^{-1}\big(\alpha_{2,n}^R\big)^\dagger H^{L}_{n-1}\big(H^L_n\big)^{-1},
\end{align}
that we insert in the first equation of each system to get
\begin{align*}
B_n\big(H^R_{n-1}\big)^{-1}\Big(I_N+\big(\alpha_{2,n}^R\big)^\dagger H^{L}_{n-1}\big(H^L_n\big)^{-1}\alpha^L_{1,n}\Big) &=-\alpha_{1,n}W_{n}^{[0]},\\
 D_n\big(H^R_{n-1}\big)^{-1}\Big(I_N+\big(\alpha_{2,n}^R\big)^\dagger H^{L}_{n-1}\big(H^L_n\big)^{-1}\alpha^L_{1,n}\Big) &=\big(H^R_{n-1}\big)^{-1}W_{n}^{[0]}.
\end{align*}
Let us notice  that, see  Proposition \ref{pro:Verblusnky for ever},
\begin{align*}
H^L_{n}\big(H^L_{n-1}\big)^{-1}&=I_N-\alpha^L_{1,n}\big(\alpha_{2,n}^R\big)^\dagger,
\end{align*}
that implies
\begin{align}\label{eq:HH}
H^L_{n-1}\big(H^L_{n}\big)^{-1}&=\big(I_N-\alpha^L_{1,n}\big(\alpha_{2,n}^R\big)^\dagger\big)^{-1}.
\end{align}
Then,
\begin{align*}
I_N+\big(\alpha_{2,n}^R\big)^\dagger H^{L}_{n-1}\big(H^L_n\big)^{-1}\alpha^L_{1,n}&=
I_N+\big(\alpha_{2,n}^R\big)^\dagger \big(I_N-\alpha^L_{1,n}\big(\alpha_{2,n}^R\big)^\dagger\big)^{-1}\alpha^L_{1,n}
\\&=\big(I_N-\big(\alpha_{2,n}^R\big)^\dagger\alpha^L_{1,n}\big)^{-1}.
\end{align*}
Here we have used the following fact:  given any two matrices $R,S\in\C^{N\times N}$ with  $I_N - RS$  not singular, then
\begin{align*}
I_N+S(I-RS)^{-1}R=(I_N-SR)^{-1}.
\end{align*}
We clean  $B$ and $D$ in the first step and, using \eqref{eq:ABCD..} and \eqref{eq:HH}, also $A$ and $C$. The final result is
\begin{align*}
A_n&=\alpha^L_{1,n}W_{n}^{[0]}\Big(I_N-\big(\alpha_{2,n}^R\big)^\dagger\alpha^L_{1,n}\Big)\big(\alpha_{2,n}^R\big)^\dagger \big(I_N-\alpha^L_{1,n}\big(\alpha_{2,n}^R\big)^\dagger\big)^{-1},\\
B_n&=-\alpha^L_{1,n}W_{n}^{[0]}\big(I_N-\big(\alpha_{2,n}^R\big)^\dagger\alpha^L_{1,n}\big)H^R_{n-1},\\
C_n&=-	\big(H^R_{n-1}\big)^{-1}W_{n}^{[0]}\Big(I_N-\big(\alpha_{2,n}^R\big)^\dagger\alpha^L_{1,n}\Big)\big(\alpha_{2,n}^R\big)^\dagger \big(I_N-\alpha^L_{1,n}\big(\alpha_{2,n}^R\big)^\dagger\big)^{-1},\\
D_n&=\big(H^R_{n-1}\big)^{-1}W_{n}^{[0]}\big(I_N-\big(\alpha_{2,n}^R\big)^\dagger\alpha^L_{1,n}\big)H^R_{n-1}.
\end{align*}
which taking into account  that for pair of matrices $R,S\in\mathbb C^{N\times N}$ with $\det(I_N-SR)\neq 0$ we have $(I_N-RS)R(I_N-SR)^{-1}=R$ and that $\big(I_N-\big(\alpha_{2,n}^R\big)^\dagger\alpha^L_{1,n}\big)H^R_{n-1}=H^R_n$ gives the desired result.
Finally, let us notice that \eqref{eq:Wr0} follows at once from  \eqref{eq:M0}.
\end{proof}

\begin{proposition}
The compatibility equation
	\begin{align}
	\frac{\d R_n(z)}{\d z}=M_{n+1}(z)R_n(z)-R_n(z)M_n(z)
	\label{eq:compatibilidad}
	\end{align}
	is fulfill.
\end{proposition}
\begin{proof}
	From $Z_{n+1}=R_nZ_n$ we have
	\begin{align}\label{eq:ccmpatibility}
	\frac{\d R_n(z)}{\d z} Z_n(z)+R_n(z)	\frac{\d Z_n(z)}{\d z}=	\frac{\d Z_{n+1}(z)}{\d z}
	\end{align}
	so that
	\begin{align*}
	\frac{\d R_n(z)}{\d z}+R_n	\frac{\d Z_n(z)}{\d z}\big( Z_n(z)\big)^{-1}=	\frac{\d Z_{n+1}(z)}{\d z}\big( Z_{n+1}(z)\big)^{-1}Z_{n+1}(z)\big( Z_n(z)\big)^{-1}
	\end{align*}
	and the result follows.
\end{proof}

\begin{proposition}
The following conditions are satisfied by the leading coefficients of the Szegő matrix $R_n(z)$ and the matrix $M_n(z)$
	\begin{align}\label{eq:comp_leading}
	M_{n+1}^{[0]}R_{n,-1}-R_{n,-1}M_{n}^{[0]}&=\ccases{
-R_{n,-1}, & \text{ordinary and Fuchsian cases},\\
0_{2N},  &\text{non Fuchsian cases}.
}
	\end{align}
\end{proposition}
\begin{proof}
We insert  in \eqref{eq:ccmpatibility} the Laurent series \eqref{eq:R_ver} of $R_n(z)$ and the Laurent series \eqref{eq:M_Laurent} of $M_n(z)$ and compute the leading coefficient in $z^{-r-1}$.
\end{proof}

\begin{lemma}
The dyadic or tensor type  representations
	\begin{align*}
R_{n,-1}&=\PARENS{\begin{matrix}
\alpha^L_{1,n+1}\\-\big(H^R_n\big)^{-1}
\end{matrix}} \Big( \big(\alpha^R_{2,n}\big)^\dagger, -H^R_n\Big), &
M_{n}^{[0]} &=\PARENS{\begin{matrix}
 \alpha^L_{1,n}\\-\big(H^R_{n-1}\big)^{-1}
 \end{matrix}} W_{n}^{[0]}\Big( \big(\alpha^R_{2,n}\big)^\dagger, -H^R_n\Big)
\end{align*}
hold.
\end{lemma}
\begin{proposition}
Compatibility conditions \eqref{eq:comp_leading} are identically  satisfied upon  \eqref{eq:R_ver} and \eqref{eq:Wr}.
\end{proposition}
\begin{proof}
	We first observe that
\begin{align*}
\Big( \big(\alpha^R_{2,n}\big)^\dagger, -H^R_n\Big)\PARENS{\begin{matrix}
\alpha^L_{1,n}\\-\big(H^R_{n-1}\big)^{-1}
\end{matrix}}&=\big(\alpha^R_{2,n}\big)^\dagger\alpha^L_{1,n}+H^R_{n}\big(H^R_{n-1}\big)^{-1}=I_N.
\end{align*}
Then, we calculate
	\begin{align*}
	M_{n+1}^{[0]}R_{n,-1}-R_{n,-1}M_{n}^{[0]}&=\begin{multlined}[t]
	\PARENS{\begin{matrix}
	\alpha^L_{1,n+1}\\-\big(H^R_n\big)^{-1}
	\end{matrix}} W_{n+1}^{[0]}\Big( \big(\alpha^R_{2,n+1}\big)^\dagger, -H^R_{n+1}\Big)\PARENS{\begin{matrix}
	\alpha^L_{1,n+1}\\-\big(H^R_n\big)^{-1}
	\end{matrix}} \Big( \big(\alpha^R_{2,n}\big)^\dagger, -H^R_n\Big)\\-\PARENS{\begin{matrix}
	\alpha^L_{1,n+1}\\-\big(H^R_n\big)^{-1}
	\end{matrix}} \Big( \big(\alpha^R_{2,n}\big)^\dagger, -H^R_n\Big)\PARENS{\begin{matrix}
	\alpha^L_{1,n}\\-\big(H^R_{n-1}\big)^{-1}
	\end{matrix}} W_{n}^{[0]}\Big( \big(\alpha^R_{2,n}\big)^\dagger, -H^R_n\Big)
	\end{multlined}\\
	&=\PARENS{\begin{matrix}
	\alpha^L_{1,n+1}\\-\big(H^R_n\big)^{-1}
	\end{matrix}} W_{n+1}^{[0]}\Big( \big(\alpha^R_{2,n}\big)^\dagger, -H^R_n\Big)-\PARENS{\begin{matrix}
	\alpha^L_{1,n+1}\\-\big(H^R_n\big)^{-1}
	\end{matrix}}W_{n}^{[0]}\Big( \big(\alpha^R_{2,n}\big)^\dagger, -H^R_n\Big)\\&=
		\PARENS{\begin{matrix}
	\alpha^L_{1,n+1}\\-\big(H^R_n\big)^{-1}
	\end{matrix}}
	 \big(W_{n+1}^{[0]}-W_{n}^{[0]}\big)
	\Big( \big(\alpha^R_{2,n}\big)^\dagger, -H^R_n\Big)\\
&=\ccases{
-R_{n,-1}, & \text{ordinary and Fuchsian cases},\\
0_{2N},  &\text{non-Fuchsian cases},
}
	\end{align*}
where we have used that
\begin{align*}
W_{n+1}^{[0]}-W_{n}^{[0]}=\ccases{
-I_N, & \text{ordinary and regular singular cases},\\
0_{N},  &\text{non-Fuchsian cases}.
}
\end{align*}
\end{proof}

We now consider the behavior about infinity and assume a descending Laurent series for the right logarithmic derivative
\begin{definition}\label{def:infinityW}
Let us assume that
\begin{align*}
W(z)=W_s z^s+W_{s-1} z^{s-1}+\cdots,
\end{align*}
where  $W_s$ is a non zero matrix and with the Laurent series  converging in the annulus $\C\setminus \{0\}$. Then, we  distinguish three cases
when $s\leq -2$, $s=-1$ and $s\geq 0$ and introduce the corresponding matrices
\begin{align*}
W^{[\infty]}_n:=\ccases{
-n I_N, & \text{case with $s<-1$,}\\
W_{-1}-n I_N, & \text{case with $s=-1$,}\\
W_s, &\text{case with s>-1.}
}
\end{align*}
\end{definition}

\begin{proposition}
	Assuming that the right logarithmic derivative of the matrix of weights is as in Definition \ref{def:infinityW} we have a corresponding Laurent series
	\begin{align*}
	M_n(z)=\ccases{
	M_n^{[\infty]}z^{-1} +M_{-2,n}z^{-2}+\cdots, &\text{cases with $s\leq -1$},\\
	M_n^{[\infty]}z^{s} +M_{s-1,n}z^{s-1}+\cdots, &\text{cases with $s\geq 0$,}	
	}
	\end{align*}
	where
	\begin{align*}
	M^{[\infty]}_n:=\PARENS{\begin{matrix}
	W^{[\infty]}_n &0_n\\0_N & 0_N
	\end{matrix}}.
	\end{align*}
\end{proposition}
\begin{proof}
It follows from \eqref{eq:matrizM} and the analicity of $Y_n(z)$ about infinity and its behavior normalized by the identity there, $Y_n(z)=I_{2N}+O(z^{-1})$, when $z\to\infty$.
\end{proof}
	
\section{The matrix  discrete Painlevé II  system}\label{S:dPII}
In this section we apply the Riemann--Hilbert problem and the properties of  matrix $M_n(z)$ to derive
matrix nonlinear difference systems of equations satisfied by the Verblunsky coefficients.

From \eqref{eq:matrizMS} and the form of $W(z)$ prescribed in Definition \ref{def:typesW} we get
 	\begin{align}\label{eq:Mn}
 	M_n(z)
 	&=\frac{\d X_n(z)}{\d z} \big(X_n(z)\big)^{-1}+ X_n(z) \PARENS{\begin{matrix}
W_{-r} z^{-r}+W_{-r+1} z^{-r+1}+\cdots+W_sz^s & 0_N \\
 	0_N & -n{I_N}{z}^{-1}
 	\end{matrix}}\big(X_n(z)\big)^{-1}.
 	\end{align}
 	Recall that we have the Taylor series about infinity \eqref{eq:taylorXn} and
 	\begin{align*}
 \frac{\d X_n(z)}{\d z}&=-X^{(1)}_nz^{-2}-2X^{(2)}_nz^{-3}+\cdots,&
 \big(X_n(z)\big)^{-1}&=I_{2N}-X^{(1)}_nz^{-1}-\Big(X^{(2)}_n-\big(X^{(1)}_n\big)^2\Big)z^{-2}+\cdots
 	\end{align*}
 	which converges on  the exterior of the unit circle $\bar{\mathbb D}$. 
 	Consequently, in the annulus $\bar{\mathbb D}$, we can write 
 	 	\begin{align*}
 	 	\frac{\d X_n(z)}{\d z}\big(X_n(z)\big)^{-1}&=-X^{(1)}_nz^{-2}-\big(2X^{(2)}_n-\big(X^{(1)}_n\big)^2\big)z^{-3}+\cdots.
 	 	\end{align*}
 	
 	Therefore, in the non-Fuchsian case to compute the  coefficient  $M^{[r]}_n$
 	of $M_n(z) $  we will require of the concourse of the matrices $\{X_n^{(j)}\}_{j=1}^{s+r}$. In this case the contribution of the derivative term
 	$\dfrac{\d X_n(z)}{\d z} \big(X_n(z)\big)^{-1}$ involves the coefficients $\{X^{(j)}_n\}_{j=1}^{r-1}$. For the ordinary and Fuchsian case, the derivative term do not contribute at all, and we only  need the concourse of $\{X_n^{(j)}\}_{j=1}^{s+1}$.
 	As we have seen in Proposition \ref{pro:verblusnky_parametrization} all the coefficients can be parametrized in terms of the Verblunsky matrices
 	$\big\{\alpha_{1,m}^L,\big(\alpha^R_{2,m}\big)^\dagger\big\}$, the quasi-tau matrices $\big\{H^R_m\big\}$ and the initial condition $X_{m=1}(z)$.
 	The idea is to compare this expression, obtained about infinity, with the result obtained in  Theorem  \ref{theorem:lacostosa} and, given that $M_n(z)$ is analytic in the annulus $\C\setminus\{0\}$, equate both results. As a consequence, we will have a set of four, in general  nonlinear, discrete matrix equations for
 	$\big\{\alpha_{1,m}^L,\big(\alpha^R_{2,m}\big)^\dagger\big\}_{m=0}^n$ and $\big\{H^R_m\big\}_{m=0}^n$. 	
 	 	
 	 		Notice that the term $-n I_Nz^{-1}$ appearing in \eqref{eq:Mn} will always contribute non trivially to the leading term computed  about infinity with, for example and among many others, terms of the form 
 	 		\begin{align*}
 	 		&	\Big[X_n^{(r-1)}, \PARENS{\begin{matrix}
 	 		W_{-1} & 0_N\\
 	 		0_N &-nI_N
 	 		\end{matrix}}\Big],  & &\frac{1}{2}\bigg[X^{(1)}_n,\Big[X_n^{(r-2)}, \PARENS{\begin{matrix}
 	 		W_{-1} & 0_N\\
 	 		0_N &-nI_N
 	 		\end{matrix}}\Big]\bigg],& &\frac{1}{2}\bigg[X^{(r-2)}_n,\Big[X_n^{(1)}, \PARENS{\begin{matrix}
 	 		W_{-1} & 0_N\\
 	 		0_N &-nI_N
 	 		\end{matrix}}\Big]\bigg], & &\dots
 	 		\end{align*}
 	 		even when   $W_{-1}$ does  cancel. Thus, the most simple situation with nonlinear contributions (cubic) of the Verblunsky matrices, appears when $W(z)$ has only three consecutive powers in $z$ with $z^{-1}$ among them. Namely, we are dealing with one of the following three cases
 	 		\begin{align}
 	 		W(z)&=W_{-1}z^{-1}+W_0+W_1 z, \label{eq:W-1}\\
 	 		W(z)&=W_{-2}z^{-2}+W_{-1}z^{-1}+W_0,\label{eq:W-2}\\
 	 		W(z)&=W_{-3}z^{-3}+W_{-2}z^{-2}+W_{-1}z^{-1},\label{eq:W-3}
 	 		\end{align}
 	 		While in the first case \eqref{eq:W-1} we deal with a Fuchsian singularity at $z=0$ in the two remaining cases \eqref{eq:W-2} and \eqref{eq:W-3} we have  non-Fuchsian singularities. As we will see \eqref{eq:W-1} and \eqref{eq:W-2} lead to  interesting matrix extensions of a  discrete Painlevé II system for the Verblunsky coefficients of Szegő biorthogonal polynomials on the unit circle. Despite these nonlinear equations do have non local terms, they  cancel when  the corresponding leading terms about infinity of $W(z)$ are proportional to the identity matrix. For the third case \eqref{eq:W-3}, we have found  that even in the scalar situation, $N=1$, there are non local terms. But not only, apart from the Verblunsky matrices $\big\{\big(\alpha_{2,m}^R\big)^\dagger, \alpha^L_{1,m}\big\}$ this case requires the concourse of the quasi-tau matrices $H^R_n$. We have chosen  to constrain our treatment  just to he more interesting  first two cases.

 	 	In both situations we need the following technical result
 	\begin{proposition}\label{pro:technical}
 	 				For any matrix $A\in\C^{N\times N}$ it holds that
 	 				\begin{align*}
 	 				-Ab^{(2)}_n+A(a_nb_n+b_nd_n)-a_nAb_n&=-A\big(b_{n+1}+b_nc_{n+1}b_n\big)
 	 				+[A,a_n]b_n,\\
 	 				c^{(2)}_{n}A-c_{n}Aa_{n}&=\big(c_{n-1}+c_{n}b_{n-1}c_{n}\big)A-c_{n}[A,a_{n}].
 	 				\end{align*}
 	 \end{proposition}
 
\begin{proof}
	It follows from \eqref{eq:abc} and\eqref{eq:abcd-verblusnky1-1}.
\end{proof}

\subsection{The Fuchsian case}
\subsubsection{Monodromy free condition}
	Let us consider the choice
	\begin{align}
	W(z)=W_{-1}z^{-1}+W_0+W_{1}z.
	\end{align}
		\begin{proposition}\label{pro:modormy free_regular_singular}
			For the Fuchsian case the matrix Hölder condition on the matrix of weights $w(\zeta)$, $\zeta\in\T$, requires  $W_{-1}$ to be a diagonalizable matrix with integer eigenvalues.
		\end{proposition}
	\begin{proof}
		We follow \cite{anosov2,dragan}. Let us take  take $W(z)=W_{-1}z^{-1}$ the monodromy matrix is $M=\exp(2\pi\operatorname{i}W_{-1})$, and the fundamental solution is $S(z)= z^{W_{-1}}S_0$ where $S_0\in\operatorname{GL}(N,\C )$ is a non singular matrix, see \S 2.3 of \cite{anosov2}. Hence, we have
		$w(z)= z^{W_{-1}}S_0w_0$, with $w_0\in\C^{N\times N}$. Here, $z^{W_{-1}}:=\exp(W_{-1}\log z)$ in a multivalued sense,\footnote{Note that its analytic continuation will have a logarithmic ramification ta the origin, see \cite{anosov2}} as we have that its matrix coefficients are linear combinations of multivalued functions  $z^\lambda (\log z)^k$, where $k\in\{0,1,\dots\}$ and $\lambda$ runs trough the eigenvalue of $W_{-1}$. To ensure that the corresponding matrix of weights $w(z)$ is single valued we require trivial monodromy and, therefore, $W_{-1}$ should be diagonalizable with integer eigenvalues,
			\begin{align*}
			W_{-1}&=P\Lambda_{-1} P^{-1}, &\Lambda_{-1}&=\operatorname{diag}(\lambda_1,\dots,\lambda_N), & P&\in\operatorname{GL}(N,\C),
			\end{align*}
			with $\lambda_i\in\mathbb Z$, $i\in\{1,\dots,N\}$. Indeed, if this is the case we have
			\begin{align*}
			\exp(W_{-1}\log z)&=P \
			\exp(\Lambda_{-1}\log z)P^{-1}\\
			&=P\operatorname{diag}(z^{\lambda_1},\dots,z^{\lambda_N})P^{-1},
			\end{align*}
	which is not multivalued because the eigenvalues $\lambda_i$ are integers.
	Any meromorphically  equivalent system 	corresponds to  gauge transformations of the following type
		\begin{align}\label{eq:gauge}
		\tilde W(z)=\frac{\d \Phi(z)}{\d z}\big(\Phi(z)\big)^{-1}+\Phi(z) W_{-1} z^{-1} \big(\Phi(z)\big)^{-1},
		\end{align}	
		where $\Phi:\C \setminus\{0\}\to\operatorname{GL}(N,\C)$ is analytic and has at most a  pole at the origin \cite{anosov2}.		
All equivalent systems 
		have equivalent monodromies: $\tilde M=\Phi(z_0)M\big(\Phi(z_0)\big)^{-1}$. In particular, trivial monodromy $M=I_N$ is preserved after  equivalence transformations.  Thus, all the equivalent  systems \eqref{eq:gauge}  have trivial monodromy. Following Theorem in \S 2.3 in  \cite{anosov2} we know that all systems with a regular singularity
		at the origin are equivalent to a system with $W(z)=W_{-1} z^{-1}$.
		Consequently, all the cases we consider will be of the form $\tilde W_{-1}z^{-1}+W_0+W_1z+\cdots$, where   $\tilde W_{-1}$ is diagonalizable with integer eigenvalues.
	\end{proof}
\subsubsection{Derivation of the matrix discrete Painlevé system}
	From
	\begin{align*}
	M_n(z)
	&=\frac{\d X_n(z)}{\d z} \big(X_n(z)\big)^{-1}+ X_n(z) \PARENS{\begin{matrix}
	W_{-1}z^{-1}+W_0+W_{1}z & 0_N \\
	0_N & -n{I_N}{z}^{-1}
	\end{matrix}}\big(X_n(z)\big)^{-1},
	\end{align*}
	we deduce
	\begin{multline}
	M^{[0]}_n=\PARENS{\begin{matrix}
	W_{-1}& 0_N \\
	0_N & -nI_N
	\end{matrix}}+X^{(1)}\PARENS{\begin{matrix}
	W_0 & 0_N \\
	0_N & 0_N
	\end{matrix}}-\PARENS{\begin{matrix}
	W_0 & 0_N \\
	0_N & 0_N
	\end{matrix}}X^{(1)}\\
	+X_n^{(2)}\PARENS{\begin{matrix}
	W_1 & 0_N \\
	0_N & 0_N
	\end{matrix}}-\PARENS{\begin{matrix}
	W_1 & 0_N \\
	0_N & 0_N
	\end{matrix}}X_n^{(2)}+\PARENS{\begin{matrix}
	W_1 & 0_N \\
	0_N & 0_N
	\end{matrix}}\big(X_n^{(1)}\big)^2- X_n^{(1)}\PARENS{\begin{matrix}
	W_1 & 0_N \\
	0_N & 0_N
	\end{matrix}}X_n^{(1)},\label{eq:Mrinfty1}
	\end{multline}
	which gives
	\begin{align}
	\label{eq:sistema01}
	W_{-1}+[a_n,W_{0}]+\big[a^{(2)}_n,W_1\big]+W_1\big((a_n)^2+b_nc_n\big) -a_nW_1a_n&=\alpha_{1,n}^L\big(W_{-1}-nI_N\big)\big(\alpha_{2,n}^R\big)^\dagger,\\
	\label{eq:sistema02}
	-W_{0}b_n-W_1b_n^{(2)}+W_1(a_nb_n+b_nd_n)-a_nW_1b_n&=-\alpha_{1,n}^L\big(W_{-1}-nI_N\big)H^R_n,
	\\
	\label{eq:sistema03}
	c_nW_{0}+c^{(2)}_nW_1-c_nW_1a_n&=-\big(H^R_{n-1}\big)^{-1} \big(W_{-1}-nI_N\big)\big(\alpha_{2,n}^R\big)^\dagger,\\
	\label{eq:sistema04}
	-nI_N-c_nW_1b_n&=\big(H^R_{n-1}\big)^{-1} \big(W_{-1}-nI_N\big)H^R_n.
	\end{align}
\begin{theorem}[Fuchsian matrix discrete Painlevé II system]
	When the right logarithmic derivative of the matrix of measures is $W(z)=W_{-1}z^{-1}+W_{0}+W_0z$, with $W_{-1}$ a diagonalizable matrix with entire eigenvalues so that \eqref{eq:log-weights1} is monodromy free,   the corresponding Verblunsky coefficients solve to the
	following nonlinear matrix difference equations
		\begin{multline}\label{eq:mdPII-0.1}
		W_{0}\alpha^L_{1,n}+W_1\alpha^L_{1,n+1}-\alpha^L_{1,n-1}\big(W_{-1}-(n-1)I_N\big),\\=	W_1\Big(\alpha^L_{1,n+1}\big(\alpha^R_{2,n}\big)^\dagger\alpha^L_{1,n}
		+\alpha^L_{1,n}\big(\alpha^R_{2,n-1}\big)^\dagger\alpha^L_{1,n}\Big)+\Big[W_1,\sum_{m=1}^{n-1}\alpha^L_{1,m}\big(\alpha^R_{2,m-1}\big)^\dagger\Big]\alpha^L_{1,n},	
		\end{multline}
		\vspace*{-0.5cm}
		\begin{multline}\label{eq:mdPII-0.2}			
		\big(\alpha^R_{2,n}\big)^\dagger W_{0}+\big(\alpha^R_{2,n-1}\big)^\dagger W_1- \big(W_{-1}-(n+1)I_N\big)\big(\alpha_{2,n+1}^R\big)^\dagger\\ =	\Big( \big(\alpha^R_{2,n}\big)^\dagger\alpha^L_{1,n}\big(\alpha^R_{2,n-1}\big)^\dagger +\big(\alpha^R_{2,n}\big)^\dagger \alpha^L_{1,n+1}\big(\alpha^R_{2,n}\big)^\dagger \Big)W_1 +	\big(\alpha^R_{2,n}\big)^\dagger \Big[W_1,\sum_{m=1}^{n+1}\alpha^L_{1,m}\big(\alpha^R_{2,m-1}\big)^\dagger\Big],
		\end{multline}
			where $n\in\{1,2,\dots\}$.
\end{theorem}

\begin{proof}
	Equations \eqref{eq:sistema02} and \eqref{eq:sistema03} 
	can be rewritten as follows
	\begin{align*}
	-W_{0}b_{n}-W_1\big(b_{n+1}+b_nc_{n+1}b_n\big)+\big[W_1,a_n\big]b_n&=-\alpha_{1,n}^L\big(W_{-1}-nI_N\big)H^R_n,
	\\
	c_nW_{0}+\big(c_{n-1}+c_{n}b_{n-1}c_{n}\big)W_1-c_{n}[W_1,a_{n}]&=-\big(H^R_{n-1}\big)^{-1} \big(W_{-1}-nI_N\big)\big(\alpha_{2,n}^R\big)^\dagger.
	\end{align*}	
	Using Proposition \ref{pro:technical} and equation \eqref{eq:abcd1} we find out that 
	\begin{align*}
	\begin{multlined}[t][\textwidth]
	-W_{0}\alpha_{n}-W_1\Big(\alpha^L_{n+1}H^R_{n}\big(H^R_{n-1}\big)^{-1}
	-\alpha^L_{n}\big(\alpha^R_{2,n-1}\big)^\dagger\alpha^L_{n}\Big)+\big[W_1,\sum_{m=1}^{n-1}\alpha^L_{1,m}\big(\alpha^R_{2,m-1}\big)^\dagger\big]\alpha^L_{n}\\=-\alpha_{1,n-1}^L\big(W_{-1}-(n-1)I_N\big),
	\end{multlined}
	\\
	\begin{multlined}[t][\textwidth]
	-\big(\alpha^R_{2,n}\big)^\dagger W_{0}
	-\Big(H^R_{n}\big(H^R_{n-1}\big)^{-1}\big(\alpha^R_{2,n-1}\big)^\dagger -\big(\alpha^R_{2,n}\big)^\dagger \alpha^L_{1,n+1}\big(\alpha^R_{2,n}\big)^\dagger \Big)W_1 +	\big(\alpha^R_{2,n}\big)^\dagger [W_1,\sum_{m=1}^{n+1}\alpha^L_{1,m}\big(\alpha^R_{2,m-1}\big)^\dagger]\\=- \big(W_{-1}-(n+1)I_N\big)\big(\alpha_{2,n+1}^R\big)^\dagger.
	\end{multlined}
	\end{align*}	
	That reads
	\begin{multline*}
	-W_{0}\alpha_{n}-W_1\Big(\alpha^L_{n+1}-\alpha^L_{n+1}\big(\alpha^R_{2,n}\big)^\dagger\alpha^L_{n}
	-\alpha^L_{n}\big(\alpha^R_{2,n-1}\big)^\dagger\alpha^L_{n}\Big)+\big[W_1,\sum_{m=1}^{n-1}\alpha^L_{1,m}\big(\alpha^R_{2,m-1}\big)^\dagger\big]\alpha^L_{n}\\=-\alpha_{1,n-1}^L\big(W_{-1}-(n-1)I_N\big),
\end{multline*}
\vspace*{-0.5cm}
	\begin{multline*}
	-\big(\alpha^R_{2,n}\big)^\dagger W_{0}
	-\Big(\big(\alpha^R_{2,n-1}\big)^\dagger- \big(\alpha^R_{2,n}\big)^\dagger\alpha^L_{n}\big(\alpha^R_{2,n-1}\big)^\dagger -\big(\alpha^R_{2,n}\big)^\dagger \alpha^L_{1,n+1}\big(\alpha^R_{2,n}\big)^\dagger \Big)W_1 +	\big(\alpha^R_{2,n}\big)^\dagger [W_1,\sum_{m=1}^{n+1}\alpha^L_{1,m}\big(\alpha^R_{2,m-1}\big)^\dagger]\\=- \big(W_{-1}-(n+1)I_N\big)\big(\alpha_{2,n+1}^R\big)^\dagger,
	\end{multline*}
	and the result follows.
\end{proof}

	\begin{proposition}
		The matrix discrete Painlevé II system, given by equations \eqref{eq:mdPII-0.1} and \eqref{eq:mdPII-0.1}, imply  the equations obtained from \eqref{eq:sistema01} and \eqref{eq:sistema04}  by a discrete derivative in $n$.
	\end{proposition}
	
	\begin{proof}	
		Consider the matrix discrete Painlevé II system  \eqref{eq:sistema02} and \eqref{eq:sistema03}.	
		Let us show that \eqref{eq:sistema02} \& \eqref{eq:sistema03} $\Rightarrow$ \eqref{eq:sistema01}', where \eqref{eq:sistema01}'
		refers to the  difference of the equation \eqref{eq:sistema01} at  sites $n+1$ and $n$, 
		\begin{multline*}
		\begin{multlined}[t][\textwidth]-[a_{n+1}	-a_n,W_{0}]-\big((a_{n+1}-a_n) a_n -b_nc_n\big)W_1\\-W_1\big((a_{n+1})^2+b_{n+1}c_{n+1}-(a_n)^2-b_nc_n-(a_{n+1}-a_n) a_n +b_nc_n\big) +a_{n+1}W_1a_{n+1}-a_nW_1a_n
		\end{multlined}\\=\alpha_{1,n}^L\big(W_{-1}-nI_N\big)\big(\alpha_{2,n}^R\big)^\dagger- \alpha_{1,n+1}^L\big(W_{-1}-(n+1)I_N\big)\big(\alpha_{2,n+1}^R\big)^\dagger,
		\end{multline*}
		where \eqref{eq:abcd-verblusnky1-1}  have been  used. 
		If we introduce \eqref{eq:abc} in the previous relation we get
		\begin{multline}\label{eq:primera derivada0}
		\begin{multlined}[b]
		[b_nc_{n+1},W_{0}]+\big(b_nc_{n+1}a_n +b_nc_n\big)W_1\\-W_1\big(-a_{n+1}b_nc_{n+1}+b_{n+1}c_{n+1}\big) +a_{n+1}W_1a_{n+1}-a_nW_1a_n
		\end{multlined}\\=\alpha_{1,n}^L\big(W_{-1}-nI_N\big)\big(\alpha_{2,n}^R\big)^\dagger- \alpha_{1,n+1}^L\big(W_{-1}-(n+1)I_N\big)\big(\alpha_{2,n+1}^R\big)^\dagger.
		\end{multline}
		The difference of equation \eqref{eq:sistema02}, multiplied on its  right  by $c_{n+1}$, with \eqref{eq:sistema03}, evaluated at $n+1$ and left multiplied by $b_n$ gives, once \eqref{eq:abcd-verblusnky1-1}  is used again, 
			\begin{multline*}	
		\begin{multlined}[b]
		-W_1\big(b_{n+1}-(a_{n+1}-a_n)b_n\big)c_{n+1}+\big[W_1,a_n\big]b_nc_{n+1}
		\\
	+\big[b_nc_{n+1},W_{0}\big]+b_n\big(c_{n}-c_{n+1}(a_{n+1}-a_{n})\big)W_1-b_nc_{n+1}\big[W_1,a_{n+1}\big]
		\end{multlined}\\=
		\alpha_{1,n}^L\big(W_{-1}-nI_N\big)\big(\alpha_{2,n}^R\big)^\dagger- \alpha_{1,n+1}^L\big(W_{-1}-(n+1)I_N\big)\big(\alpha_{2,n+1}^R\big)^\dagger,
		\end{multline*}
	and recalling \eqref{eq:abc} we conveniently express it as 
		\begin{multline}\label{eq:el otro lado0}
		\begin{multlined}[b]
		-W_1\big(b_{n+1}c_{n+1}-a_{n+1}b_nc_{n+1}\big)+a_nW_1(a_{n+1}-a_n)
		\\
+\big[b_nc_{n+1},W_{0}\big]+\big(b_nc_{n}+b_nc_{n+1}a_{n}\big)W_1+(a_{n+1}-a_n)W_1a_{n+1}
		\end{multlined}\\=
	\alpha_{1,n}^L\big(W_{-1}-nI_N\big)\big(\alpha_{2,n}^R\big)^\dagger- \alpha_{1,n+1}^L\big(W_{-1}-(n+1)I_N\big)\big(\alpha_{2,n+1}^R\big)^\dagger.
		\end{multline} 	
		Therefore, and comparison of \eqref{eq:primera derivada0} with \eqref{eq:el otro lado0} gives the desired result.
		
		Let us show that  \eqref{eq:sistema02} \& \eqref{eq:sistema03} $\Rightarrow$ \eqref{eq:sistema04}',
	where \eqref{eq:sistema04}' is 	the discrete derivative of \eqref{eq:sistema04}:
		\begin{align*}
		I_N+c_{n+1}W_1b_{n+1}-c_nW_1b_n&=\big(H^R_{n-1}\big)^{-1} \big(W_{-1}-nI_N\big)H^R_n-\big(H^R_{n}\big)^{-1} \big(W_{-1}-(n+1)I_N\big)H^R_{n+1}.
		\end{align*}
		Using \eqref{eq:abcd1}
	it can be written as follows
		\begin{multline*}
	I_N-(H_{n}^R)^{-1}\big(\alpha^R_{2,n}\big)^\dagger W_1\alpha^L_{1,n+2}H^R_{n+1}+(H_{n-1}^R)^{-1}\big(\alpha^R_{2,n-1}\big)^\dagger W_1\alpha^L_{1,n+1}H^R_{n}\\
		=\big(H^R_{n-1}\big)^{-1} \big(W_{-1}-nI_N\big)H^R_n-\big(H^R_{n}\big)^{-1} \big(W_{-1}-(n+1)I_N\big)H^R_{n+1},
		\end{multline*}
	so that
		\begin{multline*}
I_N-\big(\alpha^R_{2,n}\big)^\dagger W_1\alpha^L_{1,n+2}H^R_{n+1}(H_{n}^R)^{-1}+H_{n}^R(H_{n-1}^R)^{-1}\big(\alpha^R_{2,n-1}\big)^\dagger W_1\alpha^L_{1,n+1}\\
		=H_{n}^R\big(H^R_{n-1}\big)^{-1}\big(W_{-1}-nI_N\big)- \big(W_{-1}-(n+1)I_N\big)H^R_{n+1}(H_{n}^R)^{-1},
		\end{multline*}
and simplifying we arrive to 
				\begin{multline}\label{eq:la ecuacion cuarta0}
				-\big(\alpha^R_{2,n}\big)^\dagger W_1\alpha^L_{1,n+2}
				\Big(I_N-\big(\alpha^R_{2,n+1}\big)^\dagger\alpha^L_{1,n+1}\Big)				
			+\Big(I_N-\big(\alpha^R_{2,n}\big)^\dagger\alpha^L_{1,n}\Big)\big(\alpha^R_{2,n-1}\big)^\dagger W_1\alpha^L_{1,n+1}\\
				=-\big(\alpha^R_{2,n}\big)^\dagger\alpha^L_{1,n}\big(W_{-1}-nI_N\big)+\big(W_{-1}-(n+1)I_N\big)	\big(\alpha^R_{2,n+1}\big)^\dagger\alpha^L_{1,n+1}.
				\end{multline}
Equation \eqref{eq:la ecuacion cuarta0} is gotten by  multiplication	on the left of \eqref{eq:mdPII-0.1}, evaluated at site n+1,  by $\big(\alpha_{2,n}^R\big)^\dagger$ and
		on the right \eqref{eq:mdPII-0.2} by $\alpha_{1,n+1}^L$, and then taking its difference.
	\end{proof}
\subsubsection{Reduction to a linear system}

A  simplification, that leads to a linear system, is to take $W_1=0_N$, so that

  \begin{proposition}
  	For a matrix of weights with right  logarithmic derivative given by $W(z)=W_{-1}z^{-1}+W_0$, where $W_{-1}$ is a diagonalizable matrix with eigenvalues  being strictly negative integers and $W_0$ is not singular, the  Verblunsky coefficients are subject to the system of equations
  	  \begin{align}
  	  \label{eq:nlds1}  \big[\alpha^L_{1,n}(\alpha^R_{2,n-1})^\dagger,W_0 \big]&=\alpha^L_{1,n}\big(W_{-1}-nI_N\big)\big(\alpha_{2,n}^R\big)^\dagger-
  	  \alpha^L_{1,n-1}\big(W_{-1}-(n-1)I_N\big)\big(\alpha_{2,n-1}^R\big)^\dagger,\\ \label{eq:nlds2}   W_0 \alpha^L_{1,n+1}&=\alpha^L_{1,n}\big(W_{-1}-nI_N\big),\\\label{eq:nlds3}
  	  \big(\alpha^R_{2,n-1}\big)^\dagger	W_0&=  \big(W_{-1}-nI_N\big)\big(\alpha_{2,n}^R\big)^\dagger,\\\label{eq:nlds4}
  	  -n\big(I_N-\big(\alpha^R_{2,n}\big)^\dagger \alpha^L_{1,n}\big)^{-1}&=  W_{-1}-nI_N.
  	  \end{align}
    	The solution of which is
    	\begin{align}\label{eq:larepanocha}
    	\begin{aligned}
    	\alpha^L_{1,n}&=(W_0)^{-n+1}\alpha^L_{1,1}\big(W_{-1}-I_N\big)\cdots\big(W_{-1}-(n-1)I_N\big),\\
    	\big(\alpha_{2,n}^R\big)^\dagger&=   \big(W_{-1}-nI_N\big)^{-1} \cdots\big(W_{-1}-2I_N\big)^{-1}    \big(\alpha^R_{2,1}\big)^\dagger (W_0)^{n-1},
    	\end{aligned}
    	\end{align}
    	with initial values constrained by
    	\begin{align}\label{eq:laligadura}
    	\big(\alpha^R_{2,1}\big)^\dagger \alpha^L_{1,1}= W_{-1}\big(W_{-1}-I_N\big)^{-1}.
    	\end{align}
    \end{proposition}
\begin{proof}
	As preliminary condition we need to ensure that \eqref{eq:log-weights}  with $W_{-1}z^{-1}+W_0$ gives a single valued weight,  from Proposition \ref{pro:modormy free_regular_singular} we see that is indeed the case.
	To proceed, we   see that \eqref{eq:nlds1}   is a consequence of  \eqref{eq:nlds2} and \eqref{eq:nlds3}.
	The linear system given by \eqref{eq:nlds2} and \eqref{eq:nlds3}
can be written as
      \begin{align*}
      \alpha^L_{1,n}&=(W_0)^{-1}\alpha^L_{1,n-1}\big(W_{-1}-(n-1)I_N\big),\\
      \big(\alpha_{2,n}^R\big)^\dagger&=   \big(W_{-1}-nI_N\big)^{-1}     \big(\alpha^R_{2,n-1}\big)^\dagger W_0.
      \end{align*}
 Then,   a complete iteration leads to the solution \eqref{eq:larepanocha}.
Observe that  \eqref{eq:nlds4} implies that
       \begin{align*}
      \big[ W_{-1},\big(\alpha^R_{2,n}\big)^\dagger \alpha^L_{1,n}\big]=0_N&
       \end{align*}
and, consequently, we find
     \begin{align*}
     \big(\alpha_{2,n}^R\big)^\dagger\alpha^L_{1,n}&=\big(W_{-1}-nI_N\big)^{-1} \cdots\big(W_{-1}-2I_N\big)^{-1}    \big(\alpha^R_{2,1}\big)^\dagger \alpha^L_{1,1}\big(W_{-1}-I_N\big)\cdots\big(W_{-1}-(n-1)I_N\big)\\
     &=\big(W_{-1}-nI_N\big)^{-1} \big(W_{-1}-I_N\big)\big(\alpha^R_{2,1}\big)^\dagger \alpha^L_{1,1}
     \end{align*}
    so that
    \begin{align*}
      \big(W_{-1}-nI_N\big)\big(I_N-\big(\alpha_{2,n}^R\big)^\dagger\alpha^L_{1,n}\big)&=W_{-1}-nI_N-\big(W_{-1}-nI_N\big)\big(\alpha_{2,n}^R\big)^\dagger\alpha^L_{1,n}\\
      &=-nI_N+W_{-1}-\big(W_{-1}-I_N\big)\big(\alpha^R_{2,1}\big)^\dagger \alpha^L_{1,1}.
    \end{align*}
  Hence we derive the constraint \eqref{eq:laligadura} for the first Verblunsky coefficients.
\end{proof}

  \subsection{The non-Fuchsian case}
  A more involved example is given by a logarithmic derivative of the matrix of weights of the following type
  \begin{align*}
  W(z)=W_{-2}z^{-2}+W_{-1}z^{-1}+W_0.
  \end{align*}
 \subsubsection{Monodromy free and Stokes phenomena} According to \cite{fokas}, Proposition 1.1, when $W_{-2}$ is diagonalizable with $N$ different eigenvalues
 \begin{align*}
 W_{-2}&=P \Lambda_{-1} P^{-1}, & P&\in\operatorname{GL}(N,\C), &\Lambda_{-1}&=\operatorname{diag}(\alpha_1,\dots,\alpha_N), & \alpha_i&\neq \alpha_j, & i&\neq j,
 \end{align*}
 the unique formal fundamental solution of \eqref{eq:log-weights} is
 \begin{align*}
 S(z)&=P\Big(\sum_{k=0}^\infty \sigma_k z^k\Big)\operatorname{e}^{\Delta(z)},
&
 \Delta(z)&:=-\Lambda_{-1}z^{-1}+\Lambda_0 \log z+\Lambda_1+\Lambda_2  z+\dots
 \end{align*}
  with $\Lambda_k$ and $\sigma_k$ diagonal  and  off-diagonal matrices, respectively. These matrices $\sigma_n$ are determined by the equations
  \begin{align}\label{eq:Y_n}
  \Lambda_n+[\sigma_{n+1},\Lambda_{-1}]&=F_{n}, & n&\in\{0,1,2,\dots\}
  \end{align}
  where
  \begin{align*}
  F_0&=P^{-1} W_{-1} P, \\ F_1 &= P^{-1} W_0 P+ P^{-1} W_{-1} P \sigma_1- \sigma_1 W_{-1}- \sigma_1,\\
  	F_n &= P^{-1} W_0 P \sigma_{n-1}- \sigma_{n-1} \Lambda _1+P^{-1} W_{-1}P \sigma_{n}- \sigma_{n} \Lambda _0-n \sigma_n, & n\in\{2,3,\dots\}.
  	  \end{align*}
  From \eqref{eq:Y_n} we can get uniquely and recursively all terms in the formal expansion, as the operator $\text{ad}_{\Lambda_{-1}}$ is invertible in the space of off-diagonal matrices, in fact a polynomial in $\big(\text{ad}_{\Lambda_{-1}}\big|_{\text{off}}\big)^{-1}=P\big(\text{ad}_{\Lambda_{-1}}\big|_{\text{off}}\big)$. For example, $\Lambda_0=\Big(P^{-1} W_{-1} P\Big)_{\text{diag}}$ and $\sigma_1=
P\big(\text{ad}_{\Lambda_{-1}}\big|_{\operatorname{off}}\big) \Big(P^{-1} W_{-1} P\Big)_{\text{off}}$, where we are projecting in the spaces of diagonal and off-diagonal matrices.
Thus, multivaluedness of the weight $w(z)$ is formally avoided when the diagonal elements  are integers, $(P^{-1} W_{-1} P)_{i,i}\in\mathbb Z$, for $i\in\{1,\dots,N\}$.
However, the Stokes phenomena, i.e., the existence of sectors (delimited by the Stokes rays which are determined by the conditions $\operatorname{Re}((\alpha_j-\alpha_i)z^{-1})=0$) where the formal solution is asymptotic to the genuine fundamental matrix can not be avoided and a further study is needed.

\subsubsection{The non-Fuchsian matrix discrete Painlevé system and the Heisenberg algebra} We have seen that matrix discrete Painlevé II system \eqref{eq:dPII-1} and \eqref{eq:dPII-2} emerges naturally when the logarithmic derivative of the matrix of measures is $W(z)=W_{-2}z^{-2}+W_{-1}z^{-1}+W_0$.  Let us mention an explicit example of a matrix of measures whose right logarithmic derivative is of the mentioned type. However, in opposition with the previous discussion $W_{-2}$ is not diagonalizable. The matrix of weights is of Freud type and has the form
\begin{align*}
w(z)&=\exp(V(z)), & V(z)&=-W_{-2}z^{-1}+W_{0}z,
\end{align*}
where $W_{-2},W_{0}\in\C^{N\times N}$ are matrices such that with $W_{-1}=[W_{-2},W_{0}]$ conform a Heisenberg algebra
\begin{align*}
[W_{-2},W_{-1}]&=0_N &  [W_0,W_{-1}]&=0_N, & [W_{-2},W_{0}]&=W_{-1}.
\end{align*}
Indeed, we can compute the logarithmic right derivative with the aid of the formula
\begin{align*}
\frac{\d w (z)}{\d z}(w(z))^{-1}=\sum_{j=0}^\infty \frac{1}{(j+1)!}\big(\operatorname{ad}_{V(z)}\big)^j\Big(\frac{\d V(z)}{\d z}\Big),
\end{align*}
where $\operatorname{ad}_{A} (B)=[A,B]$ and get
\begin{align*}
\frac{\d w (z)}{\d z}(w(z))^{-1}=W_{-2}z^{-2}+W_{-1}z^{-1}+W_0.
\end{align*}
Observe that in this case the matrices $W_{-2}$ and $W_0$ must be nilpotent, and therefore not diagonalizable. 
\subsubsection{Derivation of the matrix discrete Painlevé II system}
We have
  \begin{align*}
  M_n(z)
  &=\frac{\d X_n(z)}{\d z} \big(X_n(z)\big)^{-1}+ X_n(z) \PARENS{\begin{matrix}
  W_{-2}z^{-2}+W_{-1}z^{-1}+W_0 & 0_N \\
  0_N & -n{I_N}{z}^{-1}
  \end{matrix}}\big(X_n(z)\big)^{-1},
  \end{align*}
  and we deduce
  \begin{align}
   M^{[0]}_n&= \begin{multlined}[t]
  -X_n^{(1)}+\PARENS{\begin{matrix}
  W_{-2}& 0_N \\
  0_N & 0_N
  \end{matrix}}+X_n^{(1)}\PARENS{\begin{matrix}
  W_{-1} & 0_N \\
  0_N & -n I_N
  \end{matrix}}-\PARENS{\begin{matrix}
  W_{-1} & 0_N \\
  0_N & - nI_N
  \end{matrix}}X_n^{(1)}\\+
  X_n^{(2)}\PARENS{\begin{matrix}
  W_0 & 0_N \\
  0_N & 0_N
  \end{matrix}}-\PARENS{\begin{matrix}
  W_0 & 0_N \\
  0_N & 0_N
  \end{matrix}}X_n^{(2)}+\PARENS{\begin{matrix}
  W_0 & 0_N \\
  0_N & 0_N
  \end{matrix}}\big(X_n^{(1)}\big)^2- X_n^{(1)}\PARENS{\begin{matrix}
  W_0 & 0_N \\
  0_N & 0_N
  \end{matrix}}X_n^{(1)}.
  \end{multlined}\label{eq:Mrinfty2}
  \end{align}
Consequently,
\begin{align}
\label{eq:sistema1}
-a_n+W_{-2}+[a_n,W_{-1}]+\big[a^{(2)}_n,W_0\big]+W_0\big((a_n)^2+b_nc_n\big) -a_nW_0a_n&=\alpha_{1,n}^LW_{-2}\big(\alpha_{2,n}^R\big)^\dagger,\\
\label{eq:sistema2}
-b_n-(W_{-1}+nI_N)b_n-W_0b_n^{(2)}+W_0(a_nb_n+b_nd_n)-a_nW_0b_n&=-\alpha_{1,n}^LW_{-2}H^R_n,
\\
\label{eq:sistema3}
-c_n+c_n(W_{-1}+nI_N)+c^{(2)}_nW_0-c_nW_0a_n&=-\big(H^R_{n-1}\big)^{-1} W_{-2}\big(\alpha_{2,n}^R\big)^\dagger,\\
\label{eq:sistema4}
-d_n-c_nW_0b_n&=\big(H^R_{n-1}\big)^{-1} W_{-2}H^R_n.
\end{align}
\begin{theorem}[A non-Fuchsian matrix discrete Painlevé II system]
	When the right logarithmic derivative of the matrix of measures is $W(z)=W_{-2}z^{-2}+W_{-1}z^{-1}+W_0$, with $W(z)$ such that  \eqref{eq:log-weights1} is monodromy free,   the corresponding Verblunsky coefficients provide solutions to the
	following nonlinear matrix difference equations
	\begin{multline}\label{eq:dPII-1}
	\big(W_{-1}+nI_N\big)\alpha^L_{1,n}+W_0\alpha^L_{1,n+1}-	\alpha_{1,n-1}^LW_{-2}\\
	=W_0\Big(\alpha^L_{1,n+1}\big(\alpha^R_{2,n}\big)^\dagger\alpha^L_{1,n}+\alpha^L_{1,n}
	(\alpha^R_{2,n-1})^\dagger\alpha^L_{1,n}\Big)+\Big[W_0,\sum_{m=1}^{n-1}\alpha^L_{1,m}\big(\alpha^R_{2,m-1}\big)^\dagger\Big]\alpha^L_{1,n},
	\end{multline}
	\vspace*{-0.5cm}
	\begin{multline}\label{eq:dPII-2}
	\big(\alpha^R_{2,n}\big)^\dagger  \big(W_{-1}+nI_N\big)+
	\big(\alpha^R_{2,n-1}\big)^\dagger W_0-W_{-2}\big(\alpha_{2,n+1}^R\big)^\dagger\\=\Big( \big(\alpha^R_{2,n}\big)^\dagger\alpha^L_{1,n}\big(\alpha^R_{2,n-1}\big)^\dagger+\big(\alpha^R_{2,n}\big)^\dagger\alpha^L_{1,n+1}(\alpha^R_{2,n})^\dagger\Big)W_0+\big(\alpha^R_{2,n}\big)^\dagger\Big[W_0,\sum_{m=1}^{n+1}\alpha^L_{1,m}\big(\alpha^R_{2,m-1}\big)^\dagger\Big],
	\end{multline}
	where $n\in\{1,2,\dots\}$.
\end{theorem}

\begin{proof}	
Proposition \ref{pro:technical} allows us for expressing  \eqref{eq:sistema2} and \eqref{eq:sistema3} as follows
	\begin{align*}
	(W_{-1}+nI_N)b_{n-1}+W_0\big(b_{n}+b_{n-1}c_{n}b_{n-1}\big)&=\alpha_{1,n-1}^LW_{-2}H^R_{n-1}+\big[W_0,a_{n-1}\big]b_{n-1},
	\\
c_{n+1}(W_{-1}+nI_N)+\big(c_{n}+c_{n+1}b_{n}c_{n+1}\big)W_0&=-\big(H^R_{n}\big)^{-1} W_{-2}\big(\alpha_{2,n+1}^R\big)^\dagger+c_{n+1}\big[W_0,a_{n+1}\big].
	\end{align*}
	Then,
		\begin{align*}
		(W_{-1}+nI_N)\alpha^L_{1,n}+W_0\alpha^L_{1,n+1}H^R_n\big(H^R_{n-1}\big)^{-1}-W_0\alpha^L_{1,n}\big(\alpha^R_{2,n-1}\big)^\dagger\alpha^L_{1,n}
	=\alpha_{1,n-1}^LW_{-2}+\big[W_0,a_{n-1}\big]\alpha^L_{1,n},
		\\
\begin{multlined}[t][\textwidth]
	\big(\alpha^R_{2,n}\big)^\dagger
	(W_{-1}+nI_N)+H_{n}^R(H_{n-1}^R)^{-1}\big(\alpha^R_{2,n-1}\big)^\dagger W_0- \big(\alpha^R_{2,n}\big)^\dagger \alpha^L_{1,n+1}\big(\alpha^R_{2,n}\big)^\dagger W_0\\= W_{-2}\big(\alpha_{2,n+1}^R\big)^\dagger+\big(\alpha^R_{2,n}\big)^\dagger\big[W_0,a_{n+1}\big],
\end{multlined}
		\end{align*}
and we get the result.
	\end{proof}

	\begin{proposition}
The matrix discrete Painlevé II system, given by equations \eqref{eq:dPII-1} and \eqref{eq:dPII-2}, imply  the equations obtained from \eqref{eq:sistema1} and \eqref{eq:sistema4}  by a discrete derivative in $n$.
	\end{proposition}
	
\begin{proof}	
We will consider the matrix discrete Painlevé II system written in the equivalent form \eqref{eq:sistema2} and \eqref{eq:sistema3}.	
	We first show the implication:  \eqref{eq:sistema2} \& \eqref{eq:sistema3} $\Rightarrow$ \eqref{eq:sistema1}'.
	The difference of the equations \eqref{eq:sistema1} at $n+1$ and $n$ gives
	\begin{multline*}
\begin{multlined}[t](	a_{n+1}	-a_n)-[a_{n+1}	-a_n,W_{-1}]-\big((a_{n+1}-a_n) a_n -b_nc_n\big)W_0\\-W_0\big((a_{n+1})^2+b_{n+1}c_{n+1}-(a_n)^2-b_nc_n-(a_{n+1}-a_n) a_n +b_nc_n\big) +a_{n+1}W_0a_{n+1}-a_nW_0a_n
\end{multlined}\\=\alpha_{1,n}^LW_{-2}\big(\alpha_{2,n}^R\big)^\dagger- \alpha_{1,n+1}^LW_{-2}\big(\alpha_{2,n+1}^R\big)^\dagger,
	\end{multline*}
	where we have used \eqref{eq:abcd-verblusnky1-1}.
If we introduce \eqref{eq:abc} in the previous relation we get
		\begin{align}\label{eq:primera derivada}
		\begin{multlined}[b]-b_nc_{n+1}+[b_nc_{n+1},W_{-1}]+\big(b_nc_{n+1}a_n +b_nc_n\big)W_0\\-W_0\big(-a_{n+1}b_nc_{n+1}+b_{n+1}c_{n+1}\big) +a_{n+1}W_0a_{n+1}-a_nW_0a_n
		\end{multlined}=\alpha_{1,n}^LW_{-2}\big(\alpha_{2,n}^R\big)^\dagger- \alpha_{1,n+1}^LW_{-2}\big(\alpha_{2,n+1}^R\big)^\dagger.
		\end{align}
The difference of equation \eqref{eq:sistema2}, multiplied on its  right  by $c_{n+1}$, with \eqref{eq:sistema3}, evaluated at $n+1$ and left multiplied by $b_n$, gives	\begin{multline*}	
	\begin{multlined}[b][0.8\textwidth]
		-W_0\big(b_{n+1}-(a_{n+1}-a_n)b_n\big)c_{n+1}+\big[W_0,a_n\big]b_nc_{n+1}
			\\
		-	b_nc_{n+1}+\big[b_nc_{n+1},W_{-1}\big]+b_n\big(c_{n}-c_{n+1}(a_{n+1}-a_{n})\big)W_0-b_nc_{n+1}\big[W_0,a_{n+1}\big]
	\end{multlined}\\=
			\alpha_{1,n}^LW_{-2}\big(\alpha^R_{2,n}\big)^\dagger-\alpha_{1,n+1}^LW_{-2}\big(\alpha_{2,n+1}^R\big)^\dagger,
			\end{multline*}
which after some cleaning and the use of \eqref{eq:abc} leads to
			\begin{multline}\label{eq:el otro lado}
		\begin{multlined}[b][0.8\textwidth]
			-W_0\big(b_{n+1}c_{n+1}-a_{n+1}b_nc_{n+1}\big)+a_nW_0(a_{n+1}-a_n)
			\\
			-	b_nc_{n+1}+\big[b_nc_{n+1},W_{-1}\big]+\big(b_nc_{n}+b_nc_{n+1}a_{n}\big)W_0+(a_{n+1}-a_n)W_0a_{n+1}
		\end{multlined}\\=
			\alpha_{1,n}^LW_{-2}\big(\alpha^R_{2,n}\big)^\dagger-\alpha_{1,n+1}^LW_{-2}\big(\alpha_{2,n+1}^R\big)^\dagger.
			\end{multline} 	
	Therefore, and comparison of \eqref{eq:primera derivada} with \eqref{eq:el otro lado} gives the desired result.
	
		We now prove the implication:  \eqref{eq:sistema2} \& \eqref{eq:sistema3} $\Rightarrow$ \eqref{eq:sistema4}'.
	The discrete derivative of \eqref{eq:sistema4}
	gives
	\begin{align*}
	d_{n+1}-d_n+c_{n+1}W_0b_{n+1}-c_nW_0b_n&=\big(H^R_{n-1}\big)^{-1} W_{-2}H^R_n-\big(H^R_{n}\big)^{-1} W_{-2}H^R_{n+1}.
	\end{align*}
	Now, using  \eqref{eq:abcd-verblusnky1-1}
	and using \eqref{eq:abcd-verblusnky1-0}, let us write this equation in an equivalent form,   
	\begin{multline*}
	-(H_{n}^R)^{-1}\big(\alpha^R_{2,n}\big)^\dagger
	\alpha^L_{1,n+1}H^R_n-(H_{n}^R)^{-1}\big(\alpha^R_{2,n}\big)^\dagger W_0\alpha^L_{1,n+2}H^R_{n+1}+(H_{n-1}^R)^{-1}\big(\alpha^R_{2,n-1}\big)^\dagger W_0\alpha^L_{1,n+1}H^R_{n}\\
	=\big(H^R_{n-1}\big)^{-1} W_{-2}H^R_n-\big(H^R_{n}\big)^{-1} W_{-2}H^R_{n+1},
	\end{multline*}
	that is
	\begin{multline*}
	-\big(\alpha^R_{2,n}\big)^\dagger
	\alpha^L_{1,n+1}-\big(\alpha^R_{2,n}\big)^\dagger W_0\alpha^L_{1,n+2}H^R_{n+1}(H_{n}^R)^{-1}+H_{n}^R(H_{n-1}^R)^{-1}\big(\alpha^R_{2,n-1}\big)^\dagger W_0\alpha^L_{1,n+1}\\
	=H_{n}^R\big(H^R_{n-1}\big)^{-1} W_{-2}-W_{-2}H^R_{n+1}(H_{n}^R)^{-1},
	\end{multline*}
and  after some cleaning reads
	\begin{multline}\label{eq:la ecuacion cuarta}
	\begin{multlined}[b]
	-\big(\alpha^R_{2,n}\big)^\dagger
	\alpha^L_{1,n+1}+\big(\alpha^R_{2,n-1}\big)^\dagger W_0\alpha^L_{1,n+1}
	-\big(\alpha^R_{2,n}\big)^\dagger\alpha^L_{1,n}\big(\alpha^R_{2,n-1}\big)^\dagger W_0\alpha^L_{1,n+1}+\big(\alpha^R_{2,n}\big)^\dagger\alpha^L_{1,n}W_{-2}.
	\end{multlined}
	\\=
	\big(\alpha^R_{2,n}\big)^\dagger W_0\alpha^L_{1,n+2}
	-\big(\alpha^R_{2,n}\big)^\dagger W_0\alpha^L_{1,n+2}\big(\alpha^R_{2,n+1}\big)^\dagger\alpha^L_{1,n+1}+W_{-2}\big(\alpha^R_{2,n+1}\big)^\dagger\alpha^L_{1,n+1}.
	\end{multline}
Now, we manipulate the  matrix discrete Painlevé system. Let us  multiply 	on the left of \eqref{eq:dPII-1} by $\big(\alpha_{2,n}^R\big)^\dagger$ and
on the right \eqref{eq:dPII-2} by $\alpha_{1,n+1}^L$, and then takes its difference to get \eqref{eq:la ecuacion cuarta}.
\end{proof}
\subsubsection{Reduction to a linear system}
A  simplification  is to  take $W_0=0_N$,  and then a new linear system for the Verblusnky coefficients appears
\begin{proposition}
	When $W(z)=W_{-2}z^{-2}+W_{-1}z^{-1}$  the Verblunsky coefficients are subject to
	\begin{align}
	-\alpha^L_{1,n}(\alpha^R_{2,n-1})^\dagger+
	[\alpha^L_{1,n}(\alpha^R_{2,n-1})^\dagger,W_{-1}]&=\alpha_{1,n}^LW_{-2}\big(\alpha_{2,n}^R\big)^\dagger-\alpha_{1,n-1}^LW_{-2}\big(\alpha_{2,n-1}^R\big)^\dagger,\label{eq:W2.1}
	\\
	(W_{-1}+(n+1)I_N)\alpha^L_{1,n+1}&=\alpha_{1,n}^LW_{-2},\label{eq:W2.2}
	\\
	\big(\alpha^R_{2,n-1}\big)^\dagger(W_{-1}+(n-1)I_N)&=W_{-2}\big(\alpha_{2,n}^R\big)^\dagger\label{eq:W2.3}\\
	\big(\alpha_{2,n-1}^R\big)^\dagger \alpha^L_{1,n}&=-W_{-2}\big(\alpha_{2,n}^R\big)^\dagger\alpha_{1,n}^L +
	\big(\alpha_{2,n-1}^R\big)^\dagger\alpha_{1,n-1}^LW_{-2}.
	\label{eq:W2.4}
	\end{align}
	whose general solution, when  $W_{-1}+nI_N$, for $n\in\{1,2,\dots\}$ and $W_{-2}$ are non singular matrices, is given by
	\begin{align}\label{eq:sol_eqW2}
	\begin{aligned}
	\alpha^L_{1,n}&=(W_{-1}+nI_N)^{-1}\cdots(W_{-1}+2I_N)^{-1}    \alpha^L_{1,1}(W_{-2})^{n-1},\\
	\big(\alpha_{2,n}^R\big)^\dagger&=(W_{-2})^{-n+1}\big(\alpha_{2,1}^R\big)^\dagger(W_{-1}+I_N)\cdots(W_{-1}+(n-1)I_N),
	\end{aligned}
	\end{align}
	in terms of the first Verblunsky coefficients $\big(\alpha_{2,1}^R\big)^\dagger$ and  $\alpha^L_{1,1}$.
\end{proposition}

\begin{proof}
	The system  \eqref{eq:sistema1}, \eqref{eq:sistema2}, \eqref{eq:sistema3} and \eqref{eq:sistema4} simplifies to
	\begin{align*}
	-a_n+W_{-2}+[a_n,W_{-1}]&=\alpha_{1,n}^LW_{-2}\big(\alpha_{2,n}^R\big)^\dagger,\\
	-(W_{-1}+(n+1)I_N)b_n&=-\alpha_{1,n}^LW_{-2}H^R_n,
	\\
	c_n(W_{-1}+(n-1)I_N)&=-\big(H^R_{n-1}\big)^{-1} W_{-2}\big(\alpha_{2,n}^R\big)^\dagger&\\
	-d_n&=\big(H^R_{n-1}\big)^{-1} W_{-2}H^R_n,
	\end{align*}
	when $W_0=0_N$.
	In this more simple case, we can rewrite the system by performing a discrete derivation (a difference)
	\begin{align*}
	-\big(a_n-a_{n-1}\big)+
	[a_n-a_{n-1},W_{-1}]&=\alpha_{1,n}^LW_{-2}\big(\alpha_{2,n}^R\big)^\dagger-\alpha_{1,n-1}^LW_{-2}\big(\alpha_{2,n-1}^R\big)^\dagger,
	\\
	(W_{-1}+(n+1)I_N)b_n&=\alpha_{1,n}^LW_{-2}H^R_n,&
	\\
	c_n(W_{-1}+(n-1)I_N)&=-\big(H^R_{n-1}\big)^{-1} W_{-2}\big(\alpha_{2,n}^R\big)^\dagger,\\
	-\big(d_n-d_{n-1}\big)&=\big(H^R_{n-1}\big)^{-1} W_{-2}H^R_n-\big(H^R_{n-2}\big)^{-1} W_{-2}H^R_{n-1}.
	\end{align*}
	Now, we use
	\begin{align*}
	a_n-a_{n-1}&=  \alpha^L_{1,n}(\alpha^R_{2,n-1})^\dagger,& b_n&=\alpha^L_{1,n+1}H^R_n\\
	c_n&= -(H_{n-1}^R)^{-1}\big(\alpha^R_{2,n-1}\big)^\dagger, & d_n-d_{n-1}&= -(H^R_{n-1})^{-1}\big(\alpha_{2,n-1}^R\big)^\dagger \alpha^L_{1,n}H^R_{n-1},
	\end{align*}
	to get
	\begin{align*}
	-\alpha^L_{1,n}(\alpha^R_{2,n-1})^\dagger+
	[\alpha^L_{1,n}(\alpha^R_{2,n-1})^\dagger,W_{-1}]&=\alpha_{1,n}^LW_{-2}\big(\alpha_{2,n}^R\big)^\dagger-\alpha_{1,n-1}^LW_{-2}\big(\alpha_{2,n-1}^R\big)^\dagger,
	\\
	(W_{-1}+(n+1)I_N)\alpha^L_{1,n+1}&=\alpha_{1,n}^LW_{-2},
	\\
	\big(\alpha^R_{2,n-1}\big)^\dagger(W_{-1}+(n-1)I_N)&=W_{-2}\big(\alpha_{2,n}^R\big)^\dagger,\\
	(H^R_{n-1})^{-1}\big(\alpha_{2,n-1}^R\big)^\dagger \alpha^L_{1,n}H^R_{n-1}&= \big(H^R_{n-1}\big)^{-1} W_{-2}H^R_n-\big(H^R_{n-2}\big)^{-1} W_{-2}H^R_{n-1},
	\end{align*}
	and \eqref{eq:W2.1},\eqref{eq:W2.2}, \eqref{eq:W2.3} and \eqref{eq:W2.4} follow.
	
	Then, we can write \eqref{eq:W2.2} and \eqref{eq:W2.3} as
	\begin{align*}
	\alpha^L_{1,n}&=(W_{-1}+nI_N)^{-1}\alpha_{1,n-1}^LW_{-2},
	\\
	\big(\alpha_{2,n}^R\big)^\dagger&=(  W_{-2} )^{-1}\big(\alpha^R_{2,n-1}\big)^\dagger(W_{-1}+(n-1)I_N),
	\end{align*}
	which iterated leads to the solution \eqref{eq:sol_eqW2}.    From here we deduce that
	\begin{align*}
	\alpha^L_{1,n}  W_{-2} \big(\alpha_{2,n}^R\big)^\dagger &=(W_{-1}+nI_N)^{-1}\cdots(W_{-1}+2I_N)^{-1}    \alpha^L_{1,1}W_{-2}\big(\alpha_{2,1}^R\big)^\dagger(W_{-1}+I_N)\cdots(W_{-1}+(n-1)I_N),\\
	\alpha^L_{1,n}  \big(\alpha_{2,n-1}^R\big)^\dagger&=(W_{-1}+nI_N)^{-1}\cdots(W_{-1}+2I_N)^{-1}    \alpha^L_{1,1}W_{-2}\big(\alpha_{2,1}^R\big)^\dagger(W_{-1}+I_N)\cdots(W_{-1}+(n-2)I_N)
	\end{align*}
	and, consequently,
	\begin{align*}
	\alpha^L_{1,n}  \big(\alpha_{2,n-1}^R\big)^\dagger (W_{-1}+(n-1)I_N)) &= \alpha^L_{1,n}  W_{-2} \big(\alpha_{2,n}^R\big)^\dagger, \\
	(W_{-1}+nI_N)    \alpha^L_{1,n}  \big(\alpha_{2,n-1}^R\big)^\dagger &= \alpha^L_{1,n-1}  W_{-2} \big(\alpha_{2,n-1}^R\big)^\dagger,
	\end{align*}
	so that
	\begin{align*}
	\alpha^L_{1,n}  \big(\alpha_{2,n-1}^R\big)^\dagger (W_{-1}+(n-1)I_N)) -  (W_{-1}+nI_N)    \alpha^L_{1,n}  \big(\alpha_{2,n-1}^R\big)^\dagger= \alpha^L_{1,n}  W_{-2} \big(\alpha_{2,n}^R\big)^\dagger-\alpha^L_{1,n-1}  W_{-2} \big(\alpha_{2,n-1}^R\big)^\dagger
	\end{align*}
	and \eqref{eq:W2.1} is  identically satisfied.
	
	Now, from \eqref{eq:sol_eqW2} we find
	\begin{align*}
	\big(\alpha_{2,n-1}^R\big)^\dagger    \alpha^L_{1,n}
	&=(W_{-2})^{-n+2}\big(\alpha_{2,1}^R\big)^\dagger(W_{-1}+I_N)(W_{-1}+(n-1)I_N)^{-1} (W_{-1}+nI_N)^{-1} \alpha^L_{1,1}(W_{-2})^{n-1},\\
	W_{-2}\big(\alpha_{2,n}^R\big)^\dagger\alpha_{1,n}^L &=(W_{-2})^{-n+2}\big(\alpha_{2,1}^R\big)^\dagger(W_{-1}+I_N)(W_{-1}+nI_N)^{-1}\alpha^L_{1,1}(W_{-2})^{n-1},\\
	\big(\alpha_{2,n-1}^R\big)^\dagger\alpha_{1,n-1}^LW_{-2} &=(W_{-2})^{-n+2}\big(\alpha_{2,1}^R\big)^\dagger(W_{-1}+I_N)(W_{-1}+(n-1)I_N)^{-1}\alpha^L_{1,1}(W_{-2})^{n-1}.
	\end{align*}
	Therefore,
	\begin{align*}
	\big(\alpha_{2,n-1}^R\big)^\dagger \alpha^L_{1,n}+W_{-2}\big(\alpha_{2,n}^R\big)^\dagger\alpha_{1,n}^L -
	\big(\alpha_{2,n-1}^R\big)^\dagger\alpha_{1,n-1}^LW_{-2}=(W_{-2})^{-n+2}\big(\alpha_{2,1}^R\big)^\dagger(W_{-1}+I_N)\mathcal W_n\alpha^L_{1,1}(W_{-2})^{n-1},
	\end{align*}
	with
	\begin{align*}
	\mathcal W_n:=(W_{-1}+(n-1)I_N)^{-1} (W_{-1}+nI_N)^{-1} +(W_{-1}+nI_N)^{-1}-(W_{-1}+(n-1)I_N)^{-1}.
	\end{align*}
	But, let us  notice that we have
	\begin{align*}
	\mathcal W_n&=(W_{-1}+(n-1)I_N)^{-1} (W_{-1}+nI_N)^{-1} \big(I_N+(W_{-1}+(n-1)I_N)-(W_{-1}+nI_N)\big)\\ &=0,
	\end{align*}
	and therefore \eqref{eq:W2.4} is also identically satisfied for Verblunsky coefficients as in \eqref{eq:sol_eqW2}.
\end{proof}

\subsection{Discussion on the matrix discrete Painlevé II systems. Locality}
We now compare the two matrix discrete Painlevé II systems we have obtained. For the reader convenience we write the equations again.
First, for the Fuchsian case $W(z)=W_{-1}z^{-1}+W_0+W_1z$,    the  matrix discrete Painlevé II system, given in \eqref{eq:mdPII-0.1} and \eqref{eq:mdPII-0.2},   is
		\begin{multline*}
		W_{0}\alpha^L_{1,n}+W_1\alpha^L_{1,n+1}-\alpha^L_{1,n-1}\big(W_{-1}-(n-1)I_N\big)\\=	W_1\Big(\alpha^L_{1,n+1}\big(\alpha^R_{2,n}\big)^\dagger\alpha^L_{1,n}
		+\alpha^L_{1,n}\big(\alpha^R_{2,n-1}\big)^\dagger\alpha^L_{1,n}\Big)+\Big[W_1,\sum_{m=1}^{n-1}\alpha^L_{1,m}\big(\alpha^R_{2,m-1}\big)^\dagger\Big]\alpha^L_{1,n},	
		\end{multline*}
		\vspace*{-0.5cm}
		\begin{multline*}
		\big(\alpha^R_{2,n}\big)^\dagger W_{0}+\big(\alpha^R_{2,n-1}\big)^\dagger W_1- \big(W_{-1}-(n+1)I_N\big)\big(\alpha_{2,n+1}^R\big)^\dagger\\ =	\Big( \big(\alpha^R_{2,n}\big)^\dagger\alpha^L_{1,n}\big(\alpha^R_{2,n-1}\big)^\dagger +\big(\alpha^R_{2,n}\big)^\dagger \alpha^L_{1,n+1}\big(\alpha^R_{2,n}\big)^\dagger \Big)W_1 +	\big(\alpha^R_{2,n}\big)^\dagger \Big[W_1,\sum_{m=1}^{n+1}\alpha^L_{1,m}\big(\alpha^R_{2,m-1}\big)^\dagger\Big].
		\end{multline*}
Second, for the non-Fuchsian case $W(z)=W_{-2}z^{-2}+W_{-1}z^{-1}+W_0$,    the  matrix discrete Painlevé II system, given in  \eqref{eq:dPII-1} and \eqref{eq:dPII-2}, reads
	\begin{multline*}
	\big(W_{-1}+nI_N\big)\alpha^L_{1,n}+W_0\alpha^L_{1,n+1}-	\alpha_{1,n-1}^LW_{-2}\\
	=W_0\Big(\alpha^L_{1,n+1}\big(\alpha^R_{2,n}\big)^\dagger\alpha^L_{1,n}+\alpha^L_{1,n}
	(\alpha^R_{2,n-1})^\dagger\alpha^L_{1,n}\Big)+\Big[W_0,\sum_{m=1}^{n-1}\alpha^L_{1,m}\big(\alpha^R_{2,m-1}\big)^\dagger\Big]\alpha^L_{1,n},
	\end{multline*}
	\vspace*{-0.5cm}
	\begin{multline*}
	\big(\alpha^R_{2,n}\big)^\dagger  \big(W_{-1}+nI_N\big)+
	\big(\alpha^R_{2,n-1}\big)^\dagger W_0-W_{-2}\big(\alpha_{2,n+1}^R\big)^\dagger\\=\Big( \big(\alpha^R_{2,n}\big)^\dagger\alpha^L_{1,n}\big(\alpha^R_{2,n-1}\big)^\dagger+\big(\alpha^R_{2,n}\big)^\dagger\alpha^L_{1,n+1}(\alpha^R_{2,n})^\dagger\Big)W_0+\big(\alpha^R_{2,n}\big)^\dagger\Big[W_0,\sum_{m=1}^{n+1}\alpha^L_{1,m}\big(\alpha^R_{2,m-1}\big)^\dagger\Big].
	\end{multline*}
		
We see that they are almost the same system. In fact, the nonlinear term are in complete correspondence by $W_1\to W_0$. However, the linear terms are not.  For example, we have terms like $(n-1)\alpha^L_{1,n-1}$ for the Fuchsian scenario and of the
 form $n\alpha^L_{1,n}$ for the non-Fuchsian one, spoiling a complete correspondence.

\subsubsection{Local matrix discrete Painlevé II systems}
These  matrix discrete Painlevé II systems  present cubic terms in the Verblunsky coefficients $\big(\alpha_{2,n}^R\big)^\dagger$ and $\alpha_{1,n}^L$, being all these terms  local, in the sense that they involve nearby neighbors (the Verblunsky matrices at the sites $n-1$, $n$ and $n+1$), but for the last commutator in the RHS. 
\begin{definition}
	In the matrix discrete Painlevé  II systems given in \eqref{eq:mdPII-0.1} and \eqref{eq:mdPII-0.2}, or in \eqref{eq:dPII-1} and \eqref{eq:dPII-2} we call non local terms those  terms of the form
	\begin{align*}
\sum_{m=1}^{n\pm 1}\alpha^L_{1,m}\big(\alpha^R_{2,m-1}\big)^\dagger.
	\end{align*}
	When these terms are absent we say that we have local matrix discrete Painlevé II systems. These systems are 
			\begin{align}\label{eq:mdPII-0.1-local}
			W_{0}\alpha^L_{1,n}+W_1\alpha^L_{1,n+1}-\alpha^L_{1,n-1}\big(W_{-1}-(n-1)I_N\big)&=	W_1\Big(\alpha^L_{1,n+1}\big(\alpha^R_{2,n}\big)^\dagger\alpha^L_{1,n}
			+\alpha^L_{1,n}\big(\alpha^R_{2,n-1}\big)^\dagger\alpha^L_{1,n}\Big),\\	
	\label{eq:mdPII-0.2-local}			
			\big(\alpha^R_{2,n}\big)^\dagger W_{0}+\big(\alpha^R_{2,n-1}\big)^\dagger W_1- \big(W_{-1}-(n+1)I_N\big)\big(\alpha_{2,n+1}^R\big)^\dagger&=	\Big( \big(\alpha^R_{2,n}\big)^\dagger\alpha^L_{1,n}\big(\alpha^R_{2,n-1}\big)^\dagger +\big(\alpha^R_{2,n}\big)^\dagger \alpha^L_{1,n+1}\big(\alpha^R_{2,n}\big)^\dagger \Big)W_1,
			\end{align}
	in the Fuchsian case,	while for the non-Fuchsian case they are
			\begin{align}\label{eq:dPII-1-local}
			\big(W_{-1}+nI_N\big)\alpha^L_{1,n}+W_0\alpha^L_{1,n+1}-	\alpha_{1,n-1}^LW_{-2}
			&=W_0\Big(\alpha^L_{1,n+1}\big(\alpha^R_{2,n}\big)^\dagger\alpha^L_{1,n}+\alpha^L_{1,n}
			(\alpha^R_{2,n-1})^\dagger\alpha^L_{1,n}\Big),\\
			\big(\alpha^R_{2,n}\big)^\dagger  \big(W_{-1}+nI_N\big)+
			\big(\alpha^R_{2,n-1}\big)^\dagger W_0-W_{-2}\big(\alpha_{2,n+1}^R\big)^\dagger&=\Big( \big(\alpha^R_{2,n}\big)^\dagger\alpha^L_{1,n}\big(\alpha^R_{2,n-1}\big)^\dagger+\big(\alpha^R_{2,n}\big)^\dagger\alpha^L_{1,n+1}(\alpha^R_{2,n})^\dagger\Big)W_0.
			\end{align}
\end{definition}

\paragraph{\textbf{Examples of local matrix discrete Painlevé II systems}}We now discuss some cases where we find local matrix discrete Painlevé II systems:
\begin{enumerate}
	\item If we take, in each case, $W_1=k_1 I_N$ (Fuchsian) or $W_0=k_0 I_N$ (non-Fuchsian) with $k_1,k_0\in\C$ the systems are  local and read
		\begin{align*}
		W_{0}\alpha^L_{1,n}+k_1\alpha^L_{n+1}-\alpha^L_{1,n-1}\big(W_{-1}-(n-1)I_N\big)&=	k_1\Big(\alpha^L_{1,n+1}\big(\alpha^R_{2,n}\big)^\dagger\alpha^L_{1,n}
		+\alpha^L_{1,n}\big(\alpha^R_{2,n-1}\big)^\dagger\alpha^L_{1,n}\Big),\\	
		\big(\alpha^R_{2,n}\big)^\dagger W_{0}+k_1\big(\alpha^R_{2,n-1}\big)^\dagger - \big(W_{-1}-(n+1)I_N\big)\big(\alpha_{2,n+1}^R\big)^\dagger&=	k_1\Big( \big(\alpha^R_{2,n}\big)^\dagger\alpha^L_{1,n}\big(\alpha^R_{2,n-1}\big)^\dagger +\big(\alpha^R_{2,n}\big)^\dagger \alpha^L_{1,n+1}\big(\alpha^R_{2,n}\big)^\dagger \Big),
		\end{align*}
and
	\begin{align*}
	\big(W_{-1}+nI_N\big)\alpha^L_{1,n}+k_0\alpha^L_{1,n+1}-	\alpha_{1,n-1}^LW_{-2}
	&=k_0\Big(\alpha^L_{1,n+1}\big(\alpha^R_{2,n}\big)^\dagger\alpha^L_{1,n}+\alpha^L_{1,n}
	(\alpha^R_{2,n-1})^\dagger\alpha^L_{1,n}\Big),\\
	\big(\alpha^R_{2,n}\big)^\dagger  \big(W_{-1}+nI_N\big)+
	k_0\big(\alpha^R_{2,n-1}\big)^\dagger -W_{-2}\big(\alpha_{2,n+1}^R\big)^\dagger&=k_0\Big( \big(\alpha^R_{2,n}\big)^\dagger\alpha^L_{1,n}\big(\alpha^R_{2,n-1}\big)^\dagger+\big(\alpha^R_{2,n}\big)^\dagger\alpha^L_{1,n+1}(\alpha^R_{2,n})^\dagger\Big),	
	\end{align*}
	respectively.
\item Two examples with locality, and in which we can ensure that we have an appropriate matrix of measures	follow. \begin{enumerate}
	\item A first one is  to take $W_{-1}=k_{-1}I_N$, $k_{-1}\in\mathbb Z$, then the matrix of weights
\begin{align*}
w(z)=z^{k_{-1}}\exp(k_0z)\exp(-W_{-2}z^{-1})
\end{align*}
  leads to the following matrix discrete Painlevé II system
	\begin{align*}
	(k_{-1}+n)\alpha^L_{1,n}+k_0\alpha^L_{1,n+1}-	\alpha_{1,n-1}^LW_{-2}
	&=k_0\Big(\alpha^L_{1,n+1}\big(\alpha^R_{2,n}\big)^\dagger\alpha^L_{1,n}+\alpha^L_{1,n}
	(\alpha^R_{2,n-1})^\dagger\alpha^L_{1,n}\Big),\\
	(k_{-1}+n)\big(\alpha^R_{2,n}\big)^\dagger +
	k_0\big(\alpha^R_{2,n-1}\big)^\dagger -W_{-2}\big(\alpha_{2,n+1}^R\big)^\dagger&=k_0\Big( \big(\alpha^R_{2,n}\big)^\dagger\alpha^L_{1,n}\big(\alpha^R_{2,n-1}\big)^\dagger+\big(\alpha^R_{2,n}\big)^\dagger\alpha^L_{1,n+1}(\alpha^R_{2,n})^\dagger\Big).	
	\end{align*}
	 \item A second one is
	 \begin{align*}
	 w(z)=z^{k_{-1}}\exp(k_1z^2/2)exp(W_0z)
	 \end{align*}
	with $k_{-1}\in\mathbb Z$, and the corresponding matrix discrete Painlevé II equations are
		\begin{align*}
		W_{0}\alpha^L_{1,n}+k_1\alpha^L_{1,n+1}-\alpha^L_{1,n-1}\big(k_{-1}-(n-1)I_N\big)&=	k_1\Big(\alpha^L_{1,n+1}\big(\alpha^R_{2,n}\big)^\dagger\alpha^L_{1,n}
		+\alpha^L_{1,n}\big(\alpha^R_{2,n-1}\big)^\dagger\alpha^L_{1,n}\Big),\\	
		\big(\alpha^R_{2,n}\big)^\dagger W_{0}+k_1\big(\alpha^R_{2,n-1}\big)^\dagger - \big(k_{-1}-(n+1)I_N\big)\big(\alpha_{2,n+1}^R\big)^\dagger&=	k_1\Big( \big(\alpha^R_{2,n}\big)^\dagger\alpha^L_{1,n}\big(\alpha^R_{2,n-1}\big)^\dagger +\big(\alpha^R_{2,n}\big)^\dagger \alpha^L_{1,n+1}\big(\alpha^R_{2,n}\big)^\dagger \Big).
		\end{align*}
\end{enumerate}
\item There are more  possibilities for local  matrix discrete Painlevé II systems. For example, take the following matrices of weights
\begin{align*}
w(z)&=z^{k_{-1}}\exp(-k_{-2}z^{-1})\exp(W_0z) &k_{-1}&\in\mathbb Z, &k_{-2}&\in\C,  &W_0&\in\C^{N\times N}, \\
w(z)&=z^{k_{-1}}\exp(k_0z)\exp(W_1z^2/2),&k_{-1}&\in\mathbb Z,& k_{0}&\in\C, & W_1&\in\C^{N\times N},\\
w(z)&=z^{k_{-1}}\exp(W_0(az^{-1}+bz)), &k_{-1}&\in\mathbb Z, &a,b&\in\C,  &W_0&\in\C^{N\times N}, \\
w(z)&=z^{k_{-1}}\exp(W_1(az+bz^2)),&k_{-1}&\in\mathbb Z, &a,b&\in\C, & W_1&\in\C^{N\times N}.
\end{align*}
This is so because in these four cases   the commutativity  $[w(z_1),w(z_2)]=0_N$ ,		for all $z_1,z_2\in\C$, holds and, consequently,  Proposition \ref{pro:commutativity} is applicable.\footnote{We can argue also saying that $[w(z),W_0]=0_N$, $ \forall z\in\C$, in the first and third cases, and 
	$[w(z),W_1]=0_N$, $ \forall z\in\C$, for the second and fourth cases, and we can apply Proposition \ref{pro:commutativity0}. } Hence, the non local terms disappear.
\item A more general scenario were these symmetry considerations are applicable, and the non local terms are set off,  are
\begin{enumerate}
\item For the Fuchsian case we take the triple $W_{-1},W_{0},W_{1}\in\C^{N\times N}$ in an Abelian algebra, and $W_{-1}$ a diagonalizable matrix with integer eigenvalues.
\item Similarly, for the non-Fuchsian case we choose triple $W_{-2},W_{-1},W_{0}\in\C^{N\times N}$ in an Abelian algebra and $W_{-1}$ a diagonalizable matrix with integer eigenvalues.
\end{enumerate}
\end{enumerate}
\begin{theorem}[Local matrix discrete Painlevé II systems]
The matrix discrete Painlevé II systems  are local whenever
\begin{enumerate}
	\item \textbf{Fuchsian case}:  we choose the triple  of matrices $\{W_{-1},W_0,W_1\}$ such that $[W_1,W_0]=[W_1,W_{-1}]=0_N$, 
	and such  that for some nonzero complex number  $z_0$ we have the commutativity  $[W_1,w(z_0)]=0_N$.
		The matrix of weights will have the form
		\begin{align*}
		w(z)=\exp(W_1z^2/2)\tilde w(z)
		\end{align*}
		where the associated matrix of weights $\tilde w(z)$ is the unique solution  to
		\begin{align*}
		\frac{\d\tilde w}{\d z}&=(W_{-1}z^{-1}+W_{0})\tilde w, & \tilde w(z)=&w(z_0).
		\end{align*}
	\item \textbf{Non-Fuchsian case}: we take  the triple  of matrices $\{W_{-2},W_{-1},W_0\}$ such that $[W_0,W_{-1}]=[W_1,W_{-2}]=0_N$, 
	and such  that for some nonzero complex number $z_0$ we have the commutativity  $[W_0,w(z_0)]=0_N$. 
	The matrix of weights will have the form
	\begin{align*}
	w(z)=\exp(W_0z)\tilde w(z)
	\end{align*}
	where the associated matrix of weights $\tilde w(z)$ is the unique solution  to
	         \begin{align*}
	         \frac{\d\tilde w}{\d z}&=(W_{-2}z^{-2}+W_{-1}z^{-1})\tilde w, & \tilde w(z)=&w(z_0).
	         \end{align*}
	
\end{enumerate}
\end{theorem}
\begin{proof}
 We prove the Proposition just for  the Fuchsian case; in the non-Fuchsian case the proof  goes analogously. 
If $W_1$   commutes with the matrix of weights $w(z)$, Proposition \ref{pro:commutativity0} ensures that it commutes with the Verblunsky matrices, and the locality is achieved. But, $[W_1,W_0]=[W_1,W_{-1}]=0_N$ are equivalent to $[W_1,W(z)]=0_N$ for all $z\in\mathbb C$.
Now, the  Picard's method of successive approximations  completes the argument,  see  \cite{hille}. Indeed, for a non singular point $z_0$, i.e., any non zero complex number, the matrix of weights $w(z)$ which satisfies \eqref{eq:log-weights1} with $w(z_0)=w_0$  is the solution to the integral equation
\begin{align*}
w(z)=w_0+\int_{z_0}^z W(s)w(s)\d s.
\end{align*}
Then, this solution can be obtained by the Picard  iteration method 
\begin{align*}
w_n(z)&:=w_0+\int_{z_0}^z W(s)w_{n-1}(s)\d s,& n\in\{1,2,\dots\}
\end{align*}
as the limit
\begin{align*}
w_n(z)\underset{n\to\infty}{\longrightarrow} w(z).
\end{align*}
Consequently,  if a matrix  $M$ commutes with $W(z)$ and the initial condition $w_0$ it also commutes with $w_n(z)$, for all $n$, and, therefore, with the limit $w(z)=\lim\limits_{n\to\infty}w_n(z)$. 
Finally, the gauge transformation $w(z)\to \tilde w(z):=\exp(-W_1z^2/2)w(z)$, reads
\begin{align*}
\frac{\d\tilde w(z)}{\d z}\big(\tilde w(z)\big)^{-1}&=-W_1z+\exp(-W_1z^2/2)\frac{\d w(z)}{\d z}\big( w(z)\big)^{-1}\exp(W_1z^2/2)\\
&=-W_1z+\exp(-W_1z^2/2)W(z)\exp(W_1z^2/2),\\
&=W_0+W_1z.
\end{align*}

\end{proof}
One could think that this condition leads to the Abelian triple discussed in the preliminary examples of local systems. But $[W_1,W_0]=[W_1,W_{-1}]=0_N$ do not imply that $[W_0,W_{-1}]=0_N$ that in turn  gives, for the Fuchsian situation, the following form for  the measure  $w(z)=\exp(W_{-1}\log z+W_0z+W_1 z^2/2)$. Now,  we give two examples that are not an Abelian triple for the Fuchsian case (the non-Fuchsian case  goes similarly). 
\begin{enumerate}
	\item For $N=2$ we choose
	\begin{align*}
	W_{-1}&=\PARENS{\begin{matrix}
	m_1 &0\\
	0 &m_2
	\end{matrix}}, & W_0&=\PARENS{\begin{matrix}
	a & b\\
	c & d
	\end{matrix}}, &
	W_1&=k_1I_2
	\end{align*}
	where $m_1,m_2\in\mathbb Z$ and $a,b,c,d,k\in\C $. This example  fits in the first family we considered with $W_1$ proportional to the identity.
	\item Much more interesting is the following choice for $N=3$. We pick
	\begin{align*}
		W_{-1}&=\PARENS{\begin{matrix}
		m_1 &0 & 0\\
		0 &m_1 &0\\
		0 & 0 &m_2
		\end{matrix}}, & W_0&=\PARENS{\begin{matrix} e& 0 &0\\
		0 &		a & b\\
		0& c & d
		\end{matrix}}, &
		W_1&=\PARENS{\begin{matrix}
		k_1 &0 & 0 \\
		0 &k_2 &0\\
		0 & 0 &k_2
		\end{matrix}},
	\end{align*}
	with $m_1,m_2\in\mathbb Z$ and $a,b,c,d,e,k_1,k_2\in\C$. 
	Here $W_1$ is not proportional to the identity but commutes with $W_{-1}$ and $W_0$, and $[W_{-1},W_0]\neq 0_3$. 
	Obviously, this example can be generalized. These type of choices are  relevant because  fro then the commutativity $[w(z),w(z')]=0_N$ for all $z,z'\in\C\setminus \{0\}$ does not necessarily holds. Hence,  the  matrix Szegő polynomials and Verblunsky matrices do not necessarily form an Abelian set, which indeed happens if $[W_0,W{-1}]=0_N$. 
\end{enumerate}

\subsubsection{The scalar case}		
Finally, let us mention that in the scalar case, when $N=1$, right and left indexes disappear and we get
the following systems of equations
	\begin{align*}
	k_{0}\alpha_{1,n}+k_1\alpha_{1,n+1}-\alpha_{1,n-1}\big(k_{-1}-(n-1)I_N\big)&=	k_1\Big(\alpha_{1,n+1}\big(\alpha_{2,n}\big)^\dagger\alpha_{1,n}
	+\alpha_{1,n}\big(\alpha_{2,n-1}\big)^\dagger\alpha_{1,n}\Big),\\	
	k_0\big(\alpha_{2,n}\big)^\dagger +k_1\big(\alpha_{2,n-1}\big)^\dagger - \big(k_{-1}-(n+1)I_N\big)\big(\alpha_{2,n+1}\big)^\dagger&=	k_1\Big( \big(\alpha_{2,n}\big)^\dagger\alpha_{1,n}\big(\alpha_{2,n-1}\big)^\dagger +\big(\alpha_{2,n}\big)^\dagger \alpha_{1,n+1}\big(\alpha_{2,n}\big)^\dagger \Big),
	\end{align*}
	for the Fuchsian case, which correspond to the complex weight
	\begin{align*}
	w(\zeta)&=\zeta^{k_{-1}}\operatorname{e}^{k_0\zeta+\frac{k_1}{2}\zeta^2}, &k_{-1}&\in\mathbb Z,& k_0,k_1&\in\mathbb C,& \zeta\in\T.
	\end{align*}
	This weight can not be reduced to a real weight, consequently this is a purely biorthogonal result as it not associated with any orthogonal reduction.
	For the non-Fuchsian case the system is
 	\begin{align*}
 	\big(k_{-1}+n\big)\alpha_{1,n}+k_0\alpha_{1,n+1}-	k_{-2}\alpha_{1,n-1}
 	&=k_0\big(\alpha_{1,n+1}\alpha_{2,n}\alpha_{1,n}+
 	\alpha_{2,n-1}\big(\alpha_{1,n}\big)^2\big),\\
 \big(k_{-1}+n\big)	\alpha_{2,n} +
 	k_0\alpha_{2,n-1} -k_{-2}\alpha_{2,n+1}&=k_0\big( \alpha_{2,n}\alpha_{1,n}\alpha_{2,n-1}+\alpha_{1,n+1}\big(\alpha_{2,n}\big)^2\big),	
 	\end{align*}
 	which corresponds to the complex measure
  \begin{align*}
  w(\zeta) &= \zeta^{k_{-1}}\operatorname{e}^{-k_{-2}\zeta^{-1}+k_0\zeta},&k_{-1}&\in\mathbb Z,& k_0,k_{-2}&\in\mathbb C,& \zeta\in\T.
  \end{align*}
Now, we do have a real reduction, with $k_0=-k_{-2}=k\in \mathbb R$, $k_{-1}=0$,  and  $\alpha_{1,n}=\alpha_{2,n}=\alpha_n$. Consequently, biothogonality reduces to orthogonality and we get the well known dPII equation
	\begin{align*}
	n\frac{\alpha_{n}}{1-\big(\alpha_{n}\big)^2}+k(\alpha_{n+1}+\alpha_{n-1})	&=0.
	\end{align*}
This equation appeared in the context of  unitary matrix models, see \cite{Periwal0,Periwal}, and in the study self-similarity reductions of the mKdV equation, see \cite{nijhoff}. For a discussion of some interesting properties of this modified Bessel OPUC see \cite{ismail} and for the corresponding integrable hierarchy see \cite{cresswell}.
Notice that the Fuchsian case do not have any real reduction, and consequently we have just biorthogonal
Szegő polynomials.  
\settocdepth{section}
\appendix

\section{Fuchsian  examples of matrices of weights for $N=2$ }	\label{S:examples}
We will study some examples for the Fuchsian case and $N=2$ given by
\begin{align*}
W_{-1}=\PARENS{\begin{matrix}	
	p & 0\\
	0 & p+k
	\end{matrix}
	}
\end{align*}
with $p\in\mathbb Z$ and $k\in \{0,1,\dots\}$.\footnote{ For $W_{-1}=\PARENS{\begin{smallmatrix}	
		p+k & 0\\
		0 & p
		\end{smallmatrix}
	}$ a similar discussion can be carried out.} 
Observe that for $k=\{1,2,\dots\}$ we are dealing with  resonant cases. With the aid of  \cite{wasow} we explore the cases $k=1,2,3$ and get  explicit examples of matrices of weights linked to the original assumption for the right logarithmic derivative. We obtain that if the coefficients of $W(z)$ lay  in certain algebraic hypersurface  of degree $k$  in $\C^8$, the corresponding linear ODE system \eqref{eq:log-weights1} has trivial monodromy.
This construction works for larger $k$; we will obtain higher degree algebraic hypersurfaces for the constraints ensuring single-valuedness. Using the techniques of \cite{wasow,fokas} we can apply these ideas  to the general situation $N\geq 3$.
\subsection{$N=2$ and $k=1$}
Let us assume that
\begin{align*}
W_{-1}=\PARENS{
\begin{matrix}	p & 0\\
	0 & p+1
	\end{matrix}
}
\end{align*}
with $p\in\mathbb Z$,  so that
\begin{align*}
W(z)=\PARENS{\begin{matrix}
	p & 0\\
	0 & p+1
	\end{matrix}}z^{-1} +\PARENS{\begin{matrix}
	a_0 +a_1 z & b_0 + b_1 z\\
	c_0+c_1 z & d_0+d_1 z
	\end{matrix}}
\end{align*}
for some  $(a_0,b_0,c_0,d_0,a_1,b_1,c_1,d_1)^\top\in\C^8$.
We will study the the corresponding matrix of weights $w(z)$, which is a solution of \eqref{eq:log-weights1}.
An equivalent system, obtained by a shearing transformation generated by 
\begin{align}\label{eq:shearing}
\mathcal S(z)&=\PARENS{ 
\begin{matrix}
	1 & 0\\
	0 &z^{-1}
\end{matrix}
	},
\end{align}
is given by
\begin{align}\notag
W^{(1)}(z)&=\frac{\d \mathcal S(z)}{\d z}\big(\mathcal S(z)\big)^{-1}+\mathcal S(z)W(z) \big(\mathcal S(z)\big)^{-1}\\
&=\PARENS{\begin{matrix}
	p & 0\\
	c_0 & p
	\end{matrix}}z^{-1}+ W^{(1)}_0 + W^{(1)}_1 z +W^{(1)}_2 z^2, &
W^{(1)}_0&=\PARENS{\begin{matrix}
	a_0& 0\\
	c_1 & d_0
	\end{matrix}}, & W^{(1)}_1&:=\PARENS{\begin{matrix}
	a_1 & b_0 \\
	0& d_1 
	\end{matrix}}, &W^{(1)}_2&:=\PARENS{\begin{matrix}
	0& b_1 \\
	0& 0
	\end{matrix}}. \label{eq:As}
\end{align}
This case is not resonant anymore, we have get rid of it using  the shearing transformation. Therefore, a fundamental  solution to the  associated differential system  
\begin{align}\label{eq:log_ejem}
\frac{\d w^{(1)}(z)}{\d z}=W^{(1)}(z)w^{(1)}(z)
\end{align}
is   of the form
$w^{(1)}(z)=\Phi(z)z^{\begin{psmallmatrix}
	p & 0 \\c_0 &p
	\end{psmallmatrix}}$
where $\Phi$ is analytic with a   Taylor series convergent at the annulus $\C\setminus\{0\}$
\begin{align*}
\Phi(z)=I_N+\Phi_1 z+\Phi_2z^2+\cdots.
\end{align*}
Thus, to avoid multivalued functions we require 
\begin{align*}
c_0=0, 
\end{align*}
so that the fundamental  solution will have the form
$w^{(1)}(z)=z^p\Phi(z)$.\footnote{This condition also appears naturally when one considers the meromorphic equivalent systems to $\begin{psmallmatrix}
	p & 0\\
0 & p+1
	\end{psmallmatrix}z^{-1}$. Indeed, if the equivalence is realized by $I_2z^{-m}+\Phi_1z^{-m+1}+\Phi_2z^{-m+2}+\cdots$ we get that the equivalent, free monodromy system, have the form $\begin{psmallmatrix}
	p-m & 0\\
	0 & p+1-m
	\end{psmallmatrix}z^{-1}+ \begin{psmallmatrix}
	a_1 & (p+2) b_1\\
	0 & d_1
	\end{psmallmatrix}+O(z)$, where $\Phi_1=\begin{psmallmatrix}
	a_1 & b_1\\c_1 & d_1
	\end{psmallmatrix}$.}
Then, from \eqref{eq:log_ejem} we get
\begin{align*}
\Phi_1 &=W^{(1)}_0,\\
\Phi_2 &=\frac{1}{2}\big(W^{(1)}_0\Phi_1+W^{(1)}_1\big),\\
\Phi_3 &= \frac{1}{3}\big(W^{(1)}_0\Phi_2+W^{(1)}_1\Phi_1+W^{(1)}_2\big),\\
\Phi_4 &= \frac{1}{4}\big(W^{(1)}_0\Phi_3+W^{(1)}_1\Phi_2+W^{(1)}_2\Phi_1\big),\\
&\hspace{6pt}\vdots\\
\Phi_{n+3} &= \frac{1}{n+3}\big(W^{(1)}_0\Phi_{n+2}+W^{(1)}_1\Phi_{n+1}+W^{(1)}_2\Phi_n\big), & n&\in\{1,2,\dots\}
\end{align*}
We see that all coefficients are obtained  recursively,  and the  first three coefficients are
\begin{align*}
\Phi_1&=W^{(1)}_0,\\
\Phi_2 &=\frac{1}{2}\big((W^{(1)}_0)^2+W^{(1)}_1\big),\\
\Phi_3 &=\frac{1}{6}\big((W^{(1)}_0)^3+W^{(1)}_0W^{(1)}_1+2 W^{(1)}_1W^{(1)}_0+2W^{(1)}_2\big).
\end{align*}
Finally, for the matrix of weights we get
\begin{align*}
w(z)&=\big(\mathcal S(z)\big)^{-1} w^{(1)}(z)\\
&=\PARENS{\begin{matrix}
	z^p & 0\\0 & z^{p+1}
	\end{matrix}}\big(I_2+\Phi_1 z+\Phi_2 z^2+\cdots\big)
\end{align*}
\subsection{$N=2$ and $k=2$}
Let us assume that
\begin{align*}
W_{-1}=\PARENS{
	\begin{matrix}	p & 0\\
	0 & p+2
	\end{matrix}
}
\end{align*}
with $p\in\mathbb Z$,  so that
\begin{align*}
W(z)=\PARENS{\begin{matrix}
	p & 0\\
	0 & p+2
	\end{matrix}}z^{-1} +\PARENS{\begin{matrix}
	a_0 +a_1 z & b_0 + b_1 z\\
	c_0+c_1 z & d_0+d_1 z
	\end{matrix}}
\end{align*}
for $(a_0,b_0,c_0,d_0,a_1,b_1,c_1,d_1)^\top\in\C^8$.
As for $k=1$, we perform a shearing transformation generated by  $\mathcal S(z)$, see \eqref{eq:shearing},
and obtain
\begin{align*}
W^{(1)}&=\PARENS{\begin{matrix}
	p & 0\\
	c_0 & p+1
	\end{matrix}}z^{-1}+ W^{(1)}_0 + W^{(1)}_1 z +W^{(1)}_2 z^2, 
\end{align*}
where $W^{(1)}_0,W^{(1)}_1,W^{(1)}_2$ are given in  \eqref{eq:As}. To continue with the simplification, we now perform the following  diagonalization
\begin{align}\label{eq:dia}
\PARENS{
	\begin{matrix}
	p & 0\\
 0&p+1
	\end{matrix}}
&=\mathcal T
\PARENS{
	\begin{matrix}
	p & 0\\
	c_0 &p+1
	\end{matrix}}\mathcal T^{-1}, & \mathcal T&:=	\PARENS{\begin{matrix}
	1 & 0\\
	c_0 &1
	\end{matrix}}.
\end{align}
Consequently, the new matrix is
\begin{align*}
\tilde W^{(1)}&=\mathcal TW^{(1)}\mathcal T^{-1}\\
&=\PARENS{
	\begin{matrix}
	p & 0\\
	0&p+1
	\end{matrix}}z^{-1}+ \tilde W^{(1)}_0 + \tilde W^{(1)}_1 z +\tilde W^{(1)}_2 z^2, 
\end{align*}
where
\begin{align*}
\tilde W^{(1)}_0 &:= \mathcal T W^{(1)}_0\mathcal T^{-1} &  \tilde W^{(1)}_1 &:= \mathcal TW^{(1)}_1\mathcal T^{-1}&  \tilde W^{(1)}_2&:=\mathcal TW^{(1)}_2 \mathcal T^{-1}\\
&=\PARENS{
	\begin{matrix}
	a_0 & 0\\
	c_0(a_0-d_0)+c_1 & d_0
	\end{matrix}} ,  & &=\PARENS{
	\begin{matrix}
	a_1-b_0c_0 & b_0\\
	c_0(a_1-d_1)-b_0(c_0)^2 & b_0c_0+d_1
	\end{matrix}},& &=\PARENS{
	\begin{matrix}
	-b_1c_0 & b_1\\
	-b_1(c_0)^2 & b_1c_0
	\end{matrix}}.  
\end{align*}
After a second shearing transformation generated by $S(z)$ as in \eqref{eq:shearing} we obtain
\begin{align*}
W^{(2)}(z)&=\frac{\d \mathcal S(z)}{\d z}\big(\mathcal S(z)\big)^{-1} +\mathcal S(z) \tilde W^{(1)}(z)\big(\mathcal S(z)\big)^{-1}\\
&=\PARENS{\begin{matrix}
	p& 0\\
	c_0(a_0-d_0)+c_1  &p
	\end{matrix}}z^{-1}  +W_0^{(2)} +  W_1^{(2)}  z + W_2^{(2)}  z^2+ W_3^{(2)}z^3,
\end{align*}
with
\begin{align}\label{eq:coeff_K2}
\begin{aligned}
W^{(2)}_0&:=\PARENS{\begin{matrix} a_0 & 0\\
	c_0(a_1-d_1)-b_0(c_0)^2  & d_0
	\end{matrix}},&  W^{(2)}_1&:=\PARENS{\begin{matrix} 	a_1-b_0c_0  & 0\\
	-b_1(c_0)^2& b_0c_0+d_1
	\end{matrix}},\\
W^{(2)}_2&:=\PARENS{\begin{matrix} 	-b_1c_0  & b_0\\
	0 & b_1c_0  
	\end{matrix}},&  W^{(2)}_3&:=\PARENS{\begin{matrix} 0& b_1\\
	0& 0
	\end{matrix}}.
\end{aligned}
\end{align}
Thus, the matrix
\begin{align*}
w^{(2)}(z)=\mathcal S(z)\mathcal T\mathcal S(z)w(z)
\end{align*}
satisfies an ODE of the form
\begin{align*}
\frac{\d w^{(2)}(z)}{\d z}=W^{(2)}(z) w^{(2)}(z),
\end{align*}
which is  resonance free. Thus, a fundamental solution is 
\begin{align*}
w^{(2)}(z)&=(I_2+\Phi_1z+\Phi_2z^2+\cdots)z^{\PARENS{\begin{smallmatrix}
		p& 0\\
		\vartheta_2 &p
		\end{smallmatrix}}}, & \vartheta_2:=c_0(a_0-d_0)+c_1 .
\end{align*}
Furthermore, to avoid a multivalued situation we impose to  the coefficients of $W(z)$ to belong to the quadric in $\C^8$ determined by the equation
\begin{align*}
\vartheta_2 =	c_0(a_0-d_0)+c_1 =0,
\end{align*}
and the solution will be 
\begin{align*}
w^{(2)}(z)=z^p (I_2+\Phi_1z+\Phi_2z^2+\cdots).
	\end{align*}
The corresponding  matrix of weights is
\begin{align*}
w(z)&=\big(\mathcal S(z)\big)^{-1} \mathcal T^{-1}\big(\mathcal S(z)\big)^{-1} w^{(2)}(z)\\
&= \PARENS{\begin{matrix}
	z^p & 0\\
-c_0 z^{p+1} & z^{p+2}
	\end{matrix}} (I_2+\Phi_1z+\Phi_2z^2+\cdots),
\end{align*}
where the series converge at the annulus $\C\setminus\{0\}$ and
\begin{align*}
\Phi_1 &=W_0^{(2)},\\
\Phi_2 &=\frac{1}{2}\big(W_0^{(2)}\Phi_1+W_1^{(2)}\big),\\
\Phi_3 &= \frac{1}{3}\big(W_0^{(2)}\Phi_2+W_1^{(2)}\Phi_1+W_2^{(2)}\big),\\
\Phi_4 &= \frac{1}{4}\big(W_0^{(2)}\Phi_3+W_1^{(2)}\Phi_2+W_2^{(2)}\Phi_1+W_3^{(2)}\big),\\
\Phi_5 &= \frac{1}{5}\big(W_0^{(2)}\Phi_4+W_1^{(2)}\Phi_3+W_2^{(2)}\Phi_2+W_3^{(2)}\Phi_1\big),\\
&\hspace{6pt}\vdots\\
\Phi_{n+4} &= \frac{1}{n+4}\big(W_0^{(2)}\Phi_{n+3}+W_1^{(2)}\Phi_{n+2}+W_2^{(2)}\Phi_{n+1}+W_3^{(2)}\Phi_{n}\big), & n&\in\{1,2,\dots\}.
\end{align*}

\subsection{$N=2$ and $k=3$}
Finally, let us assume that
\begin{align*}
W_{-1}=\PARENS{
	\begin{matrix}	p & 0\\
	0 & p+3
	\end{matrix}
}
\end{align*}
with $p\in\mathbb Z$,  and
\begin{align*}
W(z)=\PARENS{\begin{matrix}
	p & 0\\
	0 & p+3
	\end{matrix}}z^{-1} +\PARENS{\begin{matrix}
	a_0 +a_1 z & b_0 + b_1 z\\
	c_0+c_1 z & d_0+d_1 z
	\end{matrix}}
\end{align*}
for $(a_0,b_0,c_0,d_0,a_1,b_1,c_1,d_1)^\top\in\C^8$.
As for $k=2$, we perform a  transformation generated by  $S(z)TS(z)$, see  \eqref{eq:shearing} and \eqref{eq:dia},
and get
\begin{align*}
W^{(2)}(z)&=\PARENS{\begin{matrix}
	p& 0\\
\vartheta_2  &p+1
	\end{matrix}}z^{-1}  +W_0^{(2)} +  W_1^{(2)}  z + W_2^{(2)}  z^2+ W_3^{(2)}z^3
\end{align*}
with coefficients given in \eqref{eq:coeff_K2}. A further diagonalization 
\begin{align*}
\PARENS{
	\begin{matrix}
	p & 0\\
	0&p+1
	\end{matrix}}
&=\mathcal T^{(2)}
\PARENS{
	\begin{matrix}
	p & 0\\
	 \vartheta_2 &p+1
	\end{matrix}}\big(\mathcal T^{(2)}\big)^{-1}, &\mathcal T^{(2)}&:=	\PARENS{\begin{matrix}
	1 & 0\\
	\vartheta_2&1
	\end{matrix}},
\end{align*}
gives
\begin{align*}
\tilde W^{(2)}&=T^{(2)} W^{(2)}\big(T^{(2)}\big)^{-1}\\
&=\PARENS{
	\begin{matrix}
	p & 0\\
	0&p+1
	\end{matrix}}z^{-1}+ \tilde W_0^{(2)} +\tilde  W_1^{(2)}  z +\tilde W_2^{(2)}  z^2+ \tilde W_3^{(2)}z^3,
\end{align*}
where
\begin{align*}
\tilde W_0 ^{(2)}&:= \mathcal T ^{(2)}A_0^{(2)}\big(\mathcal T^{(2)}\big)^{-1}=\PARENS{\begin{matrix} a_0 & 0\\
	c_0(a_1-d_1)-b_0(c_0)^2  +\vartheta_2(a_0-d_0)& d_0
	\end{matrix}},\\   \tilde W_1^{(2)} &:= \mathcal T^{(2)}W_1^{(2)}\big(\mathcal T^{(2)}\big)^{-1}=\PARENS{\begin{matrix} 	a_1-b_0c_0  & 0\\
	-b_1(c_0)^2+\vartheta_2(a_1-d_1-2b_0c_0)& b_0c_0+d_1
	\end{matrix}}, \\  \tilde W_2^{(2)}&:=\mathcal T^{(2)}W_2^{(2)} \big(\mathcal T^{(2)}\big)^{-1}=\PARENS{\begin{matrix} 	-b_1c_0  -\vartheta_2b_0& b_0\\
	-2\vartheta_2 b_1c_0 -(\vartheta_2)^2b_0& b_1c_0  +\vartheta_2b_0
	\end{matrix}},\\\tilde W_3^{(2)}&:=\mathcal T^{(2)}A_3^{(2)} \big(\mathcal T^{(2)}\big)^{-1}=\PARENS{\begin{matrix} -\vartheta_2 b_1& b_1\\
	-(\vartheta_2)^2b_1& \vartheta_2b_1
	\end{matrix}}.
\end{align*}
Now, as a final step, we perform a third shearing transformation generated by $S(z)$, see\eqref{eq:shearing}, to get
\begin{align*}
W^{(3)}(z)&=\frac{\d\mathcal  S(z)}{\d z}\big(\mathcal S(z)\big)^{-1} +\mathcal S(z) \tilde W^{(2)}(z)\big(\mathcal S(z)\big)^{-1}\\
&=\PARENS{\begin{matrix}
	p& 0\\
	\vartheta_3&p
	\end{matrix}}z^{-1}  +W_0^{(3)} +  W_1^{(3)}  z + W_2^{(3)}  z^2+ W_3^{(3)}z^3+ W_4^{(3)}z^4,
\end{align*}
with
\begin{align}\label{eq:coeff_K3}
\begin{aligned}
\vartheta_3&:=	c_0(a_1-d_1)-b_0(c_0)^2  +\vartheta_2(a_0-d_0)\\
&=c_0\Big((a_1-d_1)+(a_0-d_0)^2\Big)-b_0(c_0)^2 +c_1(a_0-d_0)\\
W^{(3)}_0&:=\PARENS{\begin{matrix} a_0 & 0\\
		-b_1(c_0)^2+\vartheta_2(a_1-d_1-2b_0c_0)& d_0
	\end{matrix}},\\
 W^{(3)}_1&=\PARENS{\begin{matrix} 	a_1-b_0c_0  & 0\\
	-2\vartheta_2 b_1c_0 -(\vartheta_2)^2b_0& b_0c_0+d_1
	\end{matrix}},\\
W^{(3)}_2&:=\PARENS{\begin{matrix}- b_1c_0  -\vartheta_2b_0& 0\\
	-(\vartheta_2)^2b_1 & b_1c_0  +\vartheta_2b_0
	\end{matrix}},\\
W^{(3)}_3&:=\PARENS{\begin{matrix}  -\vartheta_2 b_1& b_0\\
	0&  \vartheta_2 b_1
	\end{matrix}},\\
W^{(3)}_3&:=\PARENS{\begin{matrix} 0& b_1\\
	0& 0
	\end{matrix}}.
\end{aligned}
\end{align}

Hence, the matrix
$w^{(3)}(z)=\mathcal S(z)\mathcal T^{(2)}\mathcal S(z)\mathcal T\mathcal S(z)w(z)$ 
satisfies the following non-resonant linear system of ODE 
\begin{align*}
\frac{\d w^{(3)}(z)}{\d z}=W^{(3)}(z) w^{(3)}(z)
\end{align*}
with a fundamental solution given by 
\begin{align*}
w^{(2)}(z)=(I_2+\Phi_1z+\Phi_2z^2+\cdots)z^{\PARENS{\begin{smallmatrix}
		p& 0\\
		\vartheta_3&p
		\end{smallmatrix}}}.
\end{align*}
Furthermore, to avoid a multivalued situation we impose to the coefficients of $W(z)$ to lay in the cubic hypersurface
\begin{align*}
\vartheta_3=c_0\Big((a_1-d_1)+(a_0-d_0)^2\Big)-b_0(c_0)^2 +c_1(a_0-d_0)=0
\end{align*}
and the solution is simplifies to $w^{(2)}(z)=z^p (I_2+\Phi_1z+\Phi_2z^2+\cdots)$.
Hence, the corresponding  matrix of weights is
\begin{align*}
w(z)&=\big(\mathcal S(z)\big)^{-1} \mathcal T^{-1}\big(\mathcal S(z)\big)^{-1} \big(\mathcal T^{(2)}\big)^{-1}\big(\mathcal S(z)\big)^{-1} w^{(3)}(z)\\
&= \PARENS{\begin{matrix}
	z^p & 0\\
	-c_0 z^{p+1}-\vartheta_2 z^{p+2} & z^{p+3}
	\end{matrix}} (I_2+\Phi_1z+\Phi_2z^2+\cdots),
\end{align*}
where
\begin{align*}
\Phi_1 &=W_0^{(3)},\\
\Phi_2 &=\frac{1}{2}\big(W_0^{(3)}\Phi_1+W_1^{(3)}\big),\\
\Phi_3 &= \frac{1}{3}\big(W_0^{(3)}\Phi_2+W_1^{(3)}\Phi_1+W_2^{(3)}\big),\\
\Phi_4 &= \frac{1}{4}\big(W_0^{(3)}\Phi_3+W_1^{(3)}\Phi_2+W_2^{(3)}\Phi_1+W_3^{(3)}\big),\\
\Phi_5 &= \frac{1}{5}\big(W_0^{(3)}\Phi_4+W_1^{(3)}\Phi_3+W_2^{(3)}\Phi_2+W_3^{(3)}\Phi_1+W_4^{(3)}\big),\\
\Phi_6 &= \frac{1}{6}\big(W_0^{(3)}\Phi_5+W_1^{(3)}\Phi_4+W_2^{(3)}\Phi_4+W_3^{(3)}\Phi_4+W_4^{(3)}\Phi_1\big),\\
&\hspace{6pt}\vdots\\
\Phi_{n+5} &= \frac{1}{n+5}\big(W_0^{(3)}\Phi_{n+4}+W_1^{(3)}\Phi_{n+3}+W_2^{(3)}\Phi_{n+2}+W_3^{(3)}\Phi_{n+1}+W_4^{(3)}\Phi_n\big), & n&\in\{1,2,\dots\}.
\end{align*}
Notice that the series converge at the annulus $\C\setminus\{0\}$.


\begin{thebibliography}{99}
	
\bibitem{AC}  M. J. Ablowitz and P. A. Clarkson, \emph{Solitons, Nonlinear Evolution Equations and Inverse Scattering,} London Mathematical  Society Lecture Notes Series \textbf{149}, Cambridge University Press, Cambridge, 1991.
	
\bibitem{AHH} M. J. Ablowitz, R. Halburd, and B. Herbst, \textit{On the extension of the Painlevé property to difference equations}, Nonlinearity
	\textbf{13} (2000) 889-905.
	
\bibitem {a-l-1} M. J. Ablowitz, and J. F. Ladik, \emph{Nonlinear differential-difference equations}, Journal of Mathematical Physics \textbf{16} (1975) 598-603.
	
\bibitem {a-l-2} M. J. Ablowitz and J. F. Ladik, \emph{Nonlinear differential-difference equations and
		Fourier analysis}, Journal of Mathematical Physics \textbf{17} (1976) 1011-1018.

  \bibitem{adler-van-moerbeke} M. Adler and P. van Moerbeke, \emph{Generalized orthogonal polynomials, discrete KP and
  	Riemann--Hilbert  problems}, Communications in Mathematical Physics \textbf{207} (1999) 589-620.

  \bibitem{adler-vanmoerbeke-2} M. Adler and P. van Moerbeke, \emph{Darboux transforms on band matrices, weights and
  	associated polynomials},  International Mathematical  Research Notices \textbf{18} (2001) 935-984.

\bibitem{Adler-Van-Moerbecke-Toeplitz} M. Adler and P. van Moerbeke, \emph{Integrals over classical groups, random permutations, Toda and Toeplitz lattices},  Communications in Pure and Applied Mathematics \textbf{54} (2001)  153-205.

\bibitem{Alfaro} M. Alfaro, \emph{Una expresión de los polinomios ortogonales sobre la circunferencia unidad}, Actas III
  J.M.H.L. (Sevilla, 1974) 2 (1982) 1-8.

\bibitem{cum} C. Álvarez-Fernández, U. Fidalgo, and M. Mañas,
  \emph{The multicomponent 2D Toda hierarchy: generalized matrix orthogonal polynomials, multiple
  	orthogonal polynomials and Riemann--Hilbert problems}, Inverse Problems \textbf{26} (2010) 055009 (17 pp.)


\bibitem{carlos} C. Álvarez-Fernández and M. Mañas, \emph{Orthogonal Laurent polynomials on the unit circle, extended CMV ordering and 2D Toda type integrable hierarchies},  Advances in Mathematics \textbf{240}  (2013) 132-193.

\bibitem{carlos2} C.  Álvarez-Fernández and M. Mañas, \emph{On the Christoffel--Darboux formula for generalized matrix orthogonal polynomials of multigraded Hankel  type}, Journal of Mathematical Analysis and Applications \textbf{418} (2014) 238-247.

\bibitem{nuevo} Carlos Álvarez-Fernández, Gerardo Ariznabarreta, Juan Carlos García-Ardila, Manuel Mañas, and Francisco Marcellán, \emph{Christoffel transformations for matrix orthogonal polynomials in the real line and the non-Abelian 2D Toda lattice hierarchy}, 
International Mathematics Research Notices. (2016) doi: 10.1093/imrn/rnw027.

\bibitem{nuevo2} Carlos Álvarez-Fernández, Gerardo Ariznabarreta, Juan Carlos García-Ardila, Manuel Mañas, and Francisco Marcellán, \emph{Transformation theory and Christoffel formulas for matrix biorthogonal  polynomials on the real line }, arXiv:1605.04617.

\bibitem{Ambrolazde} M. Ambroladze, \emph{On exceptional sets of asymptotics relations for general orthogonal polynomials},
    Journal of Approximation  Theory \textbf{82} (1995) 257-273.
    
    \bibitem{anosov2} D. V. Anosov, S. Kh. Aranson, V. I. Arnold, I. U. Bronshtein, V. Z. Grines, and Yu. S Il'ashenko, \emph{Ordinary Differential Equations and Smooth Dynamical Systems}, Springer, third edition, Berlin, 1997.

\bibitem{anosov} D. V. Anosov and A. A. Bolibruch, \emph{The Riemann-Hilbert Problem}, Vieweg Verlag, Wiesbaden, 1994

  \bibitem{nikishin} A. I. Aptekarev and E. M. Nikishin,\emph{The scattering   problem for a discrete Sturm--Liouville operator},   Mathematics of the USSR Sbornik \textbf{49} (1984) 325-355.

\bibitem{AM} G. Ariznabarreta and M. Mañas, \emph{Matrix orthogonal Laurent polynomials on the unit circle and Toda type integrable systems}, Advances in Mathematics \textbf{264} (2014) 396-464.

\bibitem{baik0} J. Baik, \emph{Riemann--Hilbert problems for last passage percolation} in \emph{Recent Developments in Integrable Systems and Riemann--Hilbert Problems}, K. McLaughlin and  X. Zhou, eds, Contemporary Mathematics \textbf{326} (2003) 1–21.

\bibitem{baik} J. Baik, P. Deift, and K. Johansson, \emph{On the distribution of the length of the longest increasing subsequence of random permutations}, Journal of the  American Mathematical Society \textbf{12} (1999) 1119–1178.

\bibitem{Barrios-Lopez} D. Barrios, G. López, \emph{Ratio asymptotics for orthogonal polynomials on arcs of the unit circle},
Constructive Approximation \textbf{15} (1999) 1-31.

\bibitem{bere} Yu. M. Berezanskii, \emph{Expansions in eigenfunctions of
	self-adjoint operators}, Translations of  Mathematical Monographs \textbf{17},  American Mathematical Society, 1968.

\bibitem{Berriochoa} E. Berriochoa, A. Cachafeiro, and J. Garcia-Amor \emph{Connection between orthogonal polynomials on the unit circle
	and bounded interval}, Journal of Computational and Applied Mathematics \textbf{177} (2005) 205-223.

\bibitem{Bertola} M. Bertola and M. Gekhtman, \emph{Biorthogonal Laurent polynomials, Toeplitz determinants, minimal Toda
	orbits and isomonodromic tau functions}, Constructive Approximation \textbf{26} (2007) 383-430.

\bibitem{birkhoff} G. D. Birkhoff, \emph{Singular points of ordinary linear differential equations},  Transactions of the American Mathematical Society \textbf{10} (1909) 391-435.

\bibitem{borrego} J. Borrego, M. Castro, and A. J. Durán, \emph{Orthogonal
	matrix polynomials satisfying differential equations with recurrence
	coefficients having non-scalar limits}, Integral Transforms and Special Functions \textbf{23} (2012) 685-700.

\bibitem{Bul} A. Bultheel, P. González-Vera,   E. Hendriksen,  and O.Nj\r{a}stad, \emph{Orthogonal rational functions}, Cambridge Monographs on
Applied and Computational Mathematics \textbf{5}, Cambridge University Press, Cambridge, 1999.

\bibitem{Cafasso} M. Cafasso, \emph{Matrix biorthogonal polynomials on the unit circle and non-Abelian Ablowitz-Ladik hierarchy}, Journal of Physics A: Mathematical \& Theoretical \textbf{42} (2009) 365211.

 \bibitem{CMV} M. J. Cantero, L. Moral, and L. Velázquez, \emph{Five-diagonal matrices and zeros of orthogonal polynomials on the unit circle},  Linear Algebra and Applications \textbf{362} (2003) 29-56.

\bibitem{cantero} M. J. Cantero, L. Moral,   and L. Velázquez, \emph{Differential properties of matrix orthogonal polynomials}, Journal of Concrete and Applicable Mathematics \textbf{3} (2008) 313-334.

\bibitem{CM} G. Cassatella-Contra and M. Mañas, \emph{Riemann--Hilbert problems, matrix orthogonal polynomials and discrete
	matrix equations with singularity confinement}, Studies in  Applied Mathematics \textbf{128} (2012)   252-274.

\bibitem{CMT} G. Cassatella-Contra, M. Mañas, and P. Tempesta, \emph{Singularity confinement for matrix discrete Painlevé equations},
Nonlinearity \textbf{27} (2014) 2321-2335.

\bibitem{Cima} J. A. Cima, A. L. Matheson, and W. T. Ross, \emph{The Cauchy Transform}, Mathematical Surveys and Monographs \textbf{125}, American Mathematical Society, 2006.

\bibitem{Clancey} K. Clancey and I. Gohberg, \emph{Factorization of Matrix Functions and Singular Integral Operators},  Operator Theory: Advances and Applications \textbf{3}, Springer Basel AG, Basel, 1981.

\bibitem{clarkson} P. A. Clarkson and K. Jordaan, \emph{The Relationship Between Semiclassical Laguerre Polynomials and the Fourth Painlevé Equation}, Constructive Approximation \textbf{29} (2014) 223-254. 

\bibitem{cresswell} C. Creswell and N. Joshi, \emph{The discrete first, second and thirty-fourth Painlevé  hierarchies}, Journal of Physics A: Mathematical \& General \textbf{32} (1999) 655-669.

\bibitem{Barroso-Snake} R. Cruz-Barroso and S. Delvaux, \emph{Orthogonal Laurent polynomials on the unit circle and snake-shaped matrix factorizations}, Journal of Approximation Theory \textbf{161} (2009) 65-87.

\bibitem{Barroso-Vera} R. Cruz-Barroso and P. González-Vera, \emph{A Christoffel--Darboux formula and a Favard's theorem for Laurent orthogonal polynomials on the unit circle}, Journal of Computational and  Applied Mathematics \textbf{179} (2005) 157-173.

\bibitem{Damanik} D. Damanik, A. Pushnitski, and B. Simon,\emph{ The Analytic Theory of Matrix Orthogonal Polynomials}, Surveys in Approximation Theory \textbf{4} (2008) 1-85.

\bibitem{daems2}  E. Daems and A. B. J. Kuijlaars, \emph{Multiple orthogonal polynomials of mixed type and non-intersecting Brownian motions}, Journal of Approximation Theory \textbf{146} (2007)  91–114.

\bibitem{daems3} E. Daems, A. B. J. Kuijlaars,  and W. Veys, \emph{Asymptotics of non-intersecting Brownian motions and a $4\times 4$ Riemann–Hilbert problem}, Journal of Approximation Theory \textbf{153} (2008) 225–256.

	\bibitem{dai} D. Dai and   A. B. J. Kuijlaars, \emph{Painlevé IV asymptotics for orthogonal polynomials with respect to a modified Laguerre weight}, Studies in Applied Mathematics \textbf{122} (2009) 29–83.

\bibitem{deift1} P. A. Deift, \emph{Orthogonal Polynomials and Random Matrices: A Riemann--Hilbert Approach},  Courant Lecture Notes \textbf{3}, American Mathematical Society, Providence, RI, 2000.

\bibitem{deift2} P. A. Deift, \emph{Riemann–Hilbert methods in the theory of orthogonal polynomials, in Spectral Theory and Mathematical Physics: a Festschrift in Honor of Barry Simon’s 60th Birthday}, Proceedings of Symposia in Pure Mathematics  \textbf{76}, 715–740, American  Mathematical  Society, Providence, RI, 2007.

\bibitem{deift5}  P. A. Deift and D. Gioev, \emph{Random matrix theory: invariant ensembles and universality}, Courant Lecture Notes in Mathematics\textbf{18}, American Mathematical Society,  Providence, RI,  2009.

\bibitem{deift3}   P. A.  Deift and X. Zhou, \emph{A steepest descent method for oscillatory Riemann–Hilbert problems. Asymptotics for the MKdV equation}, Annals of Mathematics  \textbf{137} (1993) 295–368.

\bibitem{deift4} P. A.  Deift and X. Zhou, \emph{Long-time asymptotics for solutions of the NLS equation with initial data in a weighted Sobolev space}, Communications in  Pure Applied Mathematics \textbf{56} (2003) 1029–1077.

\bibitem{duran20051} A. J. Durán and F. J. Grünbaum, \emph{Orthogonal matrix
	polynomials, scalar-type Rodrigues' formulas and Pearson equations}, Journal
of Approximation Theory \textbf{134}  (2005) 267-280.

\bibitem{duran20052} A. J. Durán and F. J. Grünbaum, \emph{Structural  formulas for orthogonal matrix polynomials satisfying second order  differential equations, I}, Constructive Approximation \textbf{22} (2005) 255-271.

\bibitem{duran1997} A. J. Durán,\emph{ Matrix inner product having a matrix  symmetric second order differential operator}, Rocky Mountain Journal of
Mathematics \textbf{27}  (1997) 585-600.

\bibitem{duran2004}A. J. Durán and F. J. Grünbaum,,  \emph{Orthogonal matrix  polynomials    satisfying     second    order    differential    equations},
International Mathematics Research Notices  \textbf{10} (2004) 461-484.

\bibitem{duran2008} A. J. Durán and M. D. de la Iglesia, \emph{Second order
	differential operators having several families of orthogonal matrix
	polynomials as eigenfunctions},   International Mathematics Research Notices    \textbf{2008} (2008).

\bibitem{Evans} L.C. Evans and R. F. Gariepy, \emph{Measure Theory and Fine Properties of Functions }, revised edition, CRC Press,  2015.

\bibitem{Fay1} L. Faybusovich and  M. Gekhtman, \emph{On Schur flows}, Journal of Physics A: Mathematical and Geneneral \textbf{32} (1999) 4671-4680.

\bibitem{Fay2} L. Faybusovich and M. Gekhtman \emph{Elementary Toda orbits and integrable lattices}, Journal of Mathematical Physics \textbf{41} (2000) 2905-2921.

\bibitem{Fay3} L. Faybusovich and M. Gekhtman, \emph{Inverse moment problem for elementary co-adjoint orbits}, Inverse Problems \textbf{17} (2001), 1295-1306.

\bibitem{fokas} A. S. Fokas, A. R. Its, A. A. Kapaev, and V. Yu. Novokshenov, \emph{Painlevé Transcendents. The Riemann--Hilbert Approach}, Mathematical Surveys and Monographs \textbf{128}, American Mathematical Society, Providence, RI, 2006.

\bibitem{FIK} A. S. Fokas, A. R. Its, and A. V. Kitaev, \emph{The isomonodromy approach to matrix models in 2D
	quantum gravity}, Communications in Mathematical Physics \textbf{147} (1992) 395-430.

\bibitem{freud0}  G. Freud, \emph{On the coefficients in the recursion formulae of orthogonal polynomials,} Proceedings of the Royal Irish Academy Section A \textbf{76} (1976) 1–6.

\bibitem{Freud} G. Freud, \emph{Orthogonal Polynomials}, Akadémiai Kiadó, Budapest and Pergamon Press, Oxford, 1971, 1985.

\bibitem{Garcia} P. García and F. Marcellán, \emph{On zeros of regular orthogonal polynomials on the unit circle}, Annales Polonici Mathematici \textbf{58} (1993) 287-298.

\bibitem{gautschi} W. Gautschi, \emph{Orthogonal Polynomials:computation and approximation}, Oxford University Press, New York, 2004.

\bibitem{geronimus-2} Ya. L. Geronimus, \emph{Polynomials orthogonal on a circle and their applications, Series and Approximations},
American  Mathematical Society  Translations, series 1, vol. 3, Providence, RI, 1962,
1-78.

\bibitem{Gakhov} F. D. Gakhov, \emph{Boundary Value Problems}, Dover Publications, Inc. New York, 1990.

\bibitem{Gelfand2005}	 I. M. Gelfand, S. Gelfand, V. S. Retakh, and R. Wilson, \emph{Quasideterminants}, Advances in Mathematics \textbf{193} (2005) 56-141.

\bibitem{Gelfand1995Noncommutative}I. M. Gelfand, D. Krob, A. Lascoux, B. Leclerc, V. S. Retakh, and J.-Y.	Thibon, \emph{Noncommutative symmetric functions}, Advances in Mathematics \textbf{112} (1995) 218-348.
		
\bibitem{Gelfand1991Determinants}I. M. Gelfand and V. S. Retakh, \emph{Determinants of matrices over noncommutative rings}, Functional Analysis and its Applications \textbf{25} (1991) 91-102.

\bibitem{geronimo} {J. S. Geronimo}, \emph{Scattering theory and matrix  orthogonal polynomials on the real line}, Circuits Systems Signal Process  \textbf{1} (1982) 471-495.

\bibitem{Godoy} E. Godoy and F. Marcellán, \emph{Orthogonal polynomials on the unit circle: distribution of zeros}, Journal of Computational and
Applied Mathematics \textbf{37} (1991) 195-208.

\bibitem{Golinskii} L. Golinskii, \emph{Schur flows and orthogonal polynomials on the unit circle}, Sbornik Mathematics \textbf{197} (2006)  1145.

\bibitem{Golinskii2} L. Golinskii and A. Zlatos, \emph{Coefficients of orthogonal polynomials on the unit circle and higher order Szegő theorems}, Constructive Approximation \textbf{26}  (2007) 361-382.

\bibitem{golinski}  L.  Golinskii and P. Nevai, \emph{Szegő Difference Equations, Transfer Matrices and Orthogonal Polynomials on the Unit Circle}, Communications in Mathematical Physics \textbf{223}, (2001) 223-259.
		
		
\bibitem{dominguez} F.  Grünbaum,  M. D. de la Iglesia, and A. Martínez-Finkelshtein, \emph{Properties of matrix orthogonal polynomials via their Riemann--Hilbert characterization}, SIGMA \textbf{7} (2011), 098, 31 pages.		
		
\bibitem{hille} E. Hille, \emph{Ordinary Differential Equations in the Complex Domain}, John Wiley\& Sons, Inc., New York, 1976.

\bibitem{hisakado} M. Hisakado, \emph{Unitary matrix models and Painlevé III}, Modern Physics Letters \textbf{A11} (1996) 3001–3010.

\bibitem{ince} E. L. Ince, \emph{Ordinary Differential Equations}, Dover Publications, Inc., New York, 1956.

\bibitem{ismail} M. E. H. Ismail, \emph{Classical and Quantum Orthogonal Polynomials in One Variabl}, Encyclopedia of Mathematics and its Applications \textbf{98}, Cambridge University Press, 2005.

\bibitem{its} A. R. Its, \emph{The Riemann--Hilbert Problem and Integrable Systems}, Notices of the AMS \textbf{11} (2003) 1389-1400.

\bibitem{kuijlaars4}  A. R.  Its, A. B. J. Kuijlaars, and J.  Östensson, \emph{Asymptotics for a special solution of the thirty fourth Painlevé equation}, Nonlinearity \textbf{22} (2009) 1523–1558.

\bibitem{Jones-1} W. B. Jones and O. Njåstad,  \emph{Applications of Szegő polynomials to digital signal processing}, Rocky Mountain Journal of Mathematics \textbf{21} (1991)  387-436.

\bibitem{Jones-2} W. B. Jones, O. Njåstad, W. J. Thron, and H. Waadeland, \emph{Szegő polynomials applied to frequency analysis},
Computational and  Applied Mathematics \textbf{46} (1993) 217-228.

\bibitem{Killip} R. Killip and I. Nenciu, \emph{CMV: The unitary analogue of Jacobi matrices}, Communications on Pure and Applied Mathematics \textbf{60} (2006) 1148-1188.

\bibitem{krein1} M. G. Krein, \emph{Infinite J-matrices and a matrix moment  problem}, Doklady Akademii Nauk SSSR \textbf{69} (1949) 125-128.

\bibitem{krein2} M. G. Krein, \emph{Fundamental aspects of the representation
	theory of hermitian operators with deficiency index (m,m)}, American Mathematical Society Translations, Series 2, vol. 97, 75-143, Providence, RI,  1971.

\bibitem{kuijlaars2} A. B. J. Kuijlaars, A.  Martínez-Finkelshtein and F. Wielonsky, \emph{Non-intersecting squared Bessel paths and multiple orthogonal polynomials for modified Bessel weights}, Communications of Mathematical Physics \textbf{286} (2009) 217–275.

\bibitem{kuijlaars3} A. B. J. Kuijlaars, \emph{Multiple orthogonal polynomial ensembles}, in \emph{Recent Trends in Orthogonal Polynomials and Approximation Theory,} Contemporary Mathematics \textbf{507},  155–176,  American Mathematical Society,  Providence, RI,  2010.

\bibitem{magnus}  A. P. Magnus, \emph{Freud’s equations for orthogonal polynomials as discrete Painlevé equations}, in \emph{Symmetries and Integrability of Difference Equations} (Canterbury, 1996), London Mathematical Society Lecture Note Series \textbf{255}, 228–243, Cambridge University Press, Cambridge, 1999.

\bibitem{McLaughlin1} K. T.-R.  McLaughlin, A. H.  Vartanian, and  X.  Zhou,  \emph{Asymptotics of Laurent polynomials of even degree orthogonal with respect to varying exponential weights}, International Mathematical Research Notices \textbf{2006} (2006).
	
\bibitem{McLaughlin2} K. T.-R.  McLaughlin, A. H.  Vartanian, and  X.  Zhou,\emph{ Asymptotics of Laurent polynomials of odd degree orthogonal with respect to varying exponential weights}, Constructive Approximation \textbf{27} (2008) 149–202.

\bibitem{mf1}A. Martínez-Finkelshtein, \emph{Szegő polynomials: a view from the Riemann-Hilbert window},  Electronic Transactions in Numerical Analysis \textbf{25}  (2006) 369-392.

 \bibitem{mf2} A. Martínez-Finkelshtein, K. T.-R. McLaughlin, and E. B. Saff, \emph{Szegő  orthogonal polynomials with respect to an analytic weight: canonical representation and strong asymptotics}, Constructive Approximation \textbf{24} (2006) 319-363.

 \bibitem{Mhaskar} H.N. Mhaskar and E. B. Saff, \emph{On the distribution of zeros of polynomials orthogonal on the unit circle},
 Journal of  Approximation Theory \textbf{63} (1990) 30-38.

 \bibitem{miranian} L. Miranian, \emph{Matrix valued orthogonal polynomials on the real line: some extensions of the classical theory}, Journal of Physics A: Mathematical and General \textbf{38} (2005) 5731-5749.

 \bibitem{miranian2} L. Miranian, \emph{Matrix Valued Orthogonal Polynomials on the Unit Circle: Some Extensions of the Classical Theory}, Canadian Mathematical Bulletin \textbf{52} (2009) 95-104.

 \bibitem{dragan} D. Miličíc, \emph{Lectures on differential equations in complex domains}, http://www.math.utah.edu/~milicic/Eprints/de.pdf.
 
 \bibitem{Muskhelishvili} N. I. Muskhelishvili, \emph{Singular Integral Equations},  Dover, Mineola, NY, 2008.

 \bibitem{Mukaihira} A. Mukaihira and Y. Nakamura, \emph{Schur flow for orthogonal polynomials on the unit circle and
 	its integrable discretization}, Journal of Computational and Applied Mathematics \textbf{139} (2002) 75-94.

 \bibitem{Nenciu} I. Nenciu, \emph{Lax pairs for the Ablowitz-Ladik system via orthogonal polynomials on the unit circle},
 International Mathematics Research Notices \textbf{11}, (2005) 647-686.

 \bibitem{Totik} P. Nevai and  V. Totik, \emph{Orthogonal polynomials and their zeros}, Acta Scientiarum Mathematicarum Szeged
 \textbf{53} (1-2) (1989) 99-104.

 \bibitem{nijhoff} F. W. Nijhoff and V. G. Papaeorigiou, \emph{Similarity reductions of integrable lattices and discrete analogues of the Painlevé II equation}, Physics  Letters \textbf{A153} (1991) 337–344.

\bibitem{olver} P. J. Olver, \emph{On Multivariate Interpolation}, Studies in Applied Mathematics  \textbf{116} (2006) 201-240.

\bibitem{Pan-1} K. Pan, \emph{Asymptotics for Szegő polynomials associated with Wiener--Levinson filters}, Journal of Computational and Applied Mathematics \textbf{46} (1993) 387-394.

\bibitem{Pan-2} K. Pan and E.B. Saff, \emph{Asymptotics for zeros of Szegő polynomials associated with trigonometric polynomials
	signals}, Journal of Approximation Theory \textbf{71} (1992) 239-251.

\bibitem{Periwal0} V. Periwal and D. Shevitz, \emph{Unitary-Matrix Models as Exactly Solvable String Theories},  Physical Review Letters \textbf{64} (1990) 1326-1329.

\bibitem{Periwal} V. Periwal and D. Shevitz, \emph{Exactly solvable unitary matrix models: multicritical potentials and correlations}, Nuclear Physics \textbf{B344} (1990) 731-746.

\bibitem{Rudin} W. Rudin, \emph{Real and complex analysis}, third edition, McGraw-Hill Book, New York, 1987.

\bibitem{sibuya} Y. Sibuya, \emph{Linear Differential Equations in the Complex Domain; Problems of Analytic Continuation}, Translations of Mathematical Monographs \textbf{82}, American Mathematical Society, Providence, RI, 1990.

\bibitem{Simon-1} B. Simon, \emph{Orthogonal Polynomials on the Unit Circle, Part 1: Classical Theory},
AMS Colloquium Series, American Mathematical Society, Providence, RI, 2005.

\bibitem{Simon-2} B. Simon, \emph{Orthogonal Polynomials on the Unit Circle, Part 2: Spectral Theory},
AMS Colloquium Series, American Mathematical Society, Providence, RI, 2005.

\bibitem{CMV-Simon} B. Simon, \emph{CMV matrices: Five years after}, Journal of Computational and Applied Mathematics \textbf{208} (2007) 120-154.

\bibitem{Simon-Schur} B. Simon, \emph{Zeros of OPUC and long time asymptotics of Schur and related flows},
Inverse Problems Imaging \textbf{1} (2007), 189-215.

\bibitem{Simon-S} B. Simon \emph{Szegő's Theorem and its descendants}, Princeton University Press, Princeton, NJ, 2011.

\bibitem{szego} G. Szegő, \emph{Orthogonal Polynomials}, Colloquium Publications \textbf{33}, American Mathematical Society, Providence, RI, 1939.

 \bibitem{tracy} C. A. Tracy and  H. Widom, \emph{Random unitary matrices, permutations and Painlevé}, Communications in  Mathematical Physics \textbf{207} (1999) 665–685.

\bibitem{VAssche} W. Van Assche, \emph{Discrete Painlevé equations for recurrence coefficients of orthogonal polynomials},
 Proceedings of the International Conference on Difference Equations, Special Functions and
  Orthogonal Polynomials,  687-725, World Scientific (2007).
  
  \bibitem{wasow} W. Wasow, \emph{Asymptotic Expansions for Ordinary Differential Equations}, John Wiley \& Sons, Inc., New York, 1965.

  \bibitem{watkins} D. S. Watkins, \emph{Some perspectives on the eigenvalue problem}, SIAM Review \textbf{35} (1993) 430-471.

\end{thebibliography}
\end{document}